\documentclass[reqno]{amsart}
\usepackage[pagewise]{lineno}%\linenumbers
\usepackage{euscript,graphicx,epstopdf,amscd,amsgen,amsfonts,amssymb,latexsym,
amsmath,amsthm,graphicx,mathrsfs,times,color,overpic,bm}
\newtheorem{mtheo}{Theorem}
\newtheorem{mcoro}{Corollary}[mtheo]

\newtheorem{theo}{Theorem}[section]

\newtheorem{lemma}[theo]{Lemma}
\newtheorem{claim}[theo]{Claim}
\newtheorem*{claim*}{Claim}

\newtheorem{corollary}[theo]{Corollary}
\newtheorem{proposition}[theo]{Proposition}

\theoremstyle{definition}
\newtheorem{definition}[theo]{Definition}
\newtheorem{example}[theo]{Example}
\newtheorem{remark}[theo]{Remark}

\newtheorem{questions}[theo]{Questions}
\newtheorem*{remark*}{Remark}

\newcommand{\eqdef}{\stackrel{\scriptscriptstyle\rm def}{=}}

\DeclareMathOperator{\welim}{we-lim}

\DeclareMathOperator{\card}{card}
\DeclareMathOperator{\interior}{int}
\DeclareMathOperator{\Lip}{Lip}

\DeclareMathOperator{\Per}{Per}

\DeclareMathOperator{\freq}{freq}
\DeclareMathOperator{\conv}{conv}
\DeclareMathOperator{\cconv}{\overline{\conv}}

\def\blambda{\bm{\lambda}}

\def\ukappa{\kappa_1}
\def\okappa{\kappa_2}

\def\c{{\rm c}}
\def\s{{\rm s}}

\def\u{{\rm u}}

\def\bN{\mathbb{N}}

\def\bZ{\mathbb{Z}}
\def\bR{\mathbb{R}}

\def\cA{\EuScript{A}}

\def\cN{\EuScript{N}}
\def\cO{\EuScript{O}}
\def\cP{\EuScript{P}}

\def\cN{\EuScript{N}}

\def\cW{\mathscr{W}}

\def\cM{\EuScript{M}}

\def\sX{\mathcal{X}}
\def\eA{\mathscr{A}}

\def\eJ{\mathcal{J}}
\def\eW{\mathcal{W}}

\def\c{{\rm c}}
\def\s{{\rm s}}
\def\u{{\rm u}}

\def\cF{\mathscr{F}}

\numberwithin{equation}{section}

\DeclareMathSymbol{\varnothing}{\mathord}{AMSb}{"3F}
\renewcommand{\emptyset}{\varnothing}

\thanks{This research has been supported [in part] by (CAPES) - Finance Code 001, by CNPq-grants and CNPq Projeto Universal , by INCT-FAPERJ (Brazil), and by National Science Centre grants 2014/13/B/ST1/01033 and 2019/33/B/ST1/00275 (Poland). The authors acknowledge the hospitality of IMPAN, IM-UFRJ, and PUC-Rio. The authors also thank D. Kwietniak for helpful conversations.}

\begin{document}

\title[Mingled hyperbolicities]{Mingled hyperbolicities:\\ ergodic properties and bifurcation phenomena\\(an approach using concavity)}
%Skew-products with concave fiber maps]{Skew-products with concave fiber maps:\\
%entropy and weak$\ast$ approximation of measures\\and\\ bifurcation phenomena}
\author[L.~J.~D\'iaz]{L. J. D\'\i az}
\address{Departamento de Matem\'atica PUC-Rio, Marqu\^es de S\~ao Vicente 225, G\'avea, Rio de Janeiro 22451-900, Brazil}
\email{lodiaz@mat.puc-rio.br}
\author[K.~Gelfert]{K.~Gelfert}
\address{Instituto de Matem\'atica Universidade Federal do Rio de Janeiro, Av. Athos da Silveira Ramos 149, Cidade Universit\'aria - Ilha do Fund\~ao, Rio de Janeiro 21945-909,  Brazil}\email{gelfert@im.ufrj.br}
\author[M.~Rams]{M. Rams} \address{Institute of Mathematics, Polish Academy of Sciences, ul. \'{S}niadeckich 8,  00-656 Warszawa, Poland}
\email{rams@impan.pl}

\begin{abstract}
We consider skew-products with concave interval fiber maps over a certain subshift obtained as the projection of orbits staying in a given region. It generates a new type of (essentially) coded shift. The fiber maps have expanding and contracting regions which dynamically interact. The dynamics also exhibits pairs of horseshoes of different type of hyperbolicity which, in some cases, are cyclically related. 

The space of ergodic measures on the base is an entropy-dense Poulsen simplex. Those measures lift canonically to ergodic measures for the skew-product. We explain when and how the spaces of (fiber) contracting and expanding ergodic measures glue along the nonhyperbolic ones. A key step is the approximation (in the weak$\ast$ topology and in entropy) of nonhyperbolic measures by ergodic ones, obtained only by means of concavity. 
Concavity is not merely a technical artificial hypothesis, but it prevents the presence of additional independent subsystems.
The description of homoclinic relations is also a key instrument.   

These skew-products are embedded in non-decreasing entropy one-parameter family of diffeomorphisms  stretching from a heterodimensional cycle to a collision of homoclinic classes. Associated bifurcation phenomena involve a jump of the space of ergodic measures and, in some cases, of entropy.
\end{abstract}


\maketitle
\tableofcontents

\section{Introduction} 
%------------------------------------------------------------------------------------------------------

%------------------------------------------------------------------------------------------------------
\subsection{Context and motivations}
%------------------------------------------------------------------------------------------------------

In dimension greater than $2$, one mechanism that prevents hyperbolicity is the simultaneous occurrence of saddles of different types of hyperbolicity inside a topologically transitive set, often referred to as \emph{index-variability}. 
Index-variability may occur when there are two regions with different type of hyperbolicity which are mingled by the dynamics, that is, there are orbits going from one region to the other and \emph{vice versa}. 
Here we explore what consequences  the coexistence of such types of regions has, in particular on the hyperbolicity, the ergodic level, and the topology of the space of measures. 

The mere existence of such regions does not, in general, lead to index-variability, unless there occurs recurrence in between those regions. This can be exemplified, oversimplifying, by the quadratic family $x\mapsto \lambda x(1-x)$ which has contraction in a neighborhood of the critical point and expansion near the fixed point $x=0$ for $\lambda>1$. Here the behavior of the orbit of the critical point plays a key role for the type of dynamics of this map. 
%When the orbit of the critical point is .... then this has consequences for the dynamics. 

The quadratic family is an exhaustively studied model for the so-called critical dynamics. The fact that it is very specific did not prevent it to be a ``guiding example". This is justified, besides by its rich dynamics, by its simplicity. Moreover, it also serves as a ``local plug'' which reappears in many phenomena, for instance in the renormalization at homoclinic tangencies \cite[Chapter 3.4]{PalTak:93}.  It also illustrates the passage from trivial dynamics to full chaos \cite{MilThu:88}. We propose a somewhat analogous model in a non-critical partially hyperbolic context
to analyze mingled regions of hyperbolicity. Our focus is on a very simple setting that captures its essence and at the same time displays all its dynamical richness.

%Here we explore what consequences  this coexistence has on the ergodic level, the structure of measures, and the topology of the space of measures. Index-variability can occur in many ways. We describe it here in the simplest setting that still captures its essence and at the same time displays all its dynamical richness. 
 
%One of our motivations comes from the quadratic family, which is an exhaustively studied quintessential model for the so-called critical dynamics. The fact that it is very specific, did not prevent it to be a ``guiding example". This is justified, besides by its rich dynamics, by it simplicity and the fact that it allows for computations. Moreover, it also serves as a ``local plug'' which reappears in many phenomena, for instance in the renormalization at homoclinic tangencies \cite[Chapter 3.4]{PalTak:93}.  Another of its interesting features is that it illustrates the passage from trivial dynamics to full chaos \cite{MilThu:88}. We propose a somewhat analogous model in a non-critical partially hyperbolic context. % with index-variability. 

The general scenery of non-critical partially hyperbolic dynamics is very ample. A first issue here is the dimension of the central direction. We will focus on the one-dimensional case.  In what follows, the terms \emph{contraction} and \emph{expansion} refer to what occurs in the central direction. 
Examples having expanding and contracting regions which are mingled by the dynamics are derived from-Anosov diffeomorphisms \cite{Man:78} and the skew-products in \cite{BonDia:96,GorIly:00}. In these examples, the dynamics is transitive and displays index-variability. 
%Still, there are vast possibilities. Focusing on dimension $3$ only, several paradigmatic examples have been studied such as derived from-Anosov diffeomorphisms \cite{Man:78} and some skew-products \cite{BonDia:96,GorIly:00}. 
Another very simple one in \cite{GorIlyKleNal:05} is also one of the motivations of this paper. It consists of a skew product over the full shift and two circle diffeomorphisms as fiber maps, one being a rotation and the other one a Morse-Smale map. One disadvantage of these examples is their global nature in the sense that they enclose all their dynamics which prevents their use as local plugs.

% \cite{Dia:95,DiaRoc:01}

 Finally, somewhat more appropriate for a semi-local analysis, there are heterodimensional cycles where the contracting and expanding regions are neighborhoods of saddles of different indices and are mingled by a cyclic intersection of the invariant manifolds of the saddles. Still, the dynamics originated from a heterodimensional cycle may or may not lead to index-variability  \cite{Dia:95,DiaRoc:97} (see further discussion in Section \ref{sec:homscen}). Observe that the constructions in all the above cited examples involve the explicit knowledge of appropriate contracting and expanding regions.

% Index-variability is either a consequence of the global dynamics or obtained after unfolding a heterodimensional cycle. 
%In fact, in the $C^1$ setting and allowing for perturbations, one phenomenon implies the other (see \cite{hayashi, jussieu}). 
%In all those examples, there is an a priori explicit use of the respective regions.

\begin{figure}[h] 
 \begin{overpic}[scale=.30]{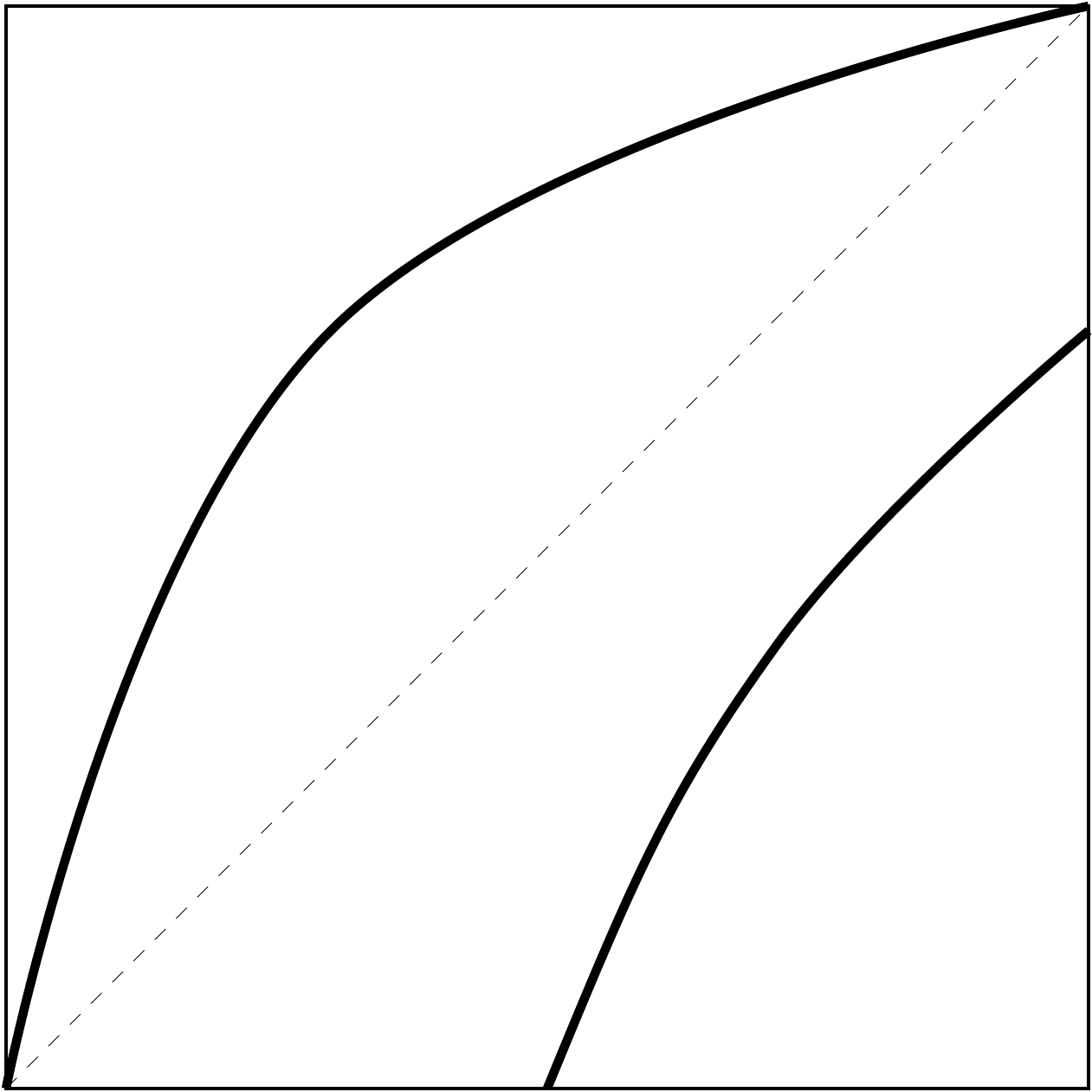}
 	\put(50,74){\small $f_0$}
 	\put(80,44){\small $f_1$}
 	\put(48,-7){\small $d$}
 	\put(0,-7){\small $0$}
 	\put(98,-7){\small $1$}
 \end{overpic}
\hspace{.8cm}  
 \begin{overpic}[scale=.3]{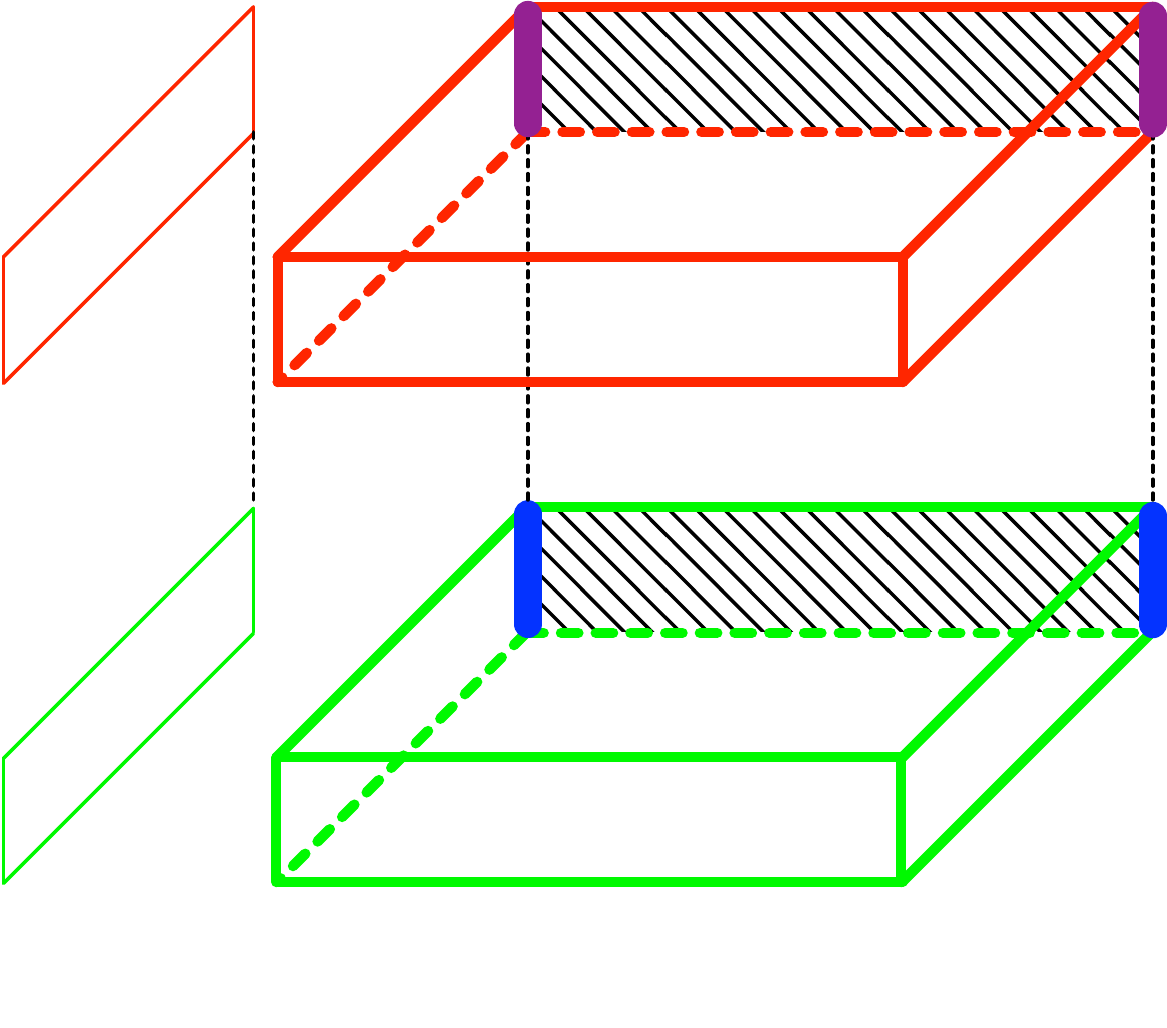}
	\put(-10,60){\small $[1]$}
	\put(-10,20){\small $[0]$}
  \end{overpic}
\hspace{0.2cm}  
\begin{overpic}[scale=.30]{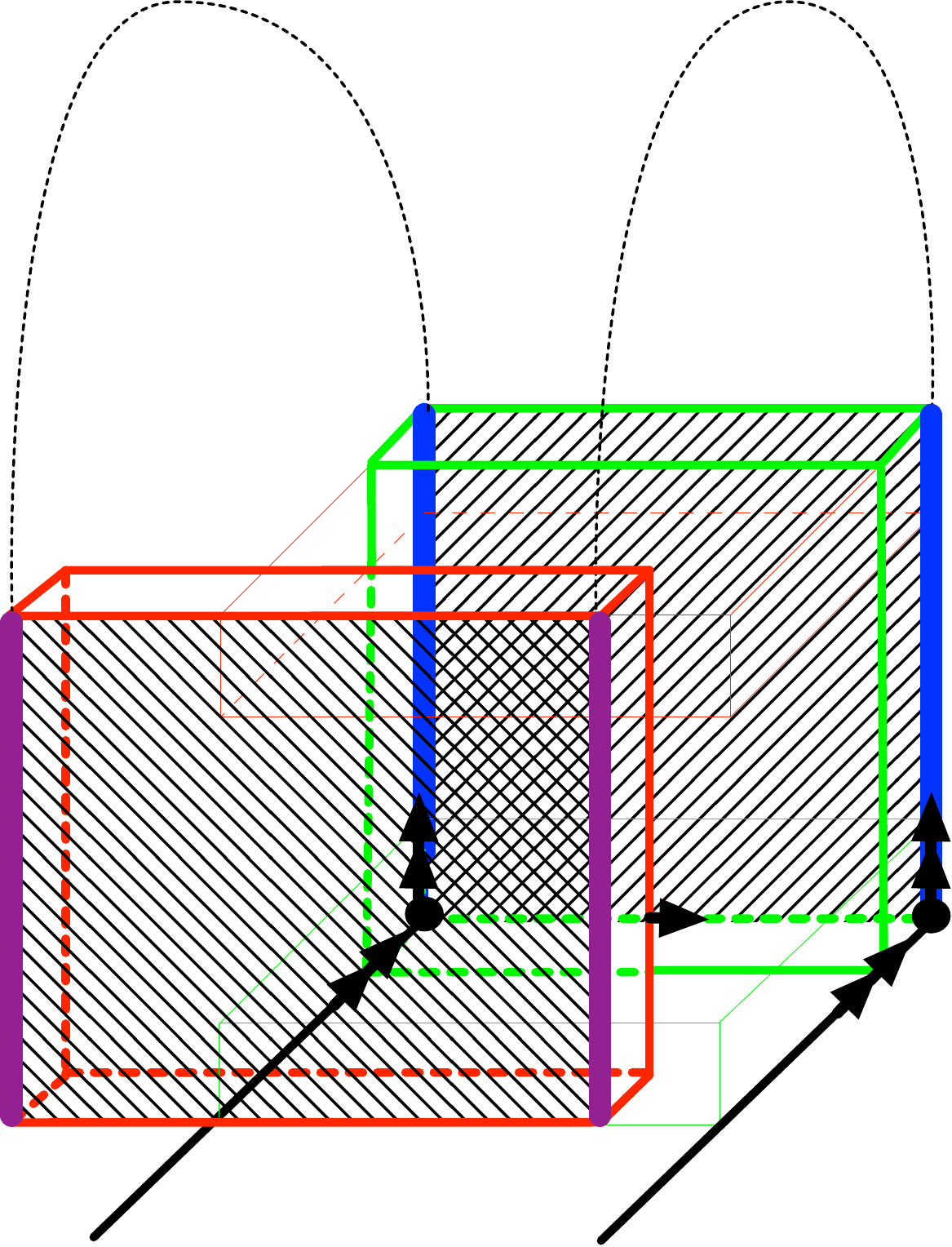}
	\put(74,20){\small $P$}
	\put(34,20){\small $Q$}
  \end{overpic}
 \caption{Fiber maps $f_0,f_1$ and the skew-product $\tilde F$}
 \label{fig.0}
\end{figure}
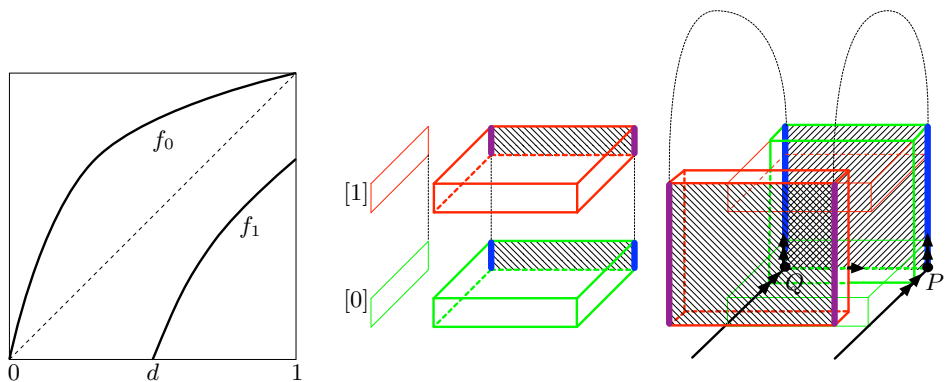

Here we focus on the essence of the dynamical interaction
between a contracting and an expanding region. These regions are neighborhoods of saddles of different type of hyperbolicity.
%and, focusing on the essential features, assuming that there are a repelling and attracting fixed point and consider some neighborhoods of them.
We consider a skew-product over a shift with interval-fiber maps, one map with a repeller and an attractor as extremal points and another one providing the interaction between these points (compare Figure \ref{fig.0} and hypotheses (H1) and (H2)). No cycle is \emph{a priori} involved, although heterodimensional cycles may appear. Hyperbolicity is only determined by the dynamics in the fibers. We assume that the fiber maps are concave. Here concavity is not merely a technical artificial hypothesis, but it prevents the presence of additional independent subsystems%
\footnote{Here we mean, using Conley's approach \cite{Con:78}, sets which are separated from the initial saddles by a filtrating neighborhood. More details are given in Section \ref{sec:bif}.}.    
%and the transition from and propose the simplest model for such an interaction. 
Besides its simplicity, as in the case of the quadratic family, this model captures the essential features of non-hyperbolic partial hyperbolicity with one-dimensional center. 
It also displays the transition from trivial to fully chaotic dynamics. The orbits of $1$ and $0$ play a key role for the dynamics, somewhat as the orbit of the critical point does in the quadratic family.
We do not rely on topological properties such as minimality, existence of blenders, accessibility, specification-like properties, or arguments  based upon synchronization, as \emph{a priori} it is unclear if they hold true.

In our setting, the examples do not \emph{a priori} satisfy any transitivity. Two sorts of dynamics can arise:  hyperbolic one with two horseshoes of contracting and expanding type, respectively, or nonhyperbolic one with two homoclinic classes (associated to saddles of different type) with nonempty intersection. The latter can have several possibilities, two extremal ones are: the homoclinic classes coincide or the homoclinic classes intersect in a single parabolic periodic orbit. Here the \emph{homoclinic class} of a hyperbolic periodic point is just the closure of the intersection of its invariant sets (no transversality is required in our setting), see Remark \ref{rem:homhet} and Section \ref{ss.homoclinicclasses}.

\subsection{Setting}
%------------------------------------------------------------------------------------------------------

Consider a skew-product over a two-symbol shift space $\Sigma_2=\{0,1\}^\bZ$ with concave interval fiber maps, a map $f_0$ with two fixed points, one expanding and one contracting, and a map $f_1$ which forces some ``interaction" between them. The two crucial hypotheses we always consider are the following:
\begin{itemize}
\item[(H1)] $f_0\colon [0,1]\to [0,1]$ and $f_1\colon [d,1]\to [0,1]$, $d\in(0,1)$, are $C^1$ increasing maps such that $f_0$ is onto, $f_0'(0)>1$, $f_0'(1)\in(0,1)$, $f_0(x)>x$ for every $x\in(0,1)$,  $f_1(d)=0$, $f_1(x)<x$ for every $x\in[d,1]$, and $f_1'(1)>0$.
\item[(H2)] $f_0'$ is strictly decreasing and $f_1'$ is not increasing.
\end{itemize}
(compare Figure \ref{fig.0}).
%These hypotheses provide a setting which is on one hand very easy to describe and mixes contracting and expanding behaviors and on the other hand allows to have a very rich dynamics. 
Note that $f_0$ has two fixed points: $0$ being expanding and $1$ being contracting. The ``interaction" above occurs, for instance, when the ``forward orbit" of the contracting point $1$ accumulates at the contracting point $0$; a very special case happens when $0$ is contained in that orbit when a so-called heterodimensional cycle exists (see Remark \ref{rem:herocycle} for further discussion). 

To define the actual skew-product we will study, consider any pair of strictly increasing differentiable extensions $\tilde f_0,\tilde f_1\colon \bR\to\bR$ of $f_0,f_1$ to the real line and let 
\begin{equation}\label{eq:parental}
	\tilde F\colon \Sigma_2\times\bR \to \Sigma_2\times\bR, \quad
	(\xi,x) \mapsto \tilde F(\xi,x)\eqdef (\sigma(\xi), \tilde f_{\xi_0}(x)).
\end{equation}
We define the maximal invariant set $\Gamma$ of $\tilde F$ in $\Sigma_2\times [0,1]$,
\begin{equation}\label{def:Gamma}
	\Gamma
	\eqdef \bigcap_{n\in\bZ}\tilde F^n(\Sigma_2\times[0,1]).
\end{equation}
We study the topological dynamics and the ergodic properties of $\tilde F$ on $\Gamma$. Observe that $\Gamma$ is a locally maximal%
\footnote{Given a compact metric space $X$ and a continuous map $T\colon X\to X$, we say that a subset $Y\subset X$ is \emph{locally maximal} if there exists an open neighborhood $V\subset X$ of $Y$ such that $Y=\bigcap_{n\in\bZ}T^n(V)$.}%
, compact, and $\tilde F$-invariant set whose dynamical properties do not depend on the chosen extensions of $f_0,f_1$. We denote by $F$ the restriction of $\tilde F$ to $\Gamma$. 

We will see that assumptions (H1) and (H2) imply the existence of pairs of horseshoes of fiber contracting and fiber expanding type, respectively. The crucial point is the interaction between them.
%The dynamics of $\tilde F$ on $\Gamma$ can \emph{a priori} exhibit two opposed behaviors: It may be hyperbolic%
%\footnote{The nonwandering set of $\tilde F$ in $\Gamma$ is the union of two hyperbolic sets.}
%or transitive (and therefore not hyperbolic). 
%In the latter case $\Gamma$ contains pairs of horseshoes as above that are heteroclinically related in a cyclic way.
To comment a little more our hypotheses, note that, after changing the reference interval accordingly, (H1) and the first hypothesis in (H2) persist under $C^2$ perturbations (the second hypothesis in (H2) also persists if we would require that $f_1'$ is strictly decreasing). In some results we will also use the following slightly stronger version of (H2), assuming additionally that $f_1'$ is strictly decreasing: 
\begin{itemize}
\item[(H2+)] There exists $M>1$ such that for every $x<y$ we have
\[	M^{-1}(y-x)
	\le \log f_i'(x)-\log f_i'(y)
	\le M(y-x),
	\quad i=0,1
\]
\end{itemize}
Note that (H2+) also persists under $C^2$ perturbations.

We follow two \emph{a priori} independent and complementary approaches. On the one hand, the system has associated a one-dimensional \emph{iterated function system} (IFS). Though observe that only certain concatenations are allowed (for instance, if $x\in[0,d)$ then only $f_0$ can be applied) giving rise to a certain subshift (defined in~\eqref{eq:defSigma}) which describes precisely the dynamics in between the saddles. This subshift, from the point of view of ergodic measures, completely encodes ergodic and entropic properties of the skew-product. In general, this subshift is not of finite type and does not satisfy specification. There are no present-day tools available to study it. Although we are able to show that it is essentially coded and as such is, following the present day classification of shift spaces, just beyond the class of transitive shifts with the specification property. As this coded shift appears naturally in a, to a certain extent, unusual context, it may serve as a good testing ground for the theory of coded systems.
On the other hand, we can view $\Gamma$ as a locally maximal invariant set of a ``three-dimensional partially hyperbolic diffeomorphism with one-dimensional center''. In such a case, the study of the so-called homoclinic classes contained in $\Gamma$ gives substantial dynamical information that we will explore. 

Let us continue to discuss our motivations. The system above can be viewed as a plug in a semi-local analysis of higher dimensional dynamics where horseshoes of different type of hyperbolicity coexist or/and are intermingled, see for instance \cite{Dia:95,BonDia:96,DiaRoc:01}. Its flavor is somewhat similar to so-called blenders (see, for instance, \cite{BonDia:96}). Here we replace the expanding-and-covering property by just concavity.   For appropriate choices (see, for instance, \cite{DiaRoc:97}) this plug models also the bifurcation of heterodimensional cycles (we will explore this in Section \ref{sec:bif}). 
Such heterodimensional cycles appear in many nonhyperbolic contexts. One yet less explored context is, for instance, the study of ${\rm SL}(2,\bR)$-matrix cocycles, following the approach in \cite[Section 11]{DiaGelRam:19}. To see how such a plug appears, consider the projective action of two $2\times 2$ real matrices, one of them hyperbolic giving rise to a map similar to $f_0$ and another one producing $f_1$. The case when there is some (admissible) concatenation sending $1$ to $0$, when a heterodimensional cycle occurs, is precisely  the situation studied in the boundary case in \cite[Theorem 4.1]{AviBocYoc:10}. Our analysis includes such boundary situations, but also goes beyond.

Let us also observe that the set $\Gamma$ has a fractal nature that fits into the category of graph- and bony-like sets introduced in \cite{Kud:10}, that is, measurable (partially multi-valued) graphs which are (with respect to certain measures) graphs from an ergodic point of view but contain continua on a set of zero measure. This is intimately related with the (atomic) disintegration of ergodic measures that we also explore.  

Another motivation is the point of view of IFSs. In general, the study of their statistical properties assumes some type of contraction or contraction-on-average. See for instance \cite{DiaFree:99} for an overview. For contracting-on-average IFS \cite{BarDemEltGer:88} establishes the uniqueness of the stationary measure. 
Examples which are beyond any contraction-like hypotheses are studied in \cite{FanSimTot:06} from the point of view of stationary measures (see also \cite{AlsMis:14}). Note that the IFS generated by $\{f_0,f_1\}$ is genuinely non-contracting and \cite{FanSimTot:06,AlsMis:14} can be seen as a boundary case of our setting, providing perhaps new perspectives. 

%------------------------------------------------------------------------------------------------------
\subsection{Summary of results}
%------------------------------------------------------------------------------------------------------

In what follows, \emph{hyperbolicity} refers only to expansion (resp. contraction) of the associated fiber maps and a closed $F$-invariant subset $\Lambda\subset\Gamma$ is \emph{hyperbolic of expanding type} if there are constants $C>0$ and $\alpha>0$ such that for every $(\xi,x)\in\Lambda$ and for every $n\ge1$ we have
\[
	\lvert(f_{\xi_{n-1}}\circ\ldots\circ f_{\xi_0})'(x)\rvert
	\ge Ce^{n\alpha}.
\]
\emph{Hyperbolicity of contracting type} is defined analogously considering backward iterates. 

A special case of a hyperbolic set is a hyperbolic (of either expanding or contracting type) periodic orbit. In our setting, the  orbit of a periodic point $R=(\xi,r)=F^n(R)$  is either hyperbolic or \emph{parabolic}, that is, 
\[
	(f_{\xi_{n-1}}\circ\ldots\circ f_{\xi_0})'(r)
	= 1.
\]
There exist two designated fixed points for $F$,
\begin{equation}\label{eq:fixedpoints}
	Q
	\eqdef (0^\bZ,0),
		\quad
	P
	\eqdef (0^\bZ,1),	
\end{equation}
which are hyperbolic of   expanding and contracting type, respectively. The map $f_1$ introduces an ``interaction'' between these two points and gives rise to rich topological dynamics and ergodic properties. 

Given an ergodic probability measure $\mu$ (with respect to $F$), its \emph{(fiber) Lyapunov exponent} is
\[
	\chi(\mu)
	\eqdef \int\log\,f_{\xi_0}'(x)\,d\mu(\xi,x).
\]
We say that $\mu$ is \emph{hyperbolic} if $\chi(\mu)\ne0$ and \emph{nonhyperbolic} otherwise. The measure is \emph{hyperbolic of contracting type} if $\chi(\mu)<0$ and \emph{hyperbolic of expanding type} if $\chi(\mu)>0$. In this way, the space $\cM_{\rm erg}(\Gamma)$ of ergodic measures (with respect to $F$ on $\Gamma$) splits as
\begin{equation}\label{eq:splittMGamma}
	\cM_{\rm erg}(\Gamma)
	= \cM_{\rm erg,<0}(\Gamma) 
	\cup \cM_{\rm erg,0}(\Gamma) 
	\cup \cM_{\rm erg,>0}(\Gamma), 
\end{equation}
into the sets of ergodic measures with negative, zero, and positive Lyapunov exponent, respectively. We explore the interplay between these three sets.
Hypothesis (H1) implies that  $\cM_{\rm erg,<0}(\Gamma)$ and $\cM_{\rm erg,>0}(\Gamma)$ both are nonempty and we show they are twin-like.
Whether or not $\cM_{\rm erg,0}(\Gamma)$ is empty depends on further analysis, and both possibilities can occur. 
In the case when the set $\cM_{\rm erg,0}(\Gamma)$ is nonempty, we prove that the closures of  $\cM_{\rm erg,<0}(\Gamma)$ and $\cM_{\rm erg,>0}(\Gamma)$ nicely glue along it, see Theorem \ref{teo:accum}. 

We first analyze the projection $\Sigma$ of the set $\Gamma$ to the base $\Sigma_2$ and the shift dynamics on it. We prove that $\Sigma$ is coded (besides some dynamically irrelevant part which can be empty), see Theorem~\ref{teo:coded}.
We also prove that the set $\cM (\Sigma)$ of invariant measures on $\Sigma$ is an entropy-dense Poulsen simplex, see Theorem~\ref{teo:poulsen}. 
We then turn to the skew-product $F$ in order to study its topological properties and the sets supporting ergodic measures. Theorem~\ref{teo:homoclinicclasses} claims that the nonwandering set of $F$ is the union of the homoclinic classes of $P$ and $Q$ and that their disjointness is equivalent to their hyperbolicity.
 
The next step is to study the interplay between the set $\cM_{\rm erg}(\Sigma)$ of ergodic measures in $\Sigma$ and $\cM_{\rm erg}(\Gamma)$.
The set $\cM_{\rm erg}(\Gamma)$ projects onto $\cM_{\rm erg}(\Sigma)$ and any  measure in $\cM_{\rm erg}(\Sigma)$ can be lifted to $\cM_{\rm erg}(\Gamma)$.
Theorem \ref{teo:1} characterizes hyperbolicity of the measures in $\cM_{\rm erg}(\Gamma)$ in terms of these projections and lifts: either a measure in $\cM_{\rm erg}(\Sigma)$ lifts to a unique measure in $\cM_{\rm erg}(\Gamma)$ which turns out to be nonhyperbolic,  or it lifts to exactly two measures which are hyperbolic of different type of hyperbolicity.
Theorem~\ref{teo:2} states that measures in $\cM_{\rm erg}(\Gamma)$ have atomic  disintegration which is described in dynamical terms and provides the graph-like structure of $\Gamma$ in Corollary \ref{cor1}.  

Theorem~\ref{teo:accum} claims that every nonhyperbolic ergodic measure
 in $\cM_{\rm erg}(\Gamma)$ is simultaneously approached in the weak$\ast$ topology and in entropy by ergodic  measures of contracting type and also by ergodic measures of expanding type. This result also extends to any nonergodic measure whose ergodic decomposition has measures of one type of hyperbolicity, see Corollary \ref{cor:accumoneside}. The gluing of $\cM_{\rm erg,<0}(\Gamma)$ and $\cM_{\rm erg,>0}(\Gamma)$ is described in Corollary \ref{cor:arcwiseconn} claiming  
 that the set $\cM_{\rm erg} (\Gamma)$ is arcwise connected if and only if $\cM_{\rm erg,0} (\Gamma)$ is nonempty. This provides meaningful information about the the gluing of  $\cM_{\rm erg,<0}(\Gamma)$ and $\cM_{\rm erg,>0}(\Gamma)$ and it refines in a natural way the variational principle for the entropy, see Corollary \ref{cor:VP}. In the case when $\cM_{\rm erg,0} (\Gamma)$ is empty then the set of ergodic measures consists of two connected components $\cM_{\rm erg,<0}(\Gamma)$ and $\cM_{\rm erg,>0}(\Gamma)$.

Finally, we study the dynamics of globally defined maps $\tilde F$ in \eqref{eq:parental} and discuss bifurcation scenarios. Let us give a rough idea of our constructions, the precise statements are in Section \ref{sec:bif}. We present models of one-parameter families of maps $\tilde F_t$, $t\in [t_{\rm h}, t_{\rm c}]$, induced by families of interval maps $\{\tilde f_0,\tilde f_{1,t}\}$. Here for the boundary parameters $t_{\rm h}$ and $t_{\rm c}$ the map $\tilde f_{1,t}$ is a ``limit case'' of our hypotheses (H1)--(H2). For each parameter $t$ we consider the maximal invariant set $\Gamma^{(t)}$ associated to $\tilde F_t$ defined as in  \eqref{def:Gamma}. These families \emph{completely unfold heterodimensional cycles}. An initial cycle associated to $P$ and $Q$ occurs for $t=t_{\rm h}$. While usually in bifurcation theory one considers parameters $t$ close to the cycle parameter, here we study an entire range of parameters. Similarly to what happens in the quadratic family, we go from essentially trivial dynamics at the cycle parameter $t_{\mathrm{h}}$ (say, dynamics at the boundary of Morse-Smale systems) up to a completely chaotic dynamics for the parameter $t_{\mathrm{c}}$ (with ``full entropy'' $\log2$). Figure~\ref{fig.bif} a) corresponds to the parameter $t_{\rm h}$ where the invariant sets of the fixed points $P$ and $Q$ (of different type of hyperbolicity) intersect cyclically giving rise to a heterodimensional cycle. Figures~\ref{fig.bif} b) and c) depict the two main possibilities for those complete unfoldings.   Figure~\ref{fig.bif} b) depicts the case when the parameter $t_{\rm c}$ corresponds to an intersection between the ``strong unstable" and the stable sets of $Q$ and $\Sigma_2\times\{0\}\subset\Gamma^{(t_{\rm c})}$. Figure~\ref{fig.bif} c) depicts the case when the parameter $t_{\rm c}$ corresponds to an intersection between the unstable and the ``strong stable''  sets of $P$ and $\Sigma_2\times\{1\}\subset\Gamma^{(t_{\rm c})}$. In  Figure~\ref{fig.bif}  b) and c) the set $\Gamma^{(t_{\rm c})}$ contains a horseshoe. See Remark \ref{rem:fig.bif} for a complete description. 
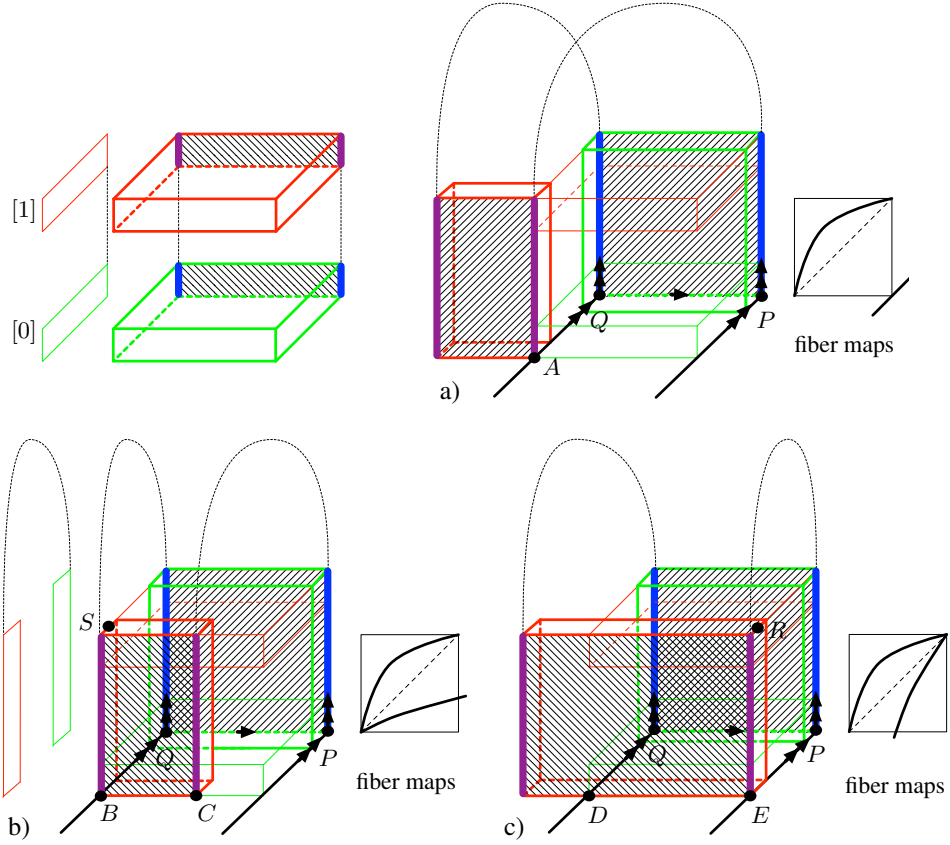
\begin{figure}[h] 
 \begin{overpic}[scale=.34]{P_0.pdf}
	\put(-10,60){\small $[1]$}
	\put(-10,20){\small $[0]$}
  \end{overpic}
  \hspace{1cm}
 \begin{overpic}[scale=.34]{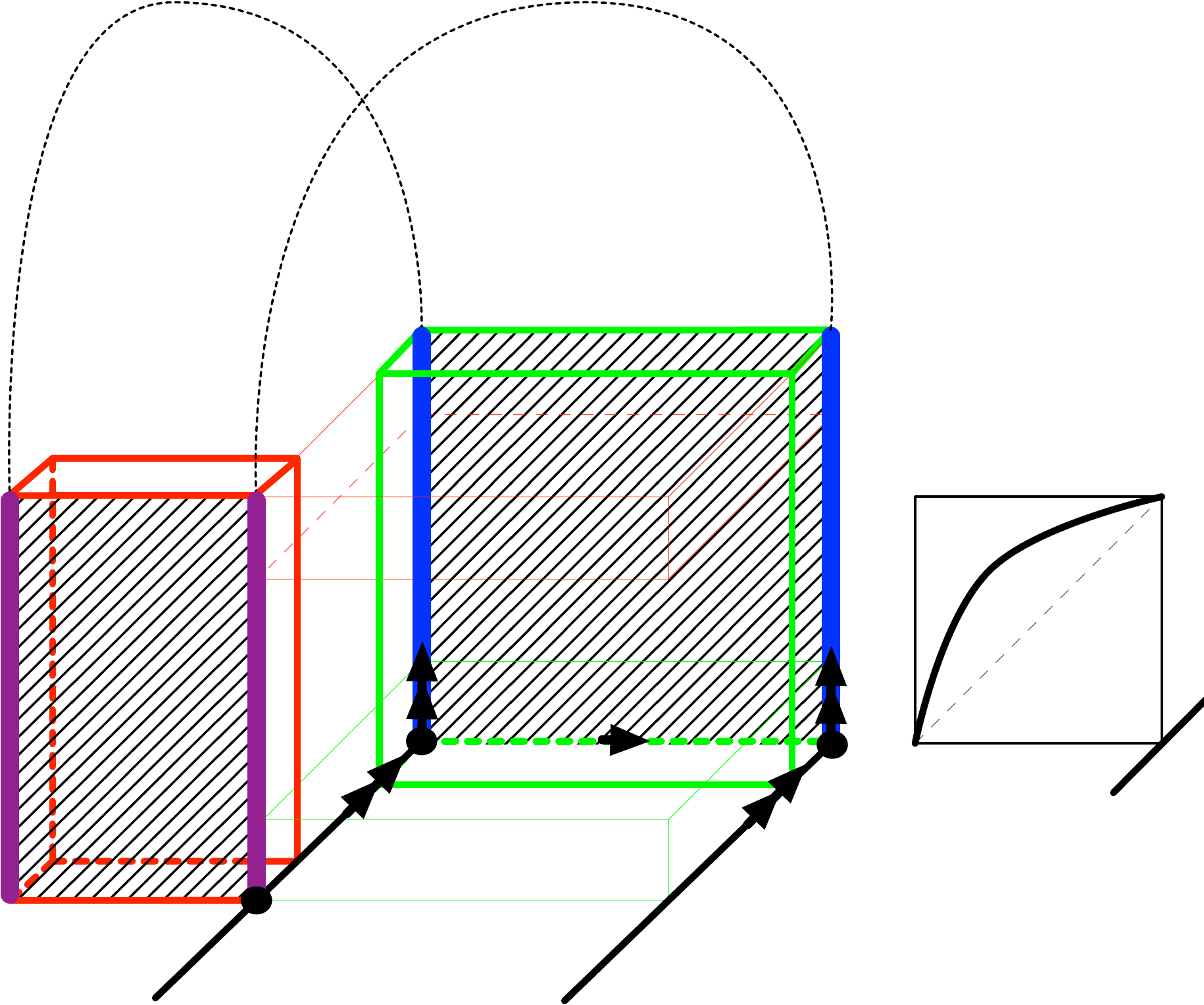}
 	\put(0,0){ a)}
	\put(23,5){\small $A$}
	\put(68,15){\small $P$}
	\put(33,15){\small $Q$}
	\put(76,10){\small fiber maps}
  \end{overpic}
\\[0.5cm]
 \begin{overpic}[scale=.34]{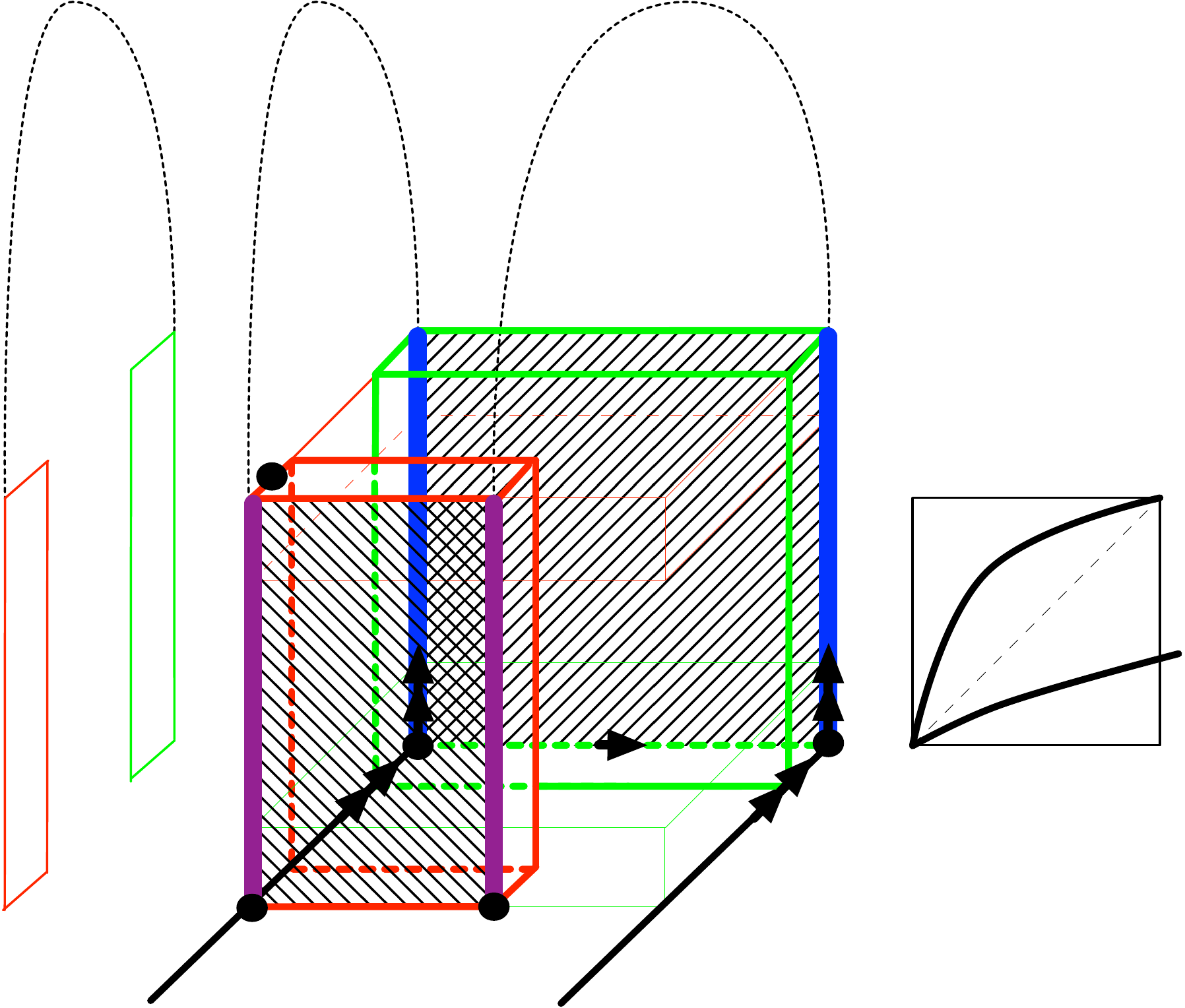}
 	\put(0,0){ b)}
	\put(21,3){\small $B$}
	\put(42,3){\small $C$}
	\put(68,15){\small $P$}
	\put(33,15){\small $Q$}
	\put(16.5,44){\small $S$}
	\put(76,10){\small fiber maps}
  \end{overpic}
  \hspace{0.5cm}
 \begin{overpic}[scale=.34]{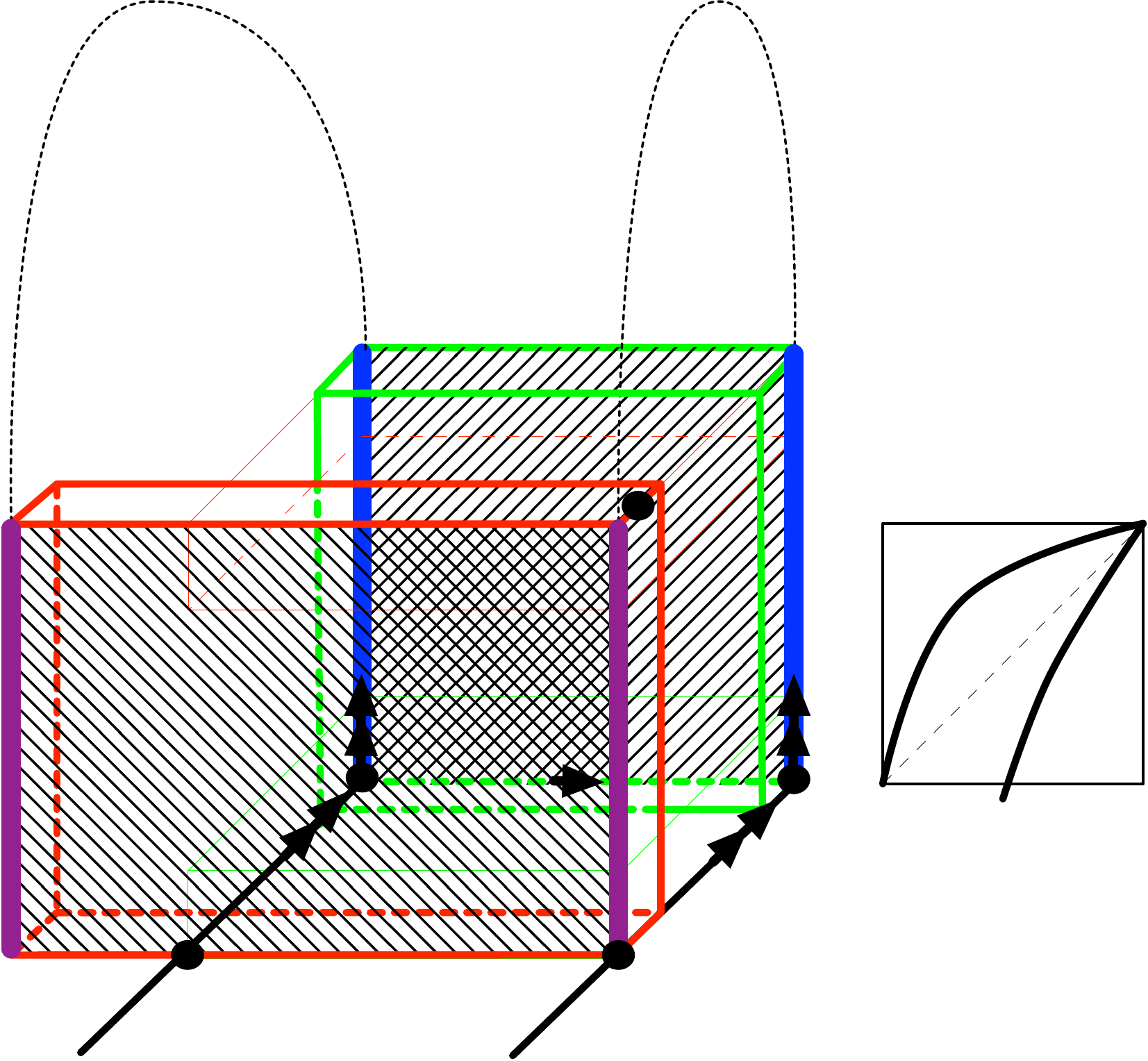}
 	\put(-5,0){ c)}
	\put(16,3){\small $D$}
	\put(54,3){\small $E$}
	\put(58,46){\small $R$}
	\put(67,17){\small $P$}
	\put(30,17){\small $Q$}
 	\put(76,10){\small fiber maps}
 \end{overpic}
 \caption{Complete unfolding of a heterodimensional cycle}
 \label{fig.bif}
 \end{figure}
 
Finally, the parameter $t_{\mathrm{c}}$ can be also seen as a parameter of collision between  $\Gamma^{(t_{\mathrm{c}})}$ and another $\tilde F_{t_{\rm c}}$-maximal invariant set from outside $\Sigma_2\times[0,1]$, leading to a collision of homoclinic classes in the spirit of \cite{DiaSan:04,DiaRoc:07}. This gives rise to explosions of topological entropy and of the space of invariant measures of $\tilde F_{t_{\rm c}}|_{\Gamma^{(t_{\rm c})}}$, see Propositions~\ref{prop:explosionsequences} and \ref{pro:entropyjump}.  

%-------------------------------------------------------------------------------------------------------
\section{Statement of results}
%-------------------------------------------------------------------------------------------------------

%-------------------------------------------------------------------------------------------------------
\subsection{Topological (and hyperbolic) dynamics}
%-------------------------------------------------------------------------------------------------------

Observe first that the restricted domain of $f_1$ will require the consideration of admissible sequences $\xi\in\Sigma_2$. Denote by 
\begin{equation}\label{eq:defpi}
	\pi\colon\Sigma_2\times[0,1]\to\Sigma_2, \quad
	\pi(\xi,x)\eqdef\xi
\end{equation}
the natural projection and, recalling the definition of $\Gamma$ in \eqref{def:Gamma}, define by
\begin{equation}\label{eq:defSigma}
	\Sigma
	\eqdef \pi(\Gamma)
\end{equation}
the (compact and $\sigma$-invariant) set of \emph{admissible} sequences. Our first result is a description of $\Sigma$. In symbolic dynamics a \emph{coded system} is a transitive subshift which is the closure of the union of an increasing family of transitive subshifts of finite type (SFT). 
Define 
\begin{equation}\label{eq:xiPQ}
	 \Sigma^{\rm het}
	 \eqdef 
	 \{\xi\in\Sigma\colon \xi= (\ldots 000 \tau_1\ldots\tau_n 000\ldots)\text{ and }
	 	(f_{\tau_n}\circ\ldots\circ f_{\tau_1})(1)=0\}
\end{equation}
and
\begin{equation}\label{eq:xiPQbis}
	\Sigma^{\rm cod}
	\eqdef \Sigma\setminus\Sigma^{\rm het}.
\end{equation}

\begin{mtheo}\label{teo:coded}
	Assume (H1). The subshift $\Sigma$ is decomposed as
\[
	\Sigma
	= \Sigma^{\rm cod} \cup \Sigma^{\rm het},
\]	
where $ \Sigma^{\rm cod}$ is compact, $\sigma$-invariant, topologically mixing, and coded and $ \Sigma^{\rm het}$ is an at most countable union of isolated points. Moreover, $h_{\rm top}(\sigma,\Sigma^{\rm het})=0$ and there is no $\sigma$-invariant probability measure supported on $\Sigma^{\rm het}$.
\end{mtheo}

For ``generic choices" of maps $f_0,f_1$, the set $\Sigma^{\rm het}$ is empty and hence $\Sigma$ is coded, see Remarks \ref{rem:21Gammahet} and \ref{rem:KuSm}. 
Accordingly, the set $\Gamma$ in general splits naturally into two subsets
\[
	\Gamma^{\rm cod}\eqdef\pi^{-1}(\Sigma^{\rm cod})
	\quad\text{ and }\quad
	\Gamma^{\rm het}\eqdef\pi^{-1}(\Sigma^{\rm het}).
\]	
Recall that a point $X\in\Gamma$ is \emph{non-wandering} for $(\Gamma,F)$ if any neighbourhood of $X$ (in $\Gamma$) contains points whose forward orbit returns to it. The set of all non-wandering points, denoted by $\Omega(\Gamma,F)$, is $F$-invariant and closed.%
\footnote{Recall that any non-wandering point is recurrent and that the converse implication is in general false. Notice the non-wandering set not only depends on the mapping but also on its domain and thus the non-wandering sets for $\tilde F$ and for $F$ may differ. For instance, the set $\Gamma^{\rm het}$ is wandering for $F$ and is non-wandering for $\tilde F$.}
Note that $\Gamma$ may contain points which fail to be non-wandering points or recurrent points, see Remark \ref{rem:examples}. Indeed, if the set $\Sigma^{\rm het}$ is nonempty then it consists of isolated points which are not recurrent, see Theorem \ref{teo:coded}.

An important structure in differentiable dynamics is the one of homoclinic relation which we adapt to our setting. 

\begin{remark}[Homoclinic and heteroclinic relations]\label{rem:homhet}
Following \cite[Chapter 9.5]{Rob:95}, two hyperbolic periodic points $A$ and $B$ are \emph{heteroclinically related} if the invariant sets of their orbits intersect cyclically. In our setting, there are two types of heteroclinic relations. If $A$ and $B$ have different type of hyperbolicity we say that they form a \emph{heterodimensional cycle}, otherwise we say that they are \emph{homoclinically related} (in agreement with the terminology by Newhouse, though here transversally is not needed). The latter defines an equivalence relation among hyperbolic periodic points of the same type of hyperbolicity. Given a hyperbolic periodic point $A$, we call the closure of its equivalence class a \emph{homoclinic class} and denote it by $H(A,F)$, see Section \ref{ss.homoclinicclasses}. In some cases we  will also consider the globally defined map $\tilde F$ and use the corresponding concepts.
\end{remark}

\begin{remark}[The set $\Gamma^{\rm het}$ and heterodimensional cycles]\label{rem:21Gammahet}
The above defined set $\Gamma^{\rm het}$ is indeed the set of points which are heteroclinic to the points $P$ and $Q$ defined in \eqref{eq:fixedpoints}.  In relation to what we defined above, note that any point $(0^\bZ,x)$, $x\in(0,1)$, is forward asymptotic to $P$ and backward asymptotic to $Q$. The nontrivial part to form a heterodimensional cycle associated to $P$ and $Q$ is precisely provided by points in $\Gamma^{\rm het}$ (if nonempty).  See Remark \ref{rem:herocycle} for further details and Figure \ref{fig.het}.
\end{remark}

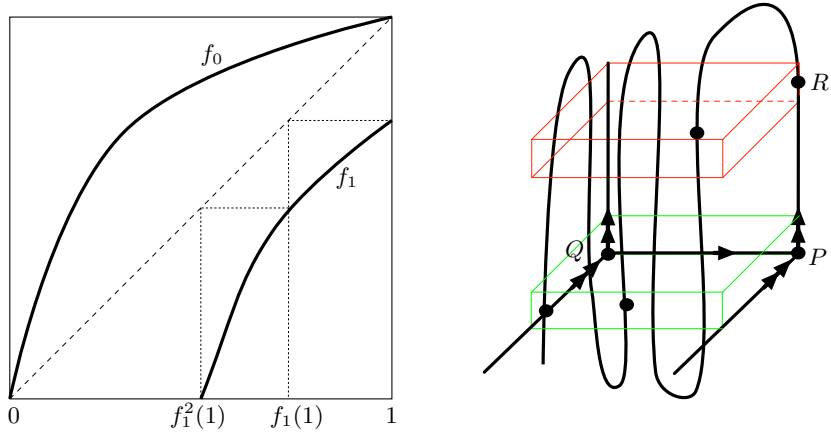
\begin{figure}[h] 
 \begin{overpic}[scale=.4]{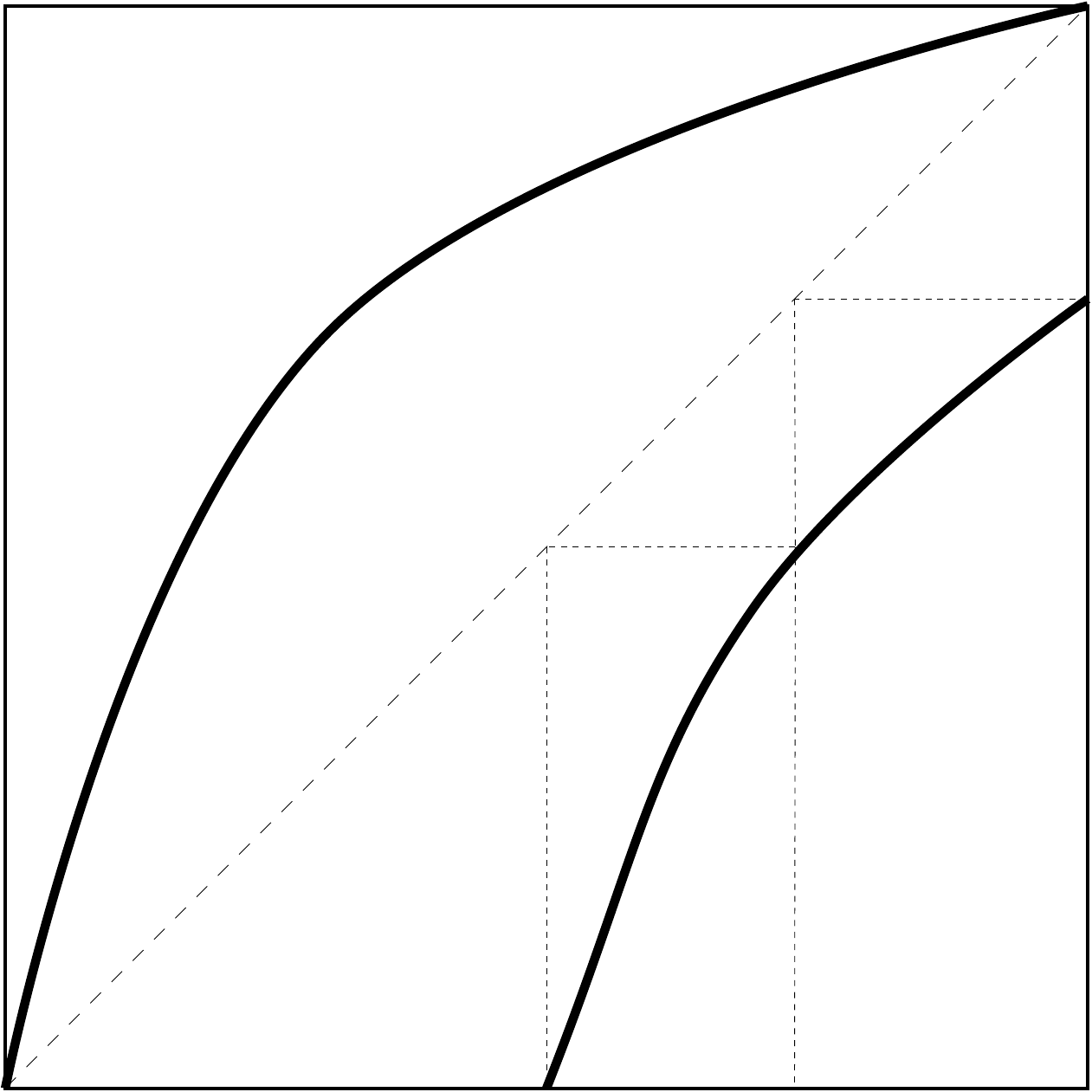}
 	\put(50,88){\small $f_0$}
 	\put(85,56){\small $f_1$}
 	\put(42,-6){\small $f_1^2(1)$}
 	\put(0,-6){\small $0$}
 	\put(98,-6){\small $1$}
 	\put(68,-6){\small $f_1(1)$}
  \end{overpic}
  \hspace{1cm}
 \begin{overpic}[scale=.40]{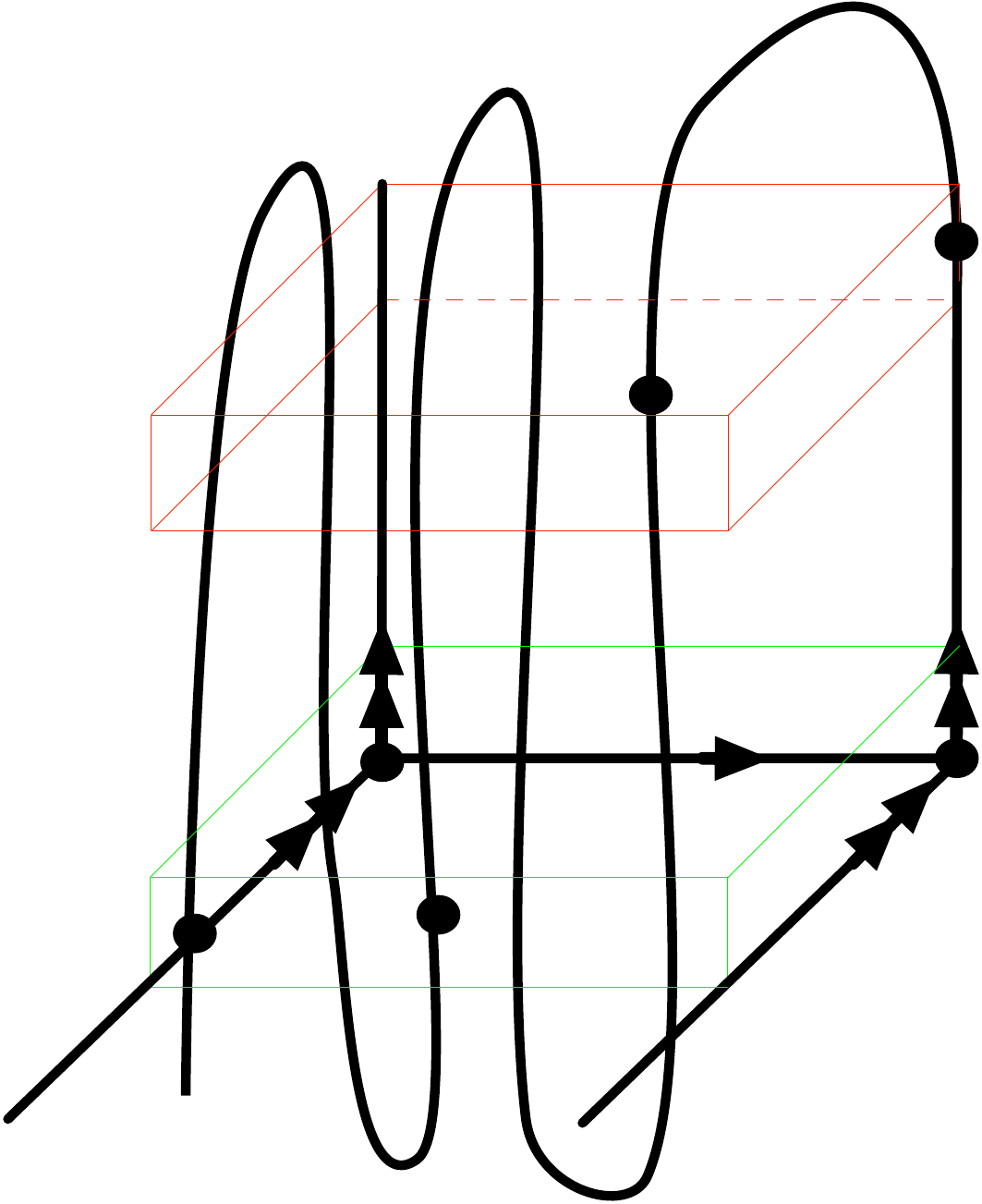}
 	\put(82.5,78){\small $R$}
 	\put(82,34){\small $P$}
	\put(21,36){\small $Q$}
  \end{overpic}
 \caption{Example when $\Gamma^{\rm het}$  is nonempty ($\Gamma ^{\rm het}$ contains the point  $R=((0^{-\bN}.110^\bN),1)$  and its $F$-orbit). }
 \label{fig.het}
\end{figure}

\begin{remark}\label{rem:KuSm}
By a Kupka-Smale genericity-like result, generically there are no heterodimensional cycles associated to $P$ and $Q$ and hence $\Gamma^{\rm het}$ is empty.
\end{remark}

\begin{remark}[Basic sets and horseshoes]
	In what follows, we call \emph{basic} any set which is locally maximal, compact, $F$-invariant, topologically transitive, and has uniform contraction (or uniform expansion) in the fiber direction. When the basic set is uncountable (not a periodic orbit) and topologically mixing, we call it a \emph{horseshoe}.
\end{remark}

\begin{remark}[Dynamics of $\Gamma^{\rm cod}$]
	In general, assuming only (H1), the topological dynamics of the map $F$ on $\Gamma^{\rm cod}$	
 may exhibit a huge variety (also reflected on the  ergodic level). A more detailed discussion is done in Section \ref{sec:homscen}.  Assuming additionally (H2), then hyperbolic periodic points of the same type of hyperbolicity are homoclinically related (see Proposition \ref{pro.l.homoclinicallyrelated}) and hence we have only two homoclinic classes.
 Still, under hypothesis (H2), we observe two main ``opposed'' types of dynamical behaviors that can occur:
\begin{itemize}
\item[1)] The map $F|_{\Gamma^{\rm cod}}$ is transitive and hence not hyperbolic (in this case, we may have either $\Gamma^{\rm het}=\emptyset$ or $\Gamma^{\rm het}\ne\emptyset$). 
\item[2)] The map $F|_{\Gamma^{\rm cod}}$ is not transitive and the set of nonwandering points of $F|_{\Gamma^{\rm cod}}$ is hyperbolic and consists of two disjoint  horseshoes, one expanding and the other one contracting in the fiber direction. In this case, $\Gamma^{\rm cod}$ also contains wandering points, for instance the subset $\{0^\bZ\}\times(0,1)$.%
\footnote{Note that this situation is compatible with $\Gamma^{\rm het}\ne\emptyset$, in which case we have that $\Omega(F)$ is hyperbolic but $\Gamma^{\rm het}\subset\Omega(\tilde F)$ and hence $\Omega(\tilde F)$ is not hyperbolic.}
\end{itemize}
Moreover, there are also a somewhat intermediate scenario:
\begin{itemize}
\item[3)] The map $F|_{\Gamma^{\rm cod}}$ is not transitive and $\Gamma^{\rm cod}$ contains ``touching" or ``overlapping" com\-po\-nents of different type of hyperbolicity.
\end{itemize}
\end{remark}

To be able to describe any further structure, in what follows, besides (H1) and (H2) we will invoke the (uniform concavity) hypothesis (H2+) strengthening (H2). Under conditions (H1)--(H2+), we can prove that every {\emph{non-wandering point of $F$}} in   $\Gamma^{\mathrm{cod}}$ (or, in view of Theorem~\ref{teo:coded}, any non-isolated non-wandering point in $\Gamma$) is approximated by hyperbolic periodic points. Indeed, we have a more accurate statement:

\begin{mtheo}\label{teo:homoclinicclasses}
Assume (H1)--(H2+). The non-wandering set $\Omega(\Gamma,F)$ is contained in $\Gamma^{\rm cod}$ and satisfies
\[
	\Omega(\Gamma,F)
	= H(P,F) \cup H(Q,F)
	= {\rm closure}\{A\in\Gamma\colon A\text{ hyperbolic and periodic}\}.
\]
Moreover, the sets $H(P,F)$ and $H(Q,F)$ both are hyperbolic  if and only if they are disjoint.
\end{mtheo}

\begin{questions}\label{rem:forgetquest}
A consequence of concavity is that hyperbolic periodic points appear in contracting--expanding pairs. Moreover, if the Lyapunov exponents of the contracting-type periodic points is bounded away from zero then the same holds for the ones of expanding-type (see Theorem \ref{teo:1}). This occurs, for example, if  $H(P,F)$ is hyperbolic. In particular, hyperbolicity of $H(P,F)$ implies that every periodic point has exponent away from zero and hence all  ergodic measures are hyperbolic (by Theorem \ref{teo:accum} and Remark \ref{Katokforevery}).

We do not know if hyperbolicity of one homoclinic class implies its local maximality.
We also do not know if hyperbolicity of one class implies hyperbolicity of the other one. If, for example, $H(P,F)$ is hyperbolic and $H(Q,F)$ is not, then all ergodic measures supported on $H(Q,F)$  are hyperbolic (\emph{a priori} either of contracting or of expanding type) with exponent bounded away from zero. It is a challenging problem if this case can occur and what further properties there are.%
\footnote{There are nonhyperbolic systems whose ergodic measures are all hyperbolic. Let us cite a few examples of different flavors.  The simplest example is $x\mapsto 4x(1-x)$ on the unit interval.  More related to diffeomorphisms, the example in \cite{BalBonSch:99} is based on an elaboration of Bowen's ``eye-like'' attractor. In \cite{CaoLuzRio:06}, the authors analyze the boundary of hyperbolicity in horseshoes with internal tangencies. Finally, homoclinic classes with internal heterodimensional cycles studied in \cite{DiaHorRioSam:09} are more in the direction of systems in this paper.} 
 
In general, it is unclear what further dynamical properties (e.g. entropy and existence of nonhyperbolic ergodic measures) the (nonempty and nontrivial) intersection $H(P,F)\cap H(Q,F)$ can have. 
\end{questions}

%%%We also have the following twinning of basic sets, which we state without proof.%
%%%\footnote{To prove it, use the key property that hyperbolic periodic points appear in contracting--expanding pairs and that periodic hyperbolic points of the same type are homoclinically related, these facts being a consequence of the concavity assumption (see Proposition \ref{pro.l.homoclinicallyrelated}).}
%%%
%%%\begin{proposition}[Twin-basic sets]
%%%	Assume (H1)--(H2).
%%%	For every basic set $\Gamma^+\subset\Gamma$ with uniform fiber expansion there exists a basic set $\Gamma^-\subset\Gamma$ with uniform fiber contraction such that $F|_{\Gamma^-}$ and $F|_{\Gamma^+}$ are topologically conjugate and that $\pi(\Gamma^-)=\pi(\Gamma^+)$. 
%%%\end{proposition}

\begin{questions}[Interplay between symbolic and global dynamics]
A natural question is if the structure of $H(P,F)$ and $H(Q,F)$ provides further information about the symbolic space $\Sigma^{\rm cod}$. If both $H(P,F)$ and $H(Q,F)$ are basic sets, then $\Sigma^{\rm cod}$ is a SFT. However, in general, this is unknown. For instance, if $Q\in H(P,F)$ (and hence this set is not hyperbolic) then $\Sigma^{\rm cod}$ is not SFT.%
\footnote{Indeed, in this case there are ``admissible words" $(\xi_0\ldots\xi_k)$ such that $f_{[\xi_0\ldots\,\xi_k]}(1)=\varepsilon_k\to0^+$ as $k\to\infty$. This then implies that $(\xi_0\ldots\xi_k0^m1)$ is admissible if and only if $m\ge m(\varepsilon_k)$, where $m(\varepsilon_k)\to\infty$ as $k\to\infty$, the latter preventing $\Sigma$ to be a SFT.} 
See also the discussion in Section \ref{sec:homscen}.
\end{questions}

\begin{remark}\label{rem:time-reversal}
	Observe that the IFS of the inverse fiber maps $\{f_0^{-1},f_1^{-1}\}$, after a change of coordinates $x\mapsto 1-x$, again satisfies (H1)--(H2). Moreover, if the original system satisfied (H2+) then the inverse system will also do so. Hence, for many statements, there is a certain symmetry with respect to time reversal.
\end{remark}

The ergodic counterpart of the results above is done in the following section, see in particular Corollary \ref{cor1}.

%-------------------------------------------------------------------------------------------------------
\subsection{Structure of the measure space and Lyapunov exponents}
%-------------------------------------------------------------------------------------------------------

Given a compact metric space $X$ and a continuous map $T$ on $X$, we denote by $\cM(X)$  the space of $T$-invariant Borel probability measures on $X$ and by $\cM_{\rm erg}(X)$ the subspace of ergodic measures. We equip $\cM(X)$ with the Wasserstein distance, denoted by $W_1$, that induces the weak$\ast$ topology (we recall its definition and some properties in Appendix \ref{App:A}). We denote by $h(\nu)$ the \emph{metric entropy} of a measure $\nu\in\cM(X)$.

The space $\cM(X)$  is a Choquet simplex whose extreme points are  the ergodic measures (see \cite[Chapter 6.2]{Wal:82}). If $\cM(X)$ is not a singleton and if the set of ergodic measures $\cM_{\rm erg}(X)$ is dense in its closed convex%
\footnote{Recall that the \emph{convex hull} of a set $\cN\subset\cM(X)$ is the smallest convex set containing $\cN$, denoted by $\conv(\cN)$, and that the \emph{closed convex hull} of $\cN$ is the smallest closed convex set containing $\cN$, denoted by $\cconv(\cN)$. By \cite[Theorem 5.2 (i)--(ii)]{Sim:11}, we have $\overline{\conv(\cN)}=\cconv(\cN)$, where $\overline{\cN}$ denotes the weak$\ast$ closure of $\cN$. }
 hull $\cM(X)$, then one refers to it as a \emph{Poulsen simplex} (see also \cite{LinOlsSte:78}).
Moreover, $\cM(X)$ is an \emph{entropy-dense Poulsen simplex} if for every $\mu\in\cM(X)$,  neighbourhood $U$ of $\mu$ in $\cM(X)$, and  $\varepsilon>0$ there exists $\nu\in\cM_{\rm erg}(X)\cap U$ such that $h(\nu)>h(\mu)-\varepsilon$.

\begin{mtheo}\label{teo:poulsen}
	Assume (H1)--(H2). The space $\cM(\Sigma)$ is an entropy-dense Poulsen simplex.
\end{mtheo}

\begin{remark}[Entropy map]
For the step skew-product map $F$, the entropy map $\mu\mapsto h(\mu)$  is upper semi-continuous on $\cM(\Gamma)$. Indeed, the map $\tilde F$ when seen as a partially hyperbolic diffeomorphism with one-dimensional central bundle is $h$-expansive (see  \cite{DiaFisPacVie:12}). Hence \cite{Bow:72} implies upper semi-continuity. Thus, by Theorem \ref{teo:poulsen}, every $\mu\in\cM(\Sigma)$ can be approximated weak$\ast$ and in entropy by ergodic measures.
\end{remark}

\begin{remark}\label{movethere}
Denote by $\pi_\ast\mu\eqdef \mu\circ\pi^{-1}$ the pushforward of a measure $\mu$ by the projection $\pi$ defined in \eqref{eq:defpi}. 
Note that for every $\mu\in\cM_{\rm erg}(\Gamma)$, the measure $\pi_\ast\mu$ is ergodic. On the other hand, given $\nu\in\cM_{\rm erg}(\Sigma)$, for every $\mu\in\cM(\Gamma)$ satisfying $\pi_\ast\mu=\nu$, every ergodic component $\mu'$ of $\mu$ satisfies $\pi_\ast\mu'=\nu$.

We  observe the following well-known fact%
\footnote{Just observe that $h_{\rm top}(F,\pi^{-1}(\xi))=0$ for every $\xi\in\Sigma$. Hence, by the Ledrappier-Walters formula \cite{LedWal:77}, $h(\nu)=\sup_{\mu\colon\pi_\ast\mu=\nu}h(\mu)$.}
\begin{equation}\label{eqlem:projent}
	h(\pi_\ast\mu)=h(\mu)
	\quad\text{  for every }\quad
	\mu\in\cM(\Gamma).
\end{equation}
\end{remark}

\begin{mtheo}\label{teo:1}
	Assume  (H1)--(H2+). There are  continuous functions $\ukappa,\okappa\colon(0,\infty)\to(0,\infty)$ which are increasing and satisfy 
	$$
	\lim_{D\to0}\ukappa(D)=0=\lim_{D\to0}\okappa(D)
	$$
	such that, given any $\nu\in\cM_{\rm erg}(\Sigma)$, one of the following two cases occurs:
\begin{itemize}
\item[a)] There exist exactly two measures $\mu_1,\mu_2\in\cM_{\rm erg}(\Gamma)$  such that $\pi_\ast\mu_1=\nu=\pi_\ast\mu_2$. In this case, both measures are hyperbolic and have fiber Lyapunov exponents with different signs. More precisely, the Wasserstein distance $D\eqdef W_1(\mu_1,\mu_2)>0$   between $\mu_1$ and $\mu_2$ satisfies	
\[
	D
	= \int x\,d\mu_2 (\xi,x)-\int x\,d\mu_1(\xi,x)
\]	
and 
\begin{equation}\label{eq:eaqaul}
	-\okappa(D) 
	\le \chi(\mu_2)
	\le -\ukappa(D)
	<0
	< \ukappa(D)
	\le \chi(\mu_1)
	\le \okappa(D).
\end{equation}
\item[b)] There exists only one measure $\mu\in\cM(\Gamma)$ such that $\pi_\ast\mu=\nu$. In this case, $\mu$ is ergodic and satisfies $\chi(\mu)=0$. 
\end{itemize}
\end{mtheo}

Corollary \ref{cor1} below  complements Theorem \ref{teo:1} in stating a  complete relation between the spaces of ergodic measures in $\Gamma$ and $\Sigma$. For this we need some further definitions.

Notice that the fixed points $Q$ and $P$ in \eqref{eq:fixedpoints} are contained in the common fiber $\{0^\bZ\}\times [0,1]$. In our setting this picture repeats on other periodic fibers. Given $\xi\in\Sigma_2$, write $\xi=\xi^-.\xi^+$ to denote its forward and backward one-sided sequences. For $\xi\in\Sigma\eqdef\pi(\Gamma)$, define its \emph{spine} by
\[
	\eJ_\xi
	\eqdef \big(\{\xi\}\times[0,1]\big) \cap\Gamma
	= \pi^{-1}(\xi)\cap\Gamma.
\]
Spines are intimately related to the homoclinic structure  and there are two possibilities: either  $\eJ_\xi$ is a continuum of the form $\eJ_\xi=\{\xi\}\times I_\xi$, where $I_\xi=[x_{\xi^+},x_{\xi^-}]$ or  $\eJ_\xi=\{\xi\}\times\{x_\xi\}$ is a singleton. In the latter case, we define $x_{\xi^+}=x_{\xi^-}\eqdef x_\xi$.
In particular, the following two sets are $\sigma$-invariant
\begin{equation}\label{eq:singspine}
\begin{split}
	\Sigma^{\rm spine}
	&\eqdef \{\xi\in\pi(\Gamma)\colon \eJ_\xi\text{ is a continuum}\},\\
	\Sigma^{\rm sing}
	&\eqdef \{\xi\in\pi(\Gamma)\colon \eJ_\xi\text{ is a singleton}\}.
\end{split}
\end{equation}
If $\xi\in\Sigma^{\rm spine}$ is periodic then $\eJ_\xi$ is bounded by two hyperbolic points $(\xi,x_{\xi^\pm})$. If $\xi\in\Sigma^{\rm sing}$ is periodic then $\eJ_\xi=(\xi,x_\xi)$ is a point on a parabolic periodic  orbit, see  Section \ref{sec:3per}.

\begin{remark}[Boundary graphs]\label{rem:graphlikestructure}
Observe that the two naturally associated functions $\xi\mapsto x_{\xi^+}$ and $\xi\mapsto x_{\xi^-}$ are measurable (see \cite[Proposition 3.1.21 and Theorem 5.3.1]{Sri:98}). 
We also observe that our construction provides that for every $\xi\in\Sigma^{\rm cod}$, it holds
\[
	(\xi,x_{\xi^+})\in H(Q,F)
	\quad\text{ and }\quad
	(\xi,x_{\xi^-})\in H(P,F),
\]
 see Proposition \ref{prol.extremalpoints}. The graphs of the functions $\xi\mapsto x_{\xi^+}$ and $\xi\mapsto x_{\xi^-}$ bound precisely $\Gamma$ and the region in between plays a role somewhat similar to a Conley pair in the study of recurrent sets (see \cite[Chapter IX]{Rob:95}). 
This structure also resembles the one of a trapping region in skew-products studied, for instance, in \cite{Kud:10,KleVol:14b,GhaHom:16}, and the structure of the two graphs provides a spiny (bony in the terminology of \cite{Kud:10,KleVol:14b}) structure. In some (necessarily nonhyperbolic) cases, the spines are contained in the nonwandering set $\Omega(\Gamma,F)$. 
\end{remark}

In our setting, concavity forces that ergodic measures are only supported on the ``boundary of that region" in the following sense. The direct product structure provides a disintegration $\{\mu_\xi\}_{\xi\in\Sigma}$ for every measure in $\cM(\Gamma)$, called the \emph{disintegration} of $\mu$, as follows: if $\mu\in\cM_{\rm erg}(\Gamma)$ then $\nu\eqdef \pi_\ast\mu\in\cM_{\rm erg}(\Sigma)$ and  for every measurable set $B\subset\Gamma$ we have
\[
	\mu(B)
	= \int_\Sigma\mu_\xi(B\cap \eJ_\xi)\,d\nu(\xi).
\]
Observe that $\mu_\xi(I_\xi)=1$ for $\nu$-almost every $\xi$. 
This disintegration is \emph{atomic} if $\mu_\xi$ is atomic $\nu$-almost everywhere.

\begin{mtheo}[Atomic disintegration]\label{teo:2}
	Assume  (H1)--(H2+).
	The disintegration of any  $\mu\in\cM_{\rm erg}(\Gamma)$  is atomic. Moreover, 
\begin{itemize}
\item $\mu$ is hyperbolic if and only if $\pi_\ast\mu(\Sigma^{\rm spine})=1$,
\item $\mu$ is nonhyperbolic if and only if $\pi_\ast\mu(\Sigma^{\rm sing})=1$.
\end{itemize}
More precisely, letting $I_\xi=[x_{\xi^+},x_{\xi^-}]$ and denoting by $\delta_x$ the Dirac measure at $x$,
\begin{itemize}
\item if $\mu\in\cM_{\rm erg,>0}(\Gamma)$, then $\mu = \int_\Sigma\delta_{x_{\xi^+}}\,d\pi_\ast\mu(\xi)$,
\item if $\mu\in\cM_{\rm erg,<0}(\Gamma)$, then $\mu = \int_\Sigma\delta_{x_{\xi^-}}\,d\pi_\ast\mu(\xi)$,
\item if $\mu\in\cM_{\rm erg,0}(\Gamma)$, then $x_{\xi^\pm}=x_\xi$ for $\pi_\ast\mu$-almost every $\xi$ and $\mu = \int_\Sigma\delta_{x_{\xi}}\,d\pi_\ast\mu(\xi)
$.
\end{itemize}
\end{mtheo}

The above implies that, besides the natural partition in \eqref{eq:splittMGamma} concerning  hyperbolicity, we have a partition of $\cM_{\rm erg}(\Sigma)$ as follows. 
As the sets $\Sigma^{\rm spine}$ and $\Sigma^{\rm sing}$ both are $\sigma$-invariant, if $\nu\in\cM_{\rm erg}(\Sigma)$ is ergodic then only one of them has full measure. Accordingly, the space of ergodic measures splits  into the two subsets
\[
	\cM_{\rm erg}^\dagger(\Sigma)
	\eqdef \{\nu\in\cM_{\rm erg}(\Sigma)\colon \nu(\Sigma^\dagger)=1\},\quad
	\dagger\in\{\text{spine},\text{sing}\}.
\]
Hence the projection $\pi_\ast$ provides bijections between the spaces
$\cM_{\rm erg,>0}(\Gamma)$, $\cM_{\rm erg,<0}(\Gamma)$, and $\cM_{\rm erg}^{\rm spine}(\Sigma)$
and between 
$\cM_{\rm erg,0}(\Gamma)$ and $\cM_{\rm erg}^{\rm sing}(\Sigma)$. 
Moreover, as ``there is no entropy in the fibers" (see  \eqref{eqlem:projent}), this also extends to the entropy of the measures.
We collect these facts in the following corollary, without further explicit proof. 

\begin{mcoro}[Twin-measures and symmetry of measure spaces]\label{cor1}
	Assume   (H1)--(H2+).
	Then for every $\mu\in\cM_{\rm erg}(\Gamma)$ and $\nu\in\cM_{\rm erg}(\Sigma)$ satisfying $\nu=\pi_\ast\mu$ we have 
\[
	h(\nu)=h(\mu).
\]	
	Moreover, exactly one of the following two possibilities holds:
\begin{itemize}
\item[a)] We have $\mu\in\cM_{\rm erg,<0}(\Gamma)\cup\cM_{\rm erg,>0}(\Gamma)$ and $\nu\in\cM_{\rm erg}^{\rm spine}(\Sigma)$. In this case, there exists exactly one other ergodic measure $\mu'$ also satisfying $\pi_\ast\mu'=\nu$. Moreover, the Lyapunov exponents of $\mu$ and $\mu'$ satisfy \eqref{eq:eaqaul} and, in particular, $\mu$ and $\mu'$ have opposite type of hyperbolicity.
\item[b)] We have $\mu\in\cM_{\rm erg,0}(\Gamma)$ and $\nu\in\cM_{\rm erg}^{\rm sing}(\Sigma)$, and $\pi^{-1}_\ast\nu=\{\mu\}$. 
\end{itemize}
\end{mcoro}
 
In item a) in the above corollary, we call the measures $\mu$ and $\mu'$ \emph{twin-measures}.  Similar phenomena are observed in \cite{DiaGelMarRam:,RodRodTahUre:12}.

\begin{remark}
There is a close relation of our results with the existence of bony attractors in, for example, \cite{Kud:10,KleVol:14b}, though here we face two essentially different properties: On one hand, we do not have ``trapping regions" (key ingredient in \cite{KleVol:14b}) and our system has rather a saddle-type nature. On the other hand, we have to deal with admissible sequences and restricted domains of the fiber maps. 
By Corollary \ref{cor1}, given $\mu\in\cM_{\rm erg,0}(\Gamma)$,  the set $\Gamma$ is a \emph{bony graph} with respect to $\nu=\pi_\ast\mu$ in the sense that its intersection with $\nu$-almost every fiber is a single point (that is $\nu(\Sigma^{\rm sing})=1$) but is a continuum ``otherwise" (that is, $\Sigma^{\rm spine}$ is nonempty but has zero measure). It is possible that $\nu$ has full support $\Sigma$, justifying the comparison with bony graphs.
\end{remark}

%-------------------------------------------------------------------------------------------------------
\subsection{Weak$\ast$ and entropy approximation of ergodic measures}\label{ss:weaentapp}
%-------------------------------------------------------------------------------------------------------

The following theorem is well-known for hyperbolic ergodic measures of diffeomorphisms (see Remark \ref{Katokforevery} below), and hence stated only for nonhyperbolic ones. 

To fix some terminology, given a compact $F$-invariant set $\Xi$, denote by $h_{\rm top}(F,\Xi)$ the \emph{topological entropy} of $F$ on a set $\Xi$ (see \cite{Wal:82} for its definition). 

\begin{remark}
As there is no entropy in the fibers (see \eqref{eqlem:projent}), the variational principle for topological entropy of $F$ on a compact $F$-invariant set $\Xi$ immediately implies that   
\[
	h_{\rm top}(\sigma,\pi(\Xi))=h_{\rm top}(F,\Xi).
\]
\end{remark}

An invariant probability measure is \emph{periodic} if it is supported on a periodic orbit.

\begin{mtheo}[Hyperbolic approximation of nonhyperbolic measures]\label{teo:accum}
	Assume  (H1)--(H2+).
	Then for every measure $\mu\in\cM_{\rm erg,0}(\Gamma)$,  every $\varepsilon_E>0$, and every $\varepsilon_H\in(0,h(\mu))$ there exist a basic set $\Gamma^+\subset\Gamma$ with uniform fiber expansion  and a basic set $\Gamma^-\subset\Gamma$ with uniform fiber contraction such that their topological entropies satisfy
\[
	h_{\rm top}(F,\Gamma^+), h_{\rm top}(F,\Gamma^-) 
	\in[h(\mu)-\varepsilon_H,h(\mu)+\varepsilon_H].
\]
Moreover, every measure $\mu^\pm\in\cM(\Gamma^\pm)$ is $\varepsilon_E$-close to $\mu$ in the Wasserstein metric. In particular, there are hyperbolic measures $\mu^+,\mu^-\in\cM_{\rm erg}(\Gamma)$ with
\[
	\chi(\mu^+)\in(0,\varepsilon_E)
	\quad\text{ and }\quad
	\chi(\mu^-)\in(-\varepsilon_E,0)
\]
and
\[
	h(\mu^\pm)\in[h(\mu)-\varepsilon_H,h(\mu)+\varepsilon_H].
\]

If $h(\mu)>0$ then $\Gamma^\pm$ are horseshoes, otherwise they are hyperbolic periodic orbits.

In particular, every measure in $\cM_{\rm erg}(\Gamma)$ is weak$\ast$ accumulated by hyperbolic periodic measures.
\end{mtheo}

The following is an immediate consequence of Theorem \ref{teo:accum}, stated without proof.

\begin{mcoro}[Variational principle for entropy]\label{cor:VP}
	Assume (H1)--(H2+). Then
\[
	h_{\rm top}(F,\Gamma)
	= \sup_{\mu\in\cM_{\rm erg,<0}(\Gamma)}h(\mu)
	= \sup_{\mu'\in\cM_{\rm erg,>0}(\Gamma)}h(\mu').	
\]
\end{mcoro}

\begin{remark}[Katok's horseshoes]\label{Katokforevery}
	For every hyperbolic measure $\mu\in\cM_{\rm erg}(\Gamma)$, it is well-known that in our ``partially hyperbolic setting" there exists a sequence of horseshoes (with uniform hyperbolicity) $(\Gamma_n)_n$ such that $\cM(\Gamma_n)$ converges weak$\ast$ to $\mu$ and also $h_{\rm top}(F,\Gamma_n)$ converges to $h(\mu)$. Here one can use Katok's horseshoe constructions either assuming $C^2$ \cite[Chapter S.4]{KatHas:95} or assuming $C^1$ regularity plus domination \cite{Cro:11,Gel:16}. In particular, every periodic measure in $\Gamma_n$ is weak$\ast$-close to $\mu$.
\end{remark}

Comparing Theorem \ref{teo:accum} with previous results, the analogous one was obtained in \cite{DiaGelRam:17} for transitive step skew-products with fiber maps being circle diffeomorphisms assuming the existence of so-called expanding/contracting blending intervals and forward/back\-ward minimality of the underlying IFS. Those are quite strong (though natural) properties and they describe the somewhat intermingled structure of two types of hyperbolicity.
The constructions in \cite{DiaGelRam:17} was extended in \cite{DiaGelSan:19,YanZha:} to some partially hyperbolic diffeomorphisms with minimal strong foliations, following the strategy outlined in \cite[Section 8.3]{DiaGelRam:17}. Here \emph{a priori} we do not have these hypotheses and in many cases they indeed fail. 

The main tool to prove Theorem \ref{teo:accum} are the so-called skeletons that rely only on ergodic properties, see Definition \ref{def:skeleton}. Their existence only requires (H1), see Proposition \ref{pro:skeleton}. Assuming additionally (H2), the main step towards the proof of Theorem \ref{teo:accum} is the following result.  

\begin{mtheo}[Shadowplay]\label{teopro:intermediate}
	Assume (H1)--(H2). For every $\mu\in\cM_{\rm erg}(\Gamma)$ there exists a sequence of basic sets $(\Upsilon_n)_n\subset\Gamma$ such that for every sequence $(\mu_n)_n$, $\mu_n\in\cM(\Upsilon_n)$, the corresponding sequence $(\nu_n)_n$, $\nu_n=\pi_\ast\mu_n$, converges weak$\ast$ to $\pi_\ast\mu$ and $h_{\rm top}(F,\Upsilon_n)$ converges to $h(\pi_\ast\mu)=h(\mu)$.
\end{mtheo}

Theorem \ref{teo:accum} will be an almost immediate consequence of Theorem \ref{teopro:intermediate}. Indeed, it only remains to show the convergence of the measures $\mu_n$ to $\mu$. For that we require Theorem \ref{teo:1} (invoking additionally (H2+)).

Finally, returning to the structure of the space of ergodic measures, in correspondence to Theorem \ref{teo:poulsen}, we state how the entropy-dense Poulsen structure of $\cM(\Sigma)$ lifts to $\cM(\Gamma)$. For that we consider the notation $\cM_{\rm erg,\le0}(\Gamma)=\cM_{\rm erg,<0}(\Gamma)\cup\cM_{\rm erg,0}(\Gamma)$ and $\cM_{\rm erg,\ge0}(\Gamma)=\cM_{\rm erg,0}(\Gamma)\cup\cM_{\rm erg,>0}(\Gamma)$. 

\begin{mcoro}\label{cor:accumoneside}
Assume (H1)--(H2+).
	Any $\mu\in\cM(\Gamma)$ having an ergodic decomposition $\mu=\int\mu'\,d\blambda(\mu')$ with ergodic measures $\mu'\in\cM_{\rm erg,\ge0}(\Gamma)$ is weak$\ast$ accumulated by periodic measures in $\cM_{\rm erg,>0}(\Gamma)$. Moreover, there is a sequence of ergodic measures in $\cM_{\rm erg,>0}(\Gamma)$ which converges weak$\ast$ and in entropy to $\mu$. 
	
	The analogous result holds for measures in $\cM_{\rm erg,\le0}(\Gamma)$ and accumulation and convergence in $\cM_{\rm erg,<0}(\Gamma)$.
\end{mcoro}

Corollary \ref{cor:accumoneside} can be seen as an extended version of \cite[Theorem 2]{BocBonGel:18} where it is assumed that in the ergodic decomposition $\mu=\int\mu'\,d\blambda(\mu')$ almost every measure is in $\cM_{\rm erg,>0}(\Gamma)$. A key point in \cite{BocBonGel:18} is the study of local unstable manifolds using Pesin theory, which here we replace by concavity. This allows us to incorporate  nonhyperbolic ergodic measures into the decomposition.

In general, in comparable settings, convex combinations of ergodic measure of \emph{different type} of hyperbolicity may not be approached by ergodic ones (weak$\ast$ and in entropy). For example, this applies to the pair of the two measures of maximal entropy of different type of hyperbolicity in the case of proximality in \cite[Corollary 3]{DiaGelRam:19}. Even though, \cite{DiaGelRam:19} presents many similarities with this paper,  it relies on some essential tools which are not available here: a) the space of admissible sequences is $\Sigma_2$, b) the fibers are circles, and c) synchronisation techniques using that the measures of maximal entropy project to a Bernoulli measure on $\Sigma_2$. 

To complete this section, we point out some further properties that immediately follow from the above results (see for instance the methods in \cite{GorPes:17,DiaGelRam:17b} that apply here \emph{ipsis litteris}), hence stated without proof.

\begin{mcoro}[Ergodic arcwise connectedness]\label{cor:arcwiseconn}
	Assume (H1)--(H2+).
	Then each of the sets $\cM_{\rm erg,<0}(\Gamma)$ and $\cM_{\rm erg,>0}(\Gamma)$ is arcwise connected.
	Moreover, $\cM_{\rm erg,0}(\Gamma)$ is nonempty if and only if the space $\cM_{\rm erg}(\Gamma)$ is arcwise connected.
\end{mcoro}

%-------------------------------------------------------------------------------------------------------
\subsection{Bifurcation settings and homoclinic scenarios}
%-------------------------------------------------------------------------------------------------------

Returning to the idea that the map $F\colon\Gamma\to\Gamma$ may serve as a plug in a semi-local analysis of the dynamics, we see in Section \ref{sec:bif} how this plug acts and interacts with other pieces of dynamics of the global map $\tilde F$.
We consider a family of maps $\tilde f_0$ and $\tilde f_{1,t}$ (for simplicity we assume that $\tilde f_0$ does not depend on $t$)
satisfying our hypotheses and study the corresponding globally defined 
one-parameter family of
skew-products $\tilde F_t$. 
For each parameter $t$ we define the maximal invariant set $\Gamma^{(t)}$ similarly as
 in \eqref{def:Gamma} and the space of admissible sequences $\Sigma^{(t)}$ as in 
\eqref{eq:defSigma}. 
This family has
two distinguished parameters $t_{\rm h}<t_{\rm c}$, corresponding to a heterodimensional cycle associated to $P$ and $Q$ and to the collision of a pair of homoclinic classes, respectively. When $t$ varies from $t_{\mathrm{h}}$ to
$t_{\mathrm{c}}$ the maps $\tilde F_t$ ``unfolds completely''   the heterodimensional cycle, in the sense that the topological entropy of $\tilde F_t |_{\Gamma^{(t)}}$ 
goes from zero at $t_{\mathrm{h}}$ to full $\log 2$ entropy (the maximal possible entropy) at $t_{\mathrm{c}}$ in a nondecreasing way.

We prove that the space of admissible sequences $\Sigma^{(t)}$
converges  to $\Sigma_2$ as $t\to t_{\mathrm{c}}$ in Hausdorff distance. To each family $(\tilde F_t)_{t\in [t_{\mathrm{h}}, t_{\mathrm{c}}]}$ there is naturally associated a constant $C(t)$ relating the derivatives of the fiber maps at $1$, see \eqref{eq:entjumpCt}. If $C(t_{\mathrm{c}})<\infty$ then there is an  explosion of the set $\Sigma^{(t)}$  at the collision parameter $t=t_{\mathrm{c}}$. Moreover, if $C(t_{\mathrm{c}})<1$ then there is also an explosion of the space of invariant measures on $\Sigma^{(t)}$ and a jump of the topological entropy of $\Sigma^{(t)}$ (and hence of $\Gamma^{(t)}$), see Propositions~\ref{prop:explosionsequences} and \ref{pro:entropyjump}. We provide an interpretation for those explosions and examples illustrating the different dynamical scenarios that may occur. We also discuss the ``twinning" and ``merging" of hyperbolic and nonhyperbolic ergodic measures corresponding to Theorem \ref{teo:1} for $\tilde F_{t_{\rm c}}$, see Proposition \ref{p.FSTmeasures}.

Regarding the dynamics at the collision parameter $t=t_{\rm c}$, we recall that collisions of homoclinic classes were studied from the merely topological point of view in \cite{DiaSan:04,DiaRoc:07}. 
We also observe an IFS somewhat similar to the one associated to the maps $\tilde f_0, \tilde f_{1,t_{\mathrm{c}}}$
considered here also appears in \cite{FanSimTot:06,AlsMis:14,AlsMis:15}, where different questions were being asked.

In Section \ref{sec:homscen} we discuss the role of our concavity hypothesis and observe the possible appearance of further homoclinic classes when this hypothesis fails. 

%-------------------------------------------------------------------------------------------------------
\subsection{Organization of the paper}	
%-------------------------------------------------------------------------------------------------------

Sections~\ref{sec:undsymspa} and ~\ref{ss.homoclinicclasses} deal with two underlying key ingredients: the symbolic space $\Sigma$ and its associated IFS and homoclinic relations and classes, respectively. Theorem~\ref{teo:coded} is proved in Section~\ref{sec:codedsystems} which is dedicated to the coded nature of $\Sigma$. In Section~\ref{sec:concavemaps} we prove some auxiliary results about concave maps. In Section \ref{sec:proofofteo:homoclinicclasses} we prove Theorem~\ref{teo:homoclinicclasses} dealing with the decomposition of the non-wandering set into homoclinic classes. Section~\ref{sec:splitt} is dedicated to the study of ``decompositions"  of the space of measures. We prove Theorem \ref{teo:1} about the structure of the space of measures in Section \ref{sec:proofTheorem1} and  Theorem~\ref{teo:2} about disintegration of measures  in Section \ref{sec:proof:teo2}. Theorem~\ref{teo:accum} about approximation of nonhyperbolic ergodic measures by hyperbolic ones is proved in Section~\ref{sec:accum}. We will prove Corollary \ref{cor:accumoneside} at the end of Section \ref{sec:weakstarentropy}. Theorem \ref{teo:poulsen} about the Poulsen structure of $\cM(\Sigma)$ is proved in Section~\ref{sec:teo:poulsen}.
In Section~\ref{sec:bif} we explore bifurcation scenarios. In Section \ref{sec:homscen} we discuss homoclinic classes and the importance of the concavity hypothesis. The paper closes with Appendix~\ref{App:A} about the Wasserstein distance.

%------------------------------------------------------------------------------------------------------
\section{Underlying structures: symbolic space and the IFS}\label{sec:undsymspa}
%------------------------------------------------------------------------------------------------------

In this section we assume (H1). In Section \ref{sec:admcom} hypothesis (H2) is not required and we will additionally assume (H2) only in Section \ref{sec:3per}. Hypothesis (H2+) is not required. 

In the following, we consider the shift space $\Sigma_2\eqdef\{0,1\}^\bZ$ equipped with the metric 
\begin{equation}\label{eqdef:distanceSigma}
	d_1(\xi,\eta)\eqdef e^{-n(\xi,\eta)}, 
	\quad\text{ where }\quad
	n(\xi,\eta)\eqdef\sup\{\lvert\ell\rvert\colon \xi_i=\eta_i\text{ for }i=-\ell,\ldots,\ell\}.
\end{equation}
We use the notation $\xi=(\xi_i)_{i\in \bZ}=(\xi^-.\xi^+)\in\Sigma_2$, where $\xi^+=(\xi_0\xi_1\ldots)\in\Sigma_2^+\eqdef \{0,1\}^{\bN_0}$ and $\xi^-=(\ldots\xi_{-2}\xi_{-1})\in\Sigma_2^-\eqdef\{0,1\}^{-\bN}$. 
We equip $\Sigma_2\times\bR$ with the metric 
\[
	d((\xi,x),(\eta,y))\eqdef\max\{ d_1(\xi,\eta),\lvert x-y\rvert\}. 
\]
Consider the projections 
\[
	\pi \colon \Sigma_2 \times [0,1]\to \Sigma_2, \quad\pi(\xi,x)\eqdef\xi, 
	\quad\text{ and }\quad
	\varrho \colon \Sigma_2 \times [0,1]\to [0,1], \quad\varrho(\xi,x)\eqdef x. 
\]

%------------------------------------------------------------------------------------------------------
\subsection{Admissible compositions}\label{sec:admcom}
%------------------------------------------------------------------------------------------------------

Given $n\ge1$, call $\tau=(\tau_1\ldots\tau_n)\in\{0,1\}^n$ a \emph{word} and  $\lvert\tau\rvert\eqdef n$ its \emph{length}; a \emph{subword} of $\tau$ is a word of the form $(\tau_i\ldots\tau_k)$ with $1\le i\le k\le n$. 
Given words $(\xi_{-m}\ldots\xi_{-1})$ and $(\tau_0\ldots\tau_{n-1})$, we denote the corresponding \emph{cylinders} by
\[\begin{split}
	[\xi_{-m}\ldots\xi_{-1}.]
	&\eqdef \{\eta\colon \eta_k=\xi_k,k=-m,\ldots,-1\},\\
	[\tau_0\ldots\tau_{n-1}]
	&\eqdef \{\eta\colon \eta_k=\tau_k,k=0,\ldots,n-1\}
	.
\end{split}\]

Given a point $x\in [0,1]$, a word $(\xi_0\ldots\xi_{n-1})\in\{0,1\}^n$ is \emph{forward admissible for} $x$ if for every $k=0,\ldots,n-1$ the map 
\[
	f_{[\xi_0\ldots\, \xi_k]}
	\eqdef f_{\xi_k}\circ \cdots \circ f_{\xi_0} 
\]	
is well defined at $x$. We denote by $I_{[\xi_0\ldots \,\xi_{n-1}]}$ the set of points $x\in[0,1]$ for which $(\xi_0\ldots\xi_{n-1})$ is admissible.
Given $\xi\in\Sigma_2$ and $n\ge1$ such that $x\in I_{[\xi_0\ldots\,\xi_{n-1}]}$, sometimes we will also adopt the notation 
\[
	f_\xi^n
	\eqdef f_{[\xi_0\ldots\,\xi_{n-1}]}.
\]
Analogously we adopt the notations $I_{[\xi_{-m}\ldots\, \xi_{-1}.]}$ and
\[
	f^{-m}_\xi
	\eqdef f_{[\xi_{-m}\ldots\, \xi_{-1}.]}
	\eqdef  f_{\xi_{-m}}^{-1} \circ \cdots \circ f_{\xi_{-1}}^{-1}.
\]

\begin{remark}\label{rem:notationnn}
Monotonicity of the maps $f_0,f_1$ implies that each of the sets $I_{[\xi_0\ldots \,\xi_n]}$ and $I_{[\xi_{-m}\ldots\, \xi_{-1}.]}$, when nonempty, is a (possibly degenerate) interval and of the form $[a_{[\xi_0\dots\, \xi_n]},1]$ and $[0,b_{[\xi_{-m}\dots\, \xi_0.]}]$, respectively, where 
\begin{equation}\label{eq:formula}
	f_{[\xi_0\dots\,\xi_n]}(a_{[\xi_0\ldots\,\xi_n]})=0
	\quad\text{ and }\quad
	f_{[\xi_{-m}\dots\,\xi_{-1}.]} (b_{[\xi_{-m}\ldots\,\xi_0.]})=1.
\end{equation}
\end{remark} 

We say that $\xi^+\in\Sigma_2^+$ is \emph{admissible} for $x$ if $(\xi_0\ldots\xi_{n})$ is forward admissible for $x$ for every $n\ge1$, analogously for $\xi^-\in\Sigma_2^-$. We say that a bi-infinite sequence $\xi\in \Sigma_2$ is \emph{admissible} for $x$ if $\xi^+$ and $\xi^-$ are both are admissible for $x$.
Denote by $I_{\xi^+}$ the set of points $x$ such that $\xi^+$ is admissible for $x$, analogously for $I_{\xi^-}$. Note that, given $\xi=(\xi^-.\xi^+)\in\Sigma_2$, the families of intervals $\{I_{ [\xi_0\ldots \,\xi_{n-1}]}\}_{n\ge1}$ and $\{I_{[\xi_{-m}\ldots\, \xi_{-1}.]}\}_{m\ge1}$ both are nested. We have
\[
	\bigcap_{n\ge1}I_{ [\xi_0\ldots \,\xi_{n-1}]} = I_{\xi^+},
	\quad
	\bigcap_{m\ge1}I_{[\xi_{-m}\ldots\, \xi_{-1}.]} = I_{\xi^-},
	\quad\text{ and }\quad
	I_\xi 
	\eqdef I_{\xi^-}\cap I_{\xi^+}.
\]
 By writing $(\xi,x)$ we assume that $\xi^+$ and $\xi^-$ are admissible for $x$, hence $x\in I_\xi$.

\begin{remark}\label{rem:onesidspin}
By the previous comments, we have $I_{\xi^+}=[x_{\xi^+},1]$ and $I_{\xi^-} = [0,x_{\xi^-}]$ for some $x_{\xi^\pm}\in[0,1]$ (provided these intervals are nonempty). Therefore, if $I_\xi\ne\emptyset$ then $I_\xi = [x_{\xi^+},x_{\xi^-}]$ (and $x_{\xi^+}\le x_{\xi^-}$).
\end{remark}

\begin{remark}\label{rem:admprehet}
Any word which is forward admissible for $0$ is of the form $0^k$ and any word which is backward admissible for $1$ is of the form $0^\ell$. 
\end{remark}

\begin{remark}\label{rem:constseqs}
	Note that 
\[\begin{split}
	\xi_n=1 &\quad\text{ if and only if }\quad a_{[\xi_0\dots\, \xi_{n-1}]}<a_{[\xi_0\dots\, \xi_n]},\\
	\xi_{-m}=1 &\quad\text{ if and only if }\quad b_{[\xi_{-m}\ldots\,\xi_0.]}<b_{[\xi_{-m+1}\ldots\,\xi_0.]}.
\end{split}\]	
\end{remark}

Given $x\in[0,1]$ and $n\ge1$, let 
\[
	\Sigma^+(n,x)
	\eqdef \{ (\xi_0\ldots\, \xi_{n-1})\colon \text{admissible for }x\},
\] 
analogously $\Sigma^-(n,x)$. We denote by $\Sigma (x)\subset\Sigma_2$ the  set of (infinite) sequences which are admissible for $x$ and by $\Sigma^+(x)$ and $\Sigma^-(x)$ the corresponding sets of admissible one-sided sequences. Note that the set $\Sigma$ defined in \eqref{eq:defSigma} coincides with the set of all admissible sequences
\[
	\Sigma
	= \bigcup_{x\in [0,1]} \Sigma(x).
\]
With the definition \eqref{def:Gamma}, we have 
\[
	\Gamma
	= \{(\xi,x) \in \Sigma_2 \times [0,1]\colon x\in [0,1]\text{ and }\xi \in \Sigma(x)\}.
\]

\begin{remark}\label{rem:different}
Clearly, it follows from (H1) that $\Sigma^+(x)\subset\Sigma^+(1)$ and $\Sigma^-(x)\subset\Sigma^-(0)$. Hence, it holds
\[
	\Sigma 
	\subset \{(\xi^-.\xi^+) \colon \xi^- \in \Sigma^-(0) \text{ and } \xi^+\in \Sigma^+(1)\},
\]
and this inclusion is in general strict. 
Moreover, for every $x$
\[
	\Sigma(x)
	= \{(\xi^-.\xi^+) \colon \xi^-\in\Sigma^-(x),\xi^+\in\Sigma^+(x)\}.
\]
\end{remark}

\begin{remark}\label{rem:magenta}
	If $\xi=(\xi_0\ldots\xi_{n-1})^\bZ$ is an admissible periodic sequence, then $f_{[\xi_0\ldots\,\xi_{n-1}]}$ has one fixed point with nonpositive Lyapunov exponent and one with nonnegative Lyapunov exponent. Note that these points may coincide and then this point is parabolic.
\end{remark}

\begin{remark}\label{rem:admissperio01seq}
	Given $\xi^+\in\Sigma^+(1)$ such that $\xi_n=1$, for the new sequence obtained inserting one $0$ in the $n$th position, we have $(\xi_0\ldots\xi_{n-1}01\xi_{n+1}\ldots)\in\Sigma^+(1)$. Hence, inductively,  $(\xi_0\ldots\xi_{n-1}0^k1\xi_{n+1}\ldots)\in\Sigma^+(1)$, for every $k\ge1$.
	
	To see why this is so, just note that $f_0$ is defined on $[0,1]$, hence $f_{[\xi_0\ldots \,\xi_{n-1}0]}(1)$ is well defined, and $f_{[\xi_0\ldots\, \xi_{n-1}0\xi_n]}(1)>  f_{[\xi_0\ldots\, \xi_{n-1}\xi_n]}(1)$. The above now follows
from the monotonicity of $f_0$ and $f_1$.

Similarly, given $\xi^+\in\Sigma^+(1)$ such that $\xi_n=1$, the  sequence obtained exchanging $\xi_n=1$ for $\xi_n=0$ and keeping all remaining terms also belongs to $\Sigma^+(1)$.
\end{remark}

Recall that by (H1) the point $d\in(0,1)$ is defined by $f_1(d)=0$.

\begin{remark}\label{rem:someseq}
	For $k\ge1$ sufficiently large, the periodic sequence $(10^k)^\bZ$ belongs to $\Sigma$. Indeed, for $\varepsilon>0$ small and $k\ge1$ sufficiently large, we have $(f_0^k\circ f_1)([d-\varepsilon,1])\subset [d,1]$. Hence, there is $x\in[d,1]$ such that $x=(f_0^k\circ f_1)(x)$ and thus the periodic sequence $\xi=(10^k)^\bZ$ is admissible for $x$.
\end{remark}

\begin{remark}\label{rem:monotonously}
	Combining Remarks \ref{rem:admissperio01seq} and \ref{rem:someseq} and also applying Remark \ref{rem:constseqs}, the following is now immediate. 
	For $k\ge1$ sufficiently large, we have $\xi=(0^k1)^\bZ\in\Sigma$ and
\[
	a_{[0^k1]}
	\le a_{[0^k10^k1\ldots \,0^k1]}
	\le \ldots 
	< x_{(0^k1)^\bN}
	= x_{\xi^+}.
\]	
	On the other hand, for every sequence  of positive integers $(\ell_n)_{n\ge1}$ satisfying $\ell_n\le k$ for every $n\ge1$ we have
\[
	a_{[0^k10^k1\ldots \,0^k1]}
	\le a_{[0^{\ell_1}10^{\ell_2}1\ldots\,0^{\ell_n}1]},
	%I_{[0^{\ell_1}10^{\ell_2}1\ldots \,0^{\ell_n}1]}
	%\supset I_{[0^k10^k1\ldots \,0^k1]}
\]	
whenever $I_{[0^{\ell_1}10^{\ell_2}1\ldots\,0^{\ell_n}1]}\ne\emptyset$.
\end{remark}

\begin{remark}\label{rem:numones}[Consecutive $1$'s in admissible sequences]
	Define $k_0\ge1$ to be the integer such that $f_{1}^{k_0-1}(1) \in [d,1]$ and $f_{1}^{k_0}(1) < d$. Note that this number is well defined since $f_1(x)<x$ for all $x\in [d,1]$.
	
		Consider a word $(\xi_0\dots \xi_n)$ which is forward admissible for $x$. The definition of $k_0$ implies that it has at most $k_0$ consecutive $1$'s. Moreover, $(\xi_0\dots \xi_n\,0)$ is also forward admissible for $x$ (just observe that the domain of $f_0$ is the whole interval $[0,1]$). 
\end{remark}

The \emph{forward orbit} of a point $x\in[0,1]$ by the IFS is defined as
\[
	\cO^+(x)
	\eqdef \bigcup_{n\ge0}\bigcup_{(\xi_0\ldots\xi_{n-1})\in\Sigma^+(n,x)}
		f_{[\xi_0\ldots\,\xi_{n-1}]}(x).
\]
The \emph{backward orbit} of $x$, $\cO^-(x)$, is defined analogously. The \emph{orbit} of $x$ by the IFS is
\[
	\cO(x)
	\eqdef \bigcup_{\xi\in\Sigma(x)}\Big(\bigcup_{n\ge0}f_{[\xi_0\ldots\,\xi_{n-1}]}(x)
							\cup\bigcup_{m\ge1}f_{[\xi_{-m}\ldots\,\xi_{-1}.]}(x)
							\Big).
\]	
In view of Remark \ref{rem:different},
\[
	\cO(x)
	= \cO^-(x)\cup\cO^+(x).
\]

\begin{lemma}
The set $\cO^-(0)$ is dense in $[0,1]$ if and only if for every $x,y\in [0,1]$, $x\neq y$, we have  $\Sigma^+(x)\ne\Sigma^+(y)$. Similarly, the set $\cO^+(1)$ is dense in $[0,1]$ if and only if for every $x,y\in [0,1], x\neq y$, we have $\Sigma^-(x)\ne\Sigma^-(y)$.
\end{lemma}

\begin{proof}
Note first that for any two points $x,y\in [0,1], x<y$, the set $\Sigma^+(x)$ differs from $\Sigma^+(y)$ if and only if $\Sigma^+(n,x)\neq \Sigma^+(n,y)$ for some $n$. Assume $\Sigma^+(n,x)\neq \Sigma^+(n,y)$ for some $n$ and assume that $n$ is the smallest one. Then there exists a word $(\xi_0\ldots \xi_{n-1})$ forward admissible for $y$ but not for $x$, while the word $(\xi_0\ldots\xi_{n-2})$ is forward admissible for both $x$ and $y$. Hence, by Remark \ref{rem:constseqs}, we have $\xi_{n-1}=1$  and $f_{[\xi_0\ldots\,\xi_{n}]}(x) <d \leq f_{[\xi_0\ldots\,\xi_{n}]}(y)$. Thus, $f_{[\xi_0 \dots \xi_n.]}(0) \in (x,y)$.

In the other direction, if $\cO^-(0)\in (x,y)$, similarly there is $(\xi_0\ldots \xi_{n-1})\in \Sigma^-(n,0)$ such that
  $x<f_{[\xi_{n-1}\ldots \xi_{0}.]}(0)<y$.
This implies that $(\xi_0\ldots\xi_{n-1}) \in \Sigma^+(n,y) \setminus \Sigma^+(n,x)$.

The second part of the lemma for forward orbits is obtained in an analogous way.
\end{proof}

%------------------------------------------------------------------------------------------------------
\subsection{Hyperbolic and parabolic periodic points}\label{sec:3per}
%------------------------------------------------------------------------------------------------------

In this section, we will assume (H1)--(H2).

\begin{lemma}\label{lem:fixpots}
	Let $(\xi_0\dots \xi_n)$ be a word  such that $I_{[\xi_0\ldots\,\xi_{n-1}]}\ne\emptyset$ and $g=f_{[\xi_0\ldots\,\xi_{n-1}]}$. There are the following possibilities:
\begin{itemize}
\item[(1)] If $g$ has some fixed point then $\xi=(\xi_0\ldots \xi_{n-1})^\bZ\in \Sigma$ and there are two cases:
\begin{itemize}
\item[(1a)] $g$ has exactly two fixed points $p_{[\xi_0\ldots\,\xi_{n-1}]}^+<p_{[\xi_0\ldots\,\xi_{n-1}]}^-$ and they are repelling and contracting, respectively. In this case, $I_\xi=[p_{[\xi_0\ldots\,\xi_{n-1}]}^+,p_{[\xi_0\ldots\,\xi_{n-1}]}^-]$. 
\item[(1b)] $g$ has exactly one fixed point $p_{[\xi_0\ldots\,\xi_{n-1}]}$ and it is parabolic. In this case, $I_\xi=\{p_{[\xi_0\ldots\,\xi_{n-1}]}\}$. 
\end{itemize}
\item[(2)] If $g$ has no fixed point then $(\xi_0\ldots \xi_{n-1})^\bZ\not\in \Sigma$.
\end{itemize}
\end{lemma}

\begin{proof}
	Let $p$ be a fixed point of $g$. Then $((\xi_0\ldots \xi_{n-1})^\bZ,p)\in \Gamma$ and the first assertion follows.
	By monotonicity and $f_1(x)<x$ (hypothesis (H1)) it follows that $(\xi_0\ldots\,\xi_{n-1})$ contains at least one $0$ and hence, since $f_0'$ is strictly decreasing and $f_1'$ is nonincreasing (hypothesis (H2)), $g$ has strictly decreasing derivative. This immediately implies the two possibilities (1a) and (1b) claimed in the lemma.
	
	Case (2) is an immediate consequence of the monotonicity of the maps $f_0,f_1$ and of the fact that the graph of $g$ is below the diagonal. 
\end{proof}

%------------------------------------------------------------------------------------------------------
\section{Coded systems}\label{sec:codedsystems}
%------------------------------------------------------------------------------------------------------

In this section we only assume (H1), hypotheses (H2)--(H2+) are not required. 

The goal of this section is to prove Theorem \ref{teo:coded}.

Let us first recall some standard definitions, see for example \cite{LinMar:95} for details. 
We only consider two-sided sequence spaces.
Given a subset $S\subset\Sigma_2$, define 
\[
	\eW_n(S)\eqdef \{\tau
		\colon \lvert\tau\rvert=n,\tau=(\xi_{k+1}\ldots\xi_{k+n}) 
				\text{ for some }\xi\in S\text{ and some }k\in\bZ\}
\]				 
the set of all \emph{allowed words of length $n$ in $S$} and let 
\[
	\eW(S)\eqdef\bigcup_{n\ge0}\eW_n(S),
\]	 
where $\eW_0(S)=\emptyset$ by convention. 
A \emph{subshift} is a $\sigma$-invariant set in $\Sigma_2$.
A subset $S\subset\Sigma_2$ is a \emph{subshift of finite type} (\emph{SFT}) if it is specified by finitely many ``forbidden" words, all of finite length, that is, if there exists a finite family $\cF\subset\eW(\Sigma_2)$ so that
\[
	S
	=\Sigma_\cF
	\eqdef\{\xi\in\Sigma_2
		\colon \eW(\{\xi\})\cap\cF=\emptyset\}.
\]
Equivalently, there exist $n\ge1$ and a finite family of words of equal length $n$, $\cF'\subset\eW_n(\Sigma_2)$, such that $S=\Sigma_{\cF'}$. It follows that any SFT is $\sigma$-invariant and that
\[		
	\Sigma_\cF
	=\{\xi\in\Sigma_2
		\colon (\xi_{k+1}\ldots \xi_{k+n})\not\in\cF
		\text{ for all }k\in\bZ\}.
\]

Let us introduce the concept of coded systems, though we will skip the original definition (see, for example, \cite[Chapter 13.5]{LinMar:95} and references therein) and instead use the characterization by Krieger in  \cite{Kri:00}. By \cite{Kri:00}, a transitive subshift $S\subset\Sigma_2$ is \emph{coded} if and only if there is an increasing family of irreducible SFTs whose union is dense in $S$.

Recall the definition of the compact and $\sigma$-invariant set $\Sigma\subset\Sigma_2$ in \eqref{eq:defSigma}. Consider the sets
\[
	\eW^{\rm het}(\Sigma)
	\eqdef \{\tau\colon \tau\in\eW(\Sigma),f_{[\tau]}(1)=0\},
	\quad
	\eW^0(\Sigma)
	\eqdef \{0^k\colon k\ge1\},
\]
and let 
\[	
	\eW^{\rm cod}(\Sigma)
	\eqdef \eW(\Sigma)\setminus (\eW^{\rm het}(\Sigma)\cup \eW^0(\Sigma)).
\]
Note that with this notation, the set $\Sigma^{\rm het}$ defined in \eqref{eq:xiPQ} is precisely
\[
	\Sigma^{\rm het}
	= \{0^{-\bN}\tau \,0^\bN\colon \tau\in\eW^{\rm het}(\Sigma)\}
	= \{\sigma^k(\xi)\colon \xi=(0^{-\bN}.\tau \,0^\bN),\tau\in\eW^{\rm het}(\Sigma),k\in\bZ\}.
\]

Let us now prepare the proof of Theorem \ref{teo:coded}. Recall notations in Section \ref{sec:undsymspa}.

\begin{proposition}\label{proclalem:222}
	For any two disjoint SFTs $S_i\subset\Sigma$, $i=1,2$, not containing $0^\bZ$, there exists a transitive SFT $S_3\subset\Sigma$ such that $S_1\cup S_2\subset S_3$ and $0^\bZ\not\in S_3$. 
\end{proposition}

\begin{proof}
	As by assumption we have $0^\bZ\not\in (S_1\cup S_2)$ and since the sets $S_i$ are compact, there is $N_0\ge1$ such that $[0^{N_0}]\cap (S_1\cup S_2)=\emptyset$. By Remark \ref{rem:someseq}, without loss of generality, we can assume that $N_0$ also satisfies that $(0^{N_0-1}1)^\bZ\in\Sigma$. 

By the choice of $N_0$, $0^{N_0}$ is a ``forbidden word" in $S_i$, $i=1,2$, and, in particular, every sequence in $S_i$ must be of the type $\xi=\xi^-.\xi^+$, with $\xi^+=(0^{\ell_1}1^{m_1}0^{\ell_2}1^{m_2}\ldots)$ such that $\ell_k\le N_0-1$ and $1\le m_k\le N$ for some $N\ge1$  for all $k\in\bZ$ (recall Remark \ref{rem:numones}). 
Considering the points in \eqref{eq:formula} and letting $a_k=a_{[0^{\ell_1}1^{m_1}\ldots \,0^{\ell_n}1^{m_k}]}$, we get a nested sequence of intervals $[a_k,1]$ such that $a_k\le x_{\xi^+}$ and $a_k$ monotonically converges to $x_{\xi^+}$ as $k\to\infty$. Note that for every $k\ge1$ we have
\[
	a_{[0^{\ell_1}10^{\ell_2}1\ldots \,0^{\ell_k}1]}
	\le a_{[0^{\ell_1}1^{m_1}0^{\ell_2}1^{m_2}\ldots \,0^{\ell_k}1^{m_k}]}.
\]
Moreover, by Remark \ref{rem:monotonously}, we have
\[
	a_{[0^{N_0-1}1\ldots \,0^{N_0-1}1]}
	\le a_{[0^{\ell_1}1\ldots \,0^{\ell_k}1]}.
\]
This allows us to conclude that  for every $\xi\in S_i$, $i=1,2$, we have 
\[
	0
	< a
	\eqdef x_{(0^{N_0-1}1)^\bN}
	\le x_{\xi^+}
\]
The argument for $x_{\xi^-}$ is analogous, and we let $b
	\eqdef x_{(0^{N_0-1}1)^{-\bN}}$, where $(0^{N_0-1}1)^{-\bN}=(\ldots 0^{N_0-1}10^{N_0-1}1)$. Hence, for $i=1,2$  
\begin{equation}\label{eq:admdefb}
	0
	< a
	= x_{(0^{N_0-1}1)^\bN}
	\le 	\min_{\xi\in S_i}x_{\xi^+}
	\le 
	\max_{\xi\in S_i}x_{\xi^-}
	\le 
	x_{(0^{N_0-1}1)^{-\bN}}
	= b
	< 1.
\end{equation}
 
\begin{claim}\label{cla:eq:nowgo}
Every $\tau\in\eW(S_1)\cup\eW(S_2)$ is forward admissible for $b$ and backward admissible for $a$ and satisfies
\[
	0
	< a
	\le f_{[\tau]}(b)
	< 1.
\]
\end{claim}

\begin{proof}
By the definition of admissibility, given $\xi\in S_i$, $i=1,2$, we have that $\xi$ is admissible for $x$ if and only if $x_{\xi^+}\le x\le x_{\xi^-}$. Hence, by the above, $a\le x_{\xi^+}\le x\le x_{\xi^-}\le b$ and, in particular, for every $n\in \bZ$ we have 
\[
	a
	\le f_\xi^n(x)
	\le b.
\]  
Therefore, every $\tau\in\eW(S_i)$ is forward admissible for $b$.
The proof of backward admissibility is analogous.
\end{proof}	

The naive idea for the construction of the SFT $S_3$ is to choose some appropriate $N_1$ and to consider all concatenations of words of lengths at least $N_1$ which come from the subshifts $S_i$, $i=1,2$, which are separated by words $0^{N_1}$. The precise definition is $S_3=\Sigma_\cF$, where a $\cF$ is a certain finite set of forbidden words of length $2N_1$. Instead of describing $\cF$, we define its complement in the set $\eW_{2N_1}(\Sigma_2)$ of allowed words of length $2N_1$.

We determine $N_1> 2N_0$ as follows. Let $k_i\ge1$ be such that $S_i$ is a SFT generated by a family of words of length $k_i$, $i=1,2$. Without loss of generality we can assume that $k_1=k_2\ge N_0$. Moreover, we can also assume that any pair of cylinders of length $N_1$ in $S_i$ for $i=1,2$, respectively, are disjoint. 
Now, as $0<a\le b<1$, we can choose $N_1\ge k_1(=k_2)$ such that
\[
	f_{[0^{N_1}]}(a)
	= f_0^{N_1}(a)
	> b.
\]
Hence, for every $x\in[a,b]$ we also have 
\begin{equation}\label{eq:admisss1}
	f_0^{-N_1}(x)
	< 
	a
	\le b
	< f_0^{N_1}(x).
\end{equation}

The complement of the set $\cF$ in the set $\eW_{2N_1}(\Sigma_2)$ is defined as follows: 
\begin{itemize}
\item[i)] any word of length $2N_1$ allowed either in $S_1$ or in $S_2$,
\item[ii)] any word $0^\ell v$, with $\ell\in\{1,\ldots,N_1\}$, $\lvert v\rvert = 2N_1-\ell$, and $v\in\eW(S_1)\cup\eW(S_2)$,
\item[iii)] any word $v 0^\ell$, with $\ell\in\{1,\ldots,N_1\}$, $\lvert v\rvert = 2N_1-\ell$, and $v\in\eW(S_1)\cup\eW(S_2)$,
\item[iv)] any word of the form $v 0^{N_1} w$, with $\lvert v\rvert\ge1$, $\lvert w\rvert\ge1$, $\lvert v\rvert+\lvert w\rvert=N_1$, and $v,w\in\eW(S_1)\cup\eW(S_2)$.
\end{itemize}

\begin{claim}
	We have $0^\bZ\not\in S_3=\emptyset$ and $S_1\cup S_2\subset S_3$.
\end{claim}

\begin{proof}
	By items ii)--iii), $0^{2N_1}$ is not allowed, as by the above $0^{N_1}$ is not allowed neither in $S_1$ nor in $S_2$. It is also not allowed by i) nor by iv) (the latter -- because one of the words $v,w$ has length at least $N_1/2>N_0$).
Hence, we have $[0^{2N_1}]\cap S_3=\emptyset$, getting the first claim.	
	By item i), we immediately get $\eW_{2N_1}(S_i)\subset\eW_{2N_1}(S_3)$, $i=1,2$, and hence $S_1\cup S_2\subset S_3$. 
\end{proof}
	
\begin{claim}
	$S_3$ is transitive.
\end{claim}	

\begin{proof}
	It is enough to check that for every pair of words $v,w\in\eW(S_3)$ there exists a  word $\eta\in\eW(\Sigma_2)$ such that $v\eta w\in\eW(S_3)$. Without loss of generality, we can assume $\lvert v\rvert,\lvert w\rvert>N_1$. There are three possible cases:
\begin{itemize}
\item[(1)] if $v$ ends with $1$ and $w$ begins with $1$, then take $\eta=0^{N_1}$,
\item[(2)] if $v$ ends with $1$ and $w$ begins with $0^\ell$ for some $\ell\in\{1,\ldots,N_1-1\}$, $\ell$ being maximal with this property, then take $\eta= 0^{N_1-\ell}$, analogously for the reversed case,
\item[(3)] if $v$ ends with $0^\ell$ and $w$ begins with $0^m$ for some $\ell,m\in\{1,\ldots,N_1-1\}$, $\ell$ and $m$ being maximal with these properties,
\begin{itemize}
\item if $\ell+m< N_1$, then take $\eta=0^{N_1-\ell-m}$,
\item if $\ell+m\ge N_1$, then take $\eta=\emptyset$.
\end{itemize}
\end{itemize}

Let us see that indeed in case (1) the word $v\eta w$ is allowed in $S_3$. 
We can write $v=v''v'$ and $w=w'w''$ where $\lvert v'\rvert+\lvert w'\rvert=N_1$. Then it is enough to apply item (iv).
Cases (2) and (3) are analogous. 
This proves the claim.		
\end{proof}
	
It remains to prove $S_3\subset\Sigma$, which is an immediate consequence of the following claim.
	
\begin{claim}
	Every $\xi\in S_3$ is admissible for some point in $(0,1)$.
\end{claim}	

\begin{proof}
Note that by the above definition of $S_3$, for $\xi\in S_3$ we have either $\xi\in S_1$ or $\xi\in S_2$ or $\xi=(\ldots \tau_10^k\tau_20^k\ldots)$ with $\tau_n$ being subwords allowed in either $S_1$ or in $S_2$. Without loss of generality, it is enough to assume that $\xi=(\ldots\tau_{-1}0^k.\tau_00^k\tau_10^k\ldots)$ and to show that $\xi$ is admissible for $b$. 

We start by checking $\xi^+$ is admissible for $b$.
 By Claim \ref{cla:eq:nowgo} the word $\tau_0$ is forward admissible for $b$ and we have $a\le f_{[\tau_0]}(b)<1$. 
 Clearly, $0^k$ is forward admissible for $f_{[\tau_0]}(b)$ and by \eqref{eq:admisss1} we have $f_{[\tau_00^k]}(b)>b$. As by Claim \ref{cla:eq:nowgo} the word $\tau_1$ is forward admissible for $b$ and hence for $f_{[\tau_00^k]}(b)$, we have that $\tau_00^k\tau_1$ is forward admissible for $b$ defined in \eqref{eq:admdefb}.
Now we proceed by induction to show that $\xi^+$ is admissible for $b$. 

To check backward admissibility, first recall that by \eqref{eq:admisss1} we have
 $f_0^{-k}(b)<a$. As by Claim \ref{cla:eq:nowgo} the word $\tau_{-1}$ is backward admissible for $a$, we have that $\tau_{-1}$ is also  backward admissible for $f_0^{-k}(b)$. We now argue inductively as before. 
\end{proof}

The proof of the proposition is now complete.	
\end{proof}

\begin{proof}[Proof of Theorem \ref{teo:coded}]
We start by analyzing the 	``heteroclinic part" $\Sigma^{\rm het}$ of $\Sigma$.

\begin{lemma}\label{lemcla:isola}
	Every $\xi\in\Sigma^{\rm het}$ is an isolated point in $\Sigma$.
\end{lemma}

\begin{proof}
It suffices to show that every $\tau\in\eW^{\rm het}(\Sigma)$ we have $[\tau]\cap\Sigma=\{0^{-\bN}.\tau 0^\bN\}$ and, in particular, $\tau$ has a unique continuation to a bi-infinite admissible sequence. By the choice of $\tau$, $f_{[\tau]}(1)=0$ and therefore  is only forward admissible at $1$ and hence $\tau$ can only  be continued to a backward admissible sequence by $0^{-\bN}$. Analogously, again by  $f_{[\tau]}(1)=0$ and also Remark \ref{rem:admprehet}, $\tau$ can only be continued to a forward admissible sequence  by $0^\bN$. This proves the lemma.
\end{proof}

As  every point in $\Sigma^{\rm het}$ is isolated and non-periodic, it is wandering. In particular, there is no invariant measure supported on this set. Further,  $\Sigma^{\rm het}$ is countable and hence its topological entropy is zero. This proves the claimed properties of $\Sigma^{\rm het}$  in the theorem.

The facts that $\Sigma^{\rm cod}$ is compact and $\sigma$-invariant follow from the above derived properties of $\Sigma^{\rm het}$. What remains to show is that it is coded. We start by the following lemma.

\begin{lemma}\label{lem:111}
	For every SFT $S\subset\Sigma$ and every word $\tau$ satisfying $[\tau]\cap S=\emptyset$, $\tau\ne(0\ldots0)$, and $f_{[\tau]}(1)\in(0,1)$, there is a(n infinite) SFT $S'$ satisfying $S'\cap S=\emptyset$, $0^\bZ\not\in S' $, and $[\tau]\cap S'\ne\emptyset$.
\end{lemma}

\begin{proof}
	Since $f_{[\tau]}(1)\in(0,1)$, there is $a\in(0,1)$ for which $\tau$ is forward admissible. Let $b=f_{[\tau]}(a)$. We can also assume that $b>0$. Choose $k\ge1$ such that $f_0^k(b)>a$. Now it is enough to consider the SFT $S'$ generated by the family of words $\{\tau 0^k,\tau 0^{k+1}\}$. 
\end{proof}

We now inductively construct an increasing countable family $\{S_k\}$ of transitive SFTs such that $\overline{\bigcup_k S_k}=\Sigma^{\rm cod}$.   
First observe that the set $\eW^{\rm cod}(\Sigma)$ is countable and let $\{\tau^{(k)}\}$ be some enumeration of it. 
Now let $S_0=\emptyset$ and for $k=1,2,\ldots$ apply the following iterative procedure:
\begin{itemize}
\item If $[\tau^{(k)}]\cap S_{k-1}\ne\emptyset$ then $S_k\eqdef S_{k-1}$,
\item otherwise, if $[\tau^{(k)}]\cap S_{k-1}=\emptyset$, then
\begin{itemize}
\item first apply Lemma \ref{lem:111} to $S_{k-1}$ to obtain a SFT $S_k'$ such that 
\[
	S_k'\cap S_{k-1}=\emptyset, 
	\quad 0^\bZ\not\in S_k',
	\quad \text{ and }\quad
	[\tau^{(k)}]\cap S_k'\ne\emptyset,
\]
\item thereafter apply Proposition \ref{proclalem:222} to $S_{k-1}$ and $S_k'$ to obtain a transitive SFT $S_k$ such that 
\[
	S_k\supset S_{k-1}\cup S_k'
	\quad\text{ and }\quad 
	0^\bZ\not\in S_k.
\]	 
\end{itemize}
\end{itemize}
Observe that we have 
\[
	S_k\supset S_{k-1},\quad
	0^\bZ\not\in S_k,
	\quad \text{ and }\quad
	S_k\cap[\tau^{(k)}]\ne\emptyset.
\]
This provides an increasing family of transitive SFTs $\{S_k\}_k$ satisfying
\[
	\overline{\bigcup_{k\ge1}S_k}
	\supset\overline{\Sigma^{\rm cod}\setminus\{0^\bZ\}}
	=\Sigma^{\rm cod}.
\]
Moreover, since points from $\Sigma^{\rm het}$ have arbitrarily long subsequences of zeros they do not belong to any $S_k$. Moreover, as by Lemma \ref{lemcla:isola} the points in $\Sigma^{\rm het}$ are isolated, they cannot belong to the closure of $\bigcup S_k$. Thus, $\overline{\bigcup_k S_k}=\Sigma^{\rm cod}$.

Finally, to see that $\sigma$ is topologically mixing on $\Sigma^{\rm cod}$, consider two forward admissible words $(\xi_0\ldots\xi_n)$ and $(\eta_0\ldots\eta_m)$ and points $x$ and $y$ for which these words are admissible, respectively. Note that for every $k\ge1$ sufficiently large it holds $f_{[\xi_0\ldots\,\xi_n0^k]}(x)>y$. Hence, the composed word $( \xi_0\ldots\xi_n0^k\eta_0\ldots\eta_m)$ is admissible at $y$. This immediately implies the mixing property.

This completes the proof of Theorem \ref{teo:coded}. 
\end{proof}

%------------------------------------------------------------------------------------------------------
\section{Underlying structures: Homoclinic classes} \label{ss.homoclinicclasses}
%------------------------------------------------------------------------------------------------------

In this entire section we only assume (H1)--(H2), hypothesis (H2+) is not required. 

We establish the notion of a \emph{homoclinic class} of a hyperbolic periodic point of the skew-product $F$ induced by the map $\tilde F$  defined in \eqref{eq:parental}, translating it from the differentiable setting. The analogous analysis can be done for $\tilde F$ but will be skipped. We see that there are only two classes: one containing contracting orbits and the other one containing expanding ones, see Propositions \ref{pro.l.homoclinicallyrelated}. These classes may intersect. Moreover, the ``boundary of the set $\Gamma$" has two graph-like parts: one contained in $H(P, F)$ and the other one in $H(Q, F)$, see Proposition \ref{prol.extremalpoints}. We also introduce homoclinic relations for parabolic periodic points, see Section \ref{ss:parabolic}, and see that they are related simultaneously to periodic points of both types of hyperbolicity and, in particular, to $P$ and $Q$, see Proposition \ref{pl.parabolichomoclinic}. This section only discusses the topological structure of homoclinic classes. The study of their hyperbolic and ergodic properties is postponed. 

%------------------------------------------------------------------------------------------------------
\subsection{Homoclinic relations and classes}
%------------------------------------------------------------------------------------------------------

 In the differentiable setting, the homoclinic class of a hyperbolic periodic point is the closure of the transverse intersections of the stable and the unstable invariant manifolds of its orbit. These homoclinic classes are transitive sets with a dense subset of periodic orbits. In our setting, the definition of a homoclinic class is similar, the only difference is that transversality is not involved (note that here we  can only speak of invariant sets and cannot invoque any differentiable structure for these sets). We will follow closely the presentation in \cite[Sections 2 and 3]{DiaEstRoc:16} where a similar discussion is done for an specific class of skew-product maps (falling in the concave class studied here) and skip some details, see this reference for further details.

In what follows, we denote by $\cO(X)$ the $ F$-orbit of a point $X$. Consider a periodic point $R=((\xi_0\ldots\xi_{n-1})^\bZ, r)$ of $ F$, note that $f_{[\xi_0\ldots\,\xi_{n-1}]}(r)=r$. Recall that its orbit is \emph{hyperbolic} if
$f'_{[\xi_0\ldots\,\xi_{n-1}]}(r) \ne 1$ (by hypothesis, this derivative is
 positive). This orbit is {\emph{contracting}} if  $f'_{[\xi_0\ldots\xi_{n-1}]}(r)\in (0,1)$, otherwise it is \emph{expanding}. When the derivative is equal to one the orbit is called \emph{parabolic}.
We define the \emph{stable set of $R$}, $\cW^\s (R, F)$, as the set of points  $X$ such that $ F^{i}(X) 
\to R$ as $i\to \infty$. 
The \emph{stable set of the orbit of $R$}, $\cW^\s(\cO(R), F)$, is the union of the stable sets
of the points in $\cO(R)$. The \emph{unstable sets of $R$ and $\cO(R)$} are defined 
by $\cW^\u (R, F)=\cW^\s(R, F^{-1})$ and 
$\cW^\u (\cO(R), F)=\cW^\s(\cO(R), F^{-1})$. 

Given now a hyperbolic periodic point $R$ and its orbit $\cO(R)\subset \Gamma$
we consider its \emph{homoclinic points} 
\[
	X
	\in \cW^\s(\cO(R), F)\cap\cW^\u(\cO(R), F)
\]
and its \emph{homoclinic class}
\[
	H(R, F)
	\eqdef \overline{\{\cW^\s(\cO(R), F)\cap\cW^\u(\cO(R), F)\}}.	
\]
Note again that the ``transversality" of the homoclinic intersections is not required. 
Noting that $H(R, F)$ is $ F$-invariant and that $\Gamma$ is a locally maximal invariant set
it follows that 
\[
	H(R, F)\subset \Gamma, 
	\quad\text{ for every hyperbolic periodic point }R\text{ with }\cO(R)\subset\Gamma.
\]
The homoclinic class $H(R, F)$ can be alternatively defined as follows.
First, we say that a pair of hyperbolic periodic points $R_1$ and $R_2$ of the same type of hyperbolicity are {\emph{homoclinically related}} if the un-/stable invariant sets of their orbits  intersect cyclically, 
\[
\cW^\s(\cO(R_1), F)\cap\cW^\u(\cO(R_2), F)\ne\emptyset
\ne 
\cW^\u(\cO(R_1), F)\cap\cW^\s(\cO(R_2), F).
\]
Note that transversality is not required, but it is required that the two orbits have the same type or hyperbolicity%
\footnote{Indeed, we may have periodic orbits with different type of hyperbolicity whose invariant sets intersect cyclically, this leads to a \emph{heterodimensional cycle} involving these orbits, see Remark \ref{rem:herocycle}}. 
It follows that \emph{being homoclinically related} defines an equivalence relation among hyperbolic periodic points of the same type of hyperbolicity. This is due the fact that the fiber dynamics has no critical points and hence the intersections between the invariant sets of $\cO(R_1)$ and $\cO(R_2)$ behave as transverse
ones and have well defined continuations. In Section~\ref{ss:parabolic}, we will extend homoclinic relation to include also parabolic periodic points and will provide the proof that this new relation is an equivalence relation.
Then 
\[
	H(R, F)
	= \mathrm{closure} \{R'\colon R' \text{ is homoclinically related to } R\}.
\]
As in the differentiable setting, $H(R, F)$ is a transitive set. Note that, in general, two homoclinic classes of periodic points may fail to be disjoint (see Remark \ref{rem:examples} item (1c)).
 
We will see that in our setting hyperbolic periodic points of the same type are homoclinically related. Hence there exist only two homoclinic classes (related to $P$ and $Q$, respectively). 

\begin{proposition}\label{pro.l.homoclinicallyrelated} 
	Every hyperbolic periodic point $R\in\Gamma$ of expanding (contracting) type is homoclinically related to $Q$ (to $P$). In particular,  two points $R_1$ and $R_2$ of expanding (contracting) type are homoclinically related and their common homoclinic class coincides with the one of $Q$ (of $P$).
\end{proposition}
 
%------------------------------------------------------------------------------------------------------
\subsection{Proof of Proposition \ref{pro.l.homoclinicallyrelated}}
%------------------------------------------------------------------------------------------------------
 
To continue our discussion, we state a simple lemma about homoclinic relations which is just a reformulation of \cite[Corollary 3.1]{DiaEstRoc:16} for the fixed points $P$ and $Q$ of $ F$. For completeness and to illustrate the dynamics, we will sketch its proof.

\begin{lemma}[Characterisation of homoclinic points]\label{l.charHQ}
Consider a point $X=(\xi,x)\in \Gamma$. The point $X$ is a homoclinic point of $Q$ if and only if
\begin{equation}\label{e.cond1}
	\xi= (0^{-\bN} \xi_{-\ell} \ldots \xi_{-1} .\xi_0 \ldots \xi_k 0^\bN),
\end{equation}
where
\begin{equation}\label{e.cond2}
	f_{[\xi_0 \ldots \,\xi_k]}(x)=0 
	\quad\mbox{and}\quad
f_{[\xi_{-\ell} \ldots \,\xi_{-1}.]}(x) \in [0,1).
\end{equation}
The point $X$ is a homoclinic point of $P$ if and only if
\[
	\xi= (0^{-\bN} \xi_{-\ell} \ldots \xi_{-1} .\xi_0 \ldots \xi_k 0^\bN)
\]
where
\[
	f_{[\xi_{-m} \ldots \,\xi_{-1}.]}(x)=1 \quad\mbox{and}\quad
	f_{[\xi_0 \ldots \,\xi_n]}(x) \in (0,1].
\]
\end{lemma}

\begin{proof}
We only prove the first part. Suppose that $X=(\xi,x)$ is a homoclinic point of $Q$. This immediately implies that  $\xi=(0^{-\bN} \xi_{-\ell} \ldots \xi_{-1} .\xi_0 \ldots \xi_k 0^\bN)$. Then the conditions
\[
	f_{[\xi_0 \ldots \,\xi_k]} (x) \in \cW^\s (0,f_0)=\{0\}
	\quad\mbox{and}\quad
	f_{[\xi_{-\ell} \ldots\, \xi_{-1}.]} (x) \in \cW^\u (0,f_0) \cap [0,1]=[0,1)
\]
prove one implication. To prove the converse one note that \eqref{e.cond2} implies that
\[
	f_{[\xi_0 \ldots\, \xi_k0^n]}(x)=0, 
	\quad \mbox{for every $n\ge 0$}
\]
and 
\[
	\lim_{n\to \infty}  f_{[0^{-n}\xi_{-\ell} \ldots\, \xi_{-1}.]}(x) 
	= 0,
\]
which together with \eqref{e.cond1} implies that $X\in \cW^\s(Q, F) \cap \cW^\u(Q, F)$, ending the proof.
\end{proof}

For the next remarks, consider two hyperbolic periodic points $R_1$ and $R_2$, where 
\begin{equation}\label{eq:R1R2}
 	R_1
	=((\xi_0\ldots\xi_{n-1})^\bZ,r_1), \quad
 	R_2
	=((\eta_0\ldots\eta_{m-1})^\bZ,r_2).
\end{equation}

\begin{remark}[Homoclinic relations]\label{rem:homoclinrel}
Assume that the points in \eqref{eq:R1R2} are of expanding type. They are homoclinically related if and only if there are points of the form
\[\begin{split}
	X
	&= \big(((\xi_0\ldots\xi_{n-1})^{-\bN} \tau_{-\ell_1}\ldots \tau_{-1}.\tau_0\ldots \tau_{\ell_2} 
			(\eta_0\ldots \eta_{m-1})^\bN),x\big), 
	\quad \mbox{and}\\
	Y
	&= \big(((\eta_0\ldots \eta_{m-1})^{-\bN} \rho_{-\ell_3} \ldots\rho_{-1}.\rho_0\ldots \rho_{\ell_4} (\xi_0\ldots\xi_{n-1})^{\bN}),y\big), 
\end{split}\]
such that
\begin{equation}\label{eq:sepjoint1}
\begin{split}
	f_{[\tau_0\ldots\, \tau_{\ell_2}]} (x)
	&\in \{r_2\}
	= \cW^\s_{\rm loc} (r_2, f_{[\eta_0\ldots\,\eta_{m-1}]})\\
	f_{[\tau_{-\ell_1}\ldots \,\tau_{-1}.]}(x) 
	&\in \cW^\u_{\rm loc} (r_1, f_{[\xi_0\ldots\,\xi_{n-1}]})
\end{split}
\end{equation}
and
\begin{equation}\label{eq:sepjoint2}
\begin{split}
	f_{[\rho_0\ldots\, \rho_{\ell_4}]} (y)
	&\in \{r_1\}
	= \cW^\s_{\rm loc} (r_1, f_{[\xi_0\ldots\,\xi_{n-1}]})\\
	f_{[\rho_{-\ell_1}\ldots \,\rho_{-1}.]}(y) 
	&\in \cW^\u_{\rm loc} (r_2, f_{[\eta_0\ldots \,\eta_{m-1}]}).
\end{split}
\end{equation}
 Note that $X\in \cW^\s (R_2, F) \cap \cW^\u (R_1, F)$ and 
 $Y\in \cW^\s (R_2, F) \cap \cW^\u (R_1, F)$. 
 
 There is a similar version for homoclinic relations of contracting points.
\end{remark}

\begin{remark}
	An immediate consequence of Lemma \ref{l.charHQ} and Remark \ref{rem:homoclinrel} is that the homoclinic classes $H(P, F)$ and $H(Q, F)$ are both non-trivial and hence $ F$ has infinitely many hyperbolic periodic points (homoclinically related either to $P$ or $Q$). Indeed, as $f_1(d)=0$ it follows that
\[\begin{split}
	&((0^{-\bN}.10^{\bN}),d) \in \cW^\s(Q, F) \cap \cW^\u(Q, F) 
		\subset H(Q, F)
	\quad\mbox{and}\\
	&(0^{-N}1.0^\bN,f_1(1) )\in  \cW^\s(P, F) \cap \cW^\u(P, F)
		\subset H(P, F). 
\end{split}\]
Hence, both homoclinic classes are infinite sets and hence they contain infinitely many hyperbolic periodic points. We need to understand the homoclinic relations among them. We know that some of them are related to $P$ and some to $Q$. The point of Proposition \ref{pro.l.homoclinicallyrelated} is that these are the only two possibilities.
\end{remark}

For completeness, we state the corresponding result for Remark \ref{rem:homoclinrel} for periodic points \eqref{eq:R1R2} of different type of hyperbolicity. 

\begin{remark}[Heterodimensional cycles]\label{rem:herocycle}
Assume now that $R_1$ and $R_2$ in \eqref{eq:R1R2} are of contracting and expanding  type, respectively. Note that in this case
\[
	\cW^\u_{\rm loc} (r_1, f_{[\xi_0\ldots\, \xi_{n-1}]}) = \{r_1\}
	\quad\text{ and }\quad
	\cW^\s_{\rm loc} (r_2, f_{[\eta_0\ldots\, \eta_{m-1}]}) = \{r_2\},
\]
while $\cW^\s_{\rm loc} (r_1, f_{[\xi_0\ldots \,\xi_{n-1}]})$ and $\cW^\u_{\rm loc} (r_2, f_{[\eta_0\ldots \,\eta_{m-1}]})$ are open intervals. Assume that there are points
\[\begin{split}
	X
	&= \big(((\xi_0\ldots\xi_{n-1})^{-\bN}.\tau_0\ldots \tau_{\ell_1} 
			(\eta_0\ldots \eta_{m-1})^\bN),r_1\big), 
	\quad \mbox{and}\\
	Y
	&= \big(((\eta_0\ldots \eta_{m-1})^{-\bN} .\rho_0\ldots \rho_{\ell_2} (\xi_0\ldots\xi_{n-1})^{\bN}),y\big), 
\end{split}\]
such that
\begin{itemize}
\item $f_{[\tau_0\ldots\, \tau_{\ell_1}]} (r_1)\in \{r_2\}
	= \cW^\s_{\rm loc} (r_2, f_{[\eta_0\ldots\, \eta_{m-1}]})$,
\item $y \in \cW^\u_{\rm loc} (r_2, f_{[\eta_0\ldots \,\eta_{m-1}]})$,
\item $f_{[\rho_0\ldots \,\rho_{\ell_2}]} (y)
		\in \cW^\s_{\rm loc} (r_1, f_{[\xi_0\ldots\,\xi_{n-1}]})$.
\end{itemize}
Then 
$X\in \cW^\s (R_2, F) \cap \cW^\u (R_1, F)$ and 
 $Y\in \cW^\s (R_2, F) \cap \cW^\u (R_1, F)$. In this case, following the terminology in the differentiable case, we say that $R_1$ and $R_2$ form a heterodimensional cycle. 
\end{remark}

\begin{proof}[Proof of Proposition \ref{pro.l.homoclinicallyrelated}]
We only consider the expanding case, the other one is analogous. 
Write $R=((\xi_0\ldots\xi_{n-1})^\bZ, r)$. Invoking Lemma \ref{lem:fixpots}, the fact that $r$ is an expanding fixed  point of $f_{[\xi_0\dots\xi_{n-1}]}$ implies that $f_{[\xi_0\dots\xi_{n-1}]}$ has exactly two periodic points in $I_{[\xi_0\dots \xi_{n-1}]}=[a,1]$ (with $a=a_{[\xi_0\dots \xi_{n-1}]}$ and hence $f_{[\xi_0\ldots\,\xi_{n-1}]}(a)=0$, using the notation in Remark \ref{rem:notationnn}): the point $r=p_{[\xi_0\ldots\,\xi_{n-1}]}^+$ and a point $r'=p_{[\xi_0\ldots\,\xi_{n-1}]}^-$ with $r<r'$ and 
$f'_{[\xi_0\dots\xi_{n-1}]}(r')<1$ (using the notation in Lemma \ref{lem:fixpots}). We immediately have
\[
	[a,r') 
	\subset \cW^\u (r, f_{[\xi_0\dots\xi_{n-1}]})
	\quad\mbox{and}\quad
	[0,1) 
	\subset \cW^\u (0, f_{0}).
\]
Consider the points
\[
	X
	= (((\xi_0\dots\xi_{n-1})^{-\bN}.(\xi_0\dots\xi_{n-1}) 0^{\bN}),a)
	\quad\mbox{and}\quad
	Y= ((0^{-N} . (\xi_0\dots\xi_{n-1})^{\bN}),r).
\]
Hence we have that $X\in \cW^\u(R,F)$ and from $r\in [0,1)\subset \cW^\u (0,f_0)$ we get $Y\in \cW^\u(Q,F)$. Similarly, $f_{[\xi_0\dots\xi_{n-1}]}(a)=0$ implies that $X\in \cW^\s (Q,F)$. Finally, $Y\in \cW^\s(R,F)$ is obvious. Therefore, the invariant sets of $Q$ and $\cO(R)$ intersect cyclically and hence the points $Q$ and $R$ are homoclinically related.
\end{proof}

\begin{corollary}\label{c.horseshoes}
Consider two basic sets $\Gamma_1,\Gamma_2\subset\Gamma$ of the same type of hyperbolicity. Then there is a horseshoe $\Gamma_3\subset \Gamma$ containing $\Gamma_1$ and $\Gamma_2$.
\end{corollary}

We conclude this subsection justifying the comments in Remark \ref{rem:graphlikestructure}. Recall the definition of $\Sigma^{\rm cod}$ in \eqref{eq:xiPQbis}.

\begin{proposition}\label{prol.extremalpoints}
	Given $\xi=(\xi^-.\xi^+)\in \Sigma^{\rm cod}$, with $I_\xi = [x_{\xi^+},x_{\xi^-}]$ we have 
\[
	(\xi, x_{\xi^-}) \in H(P, F)
	\quad \mbox{and}\quad
	(\xi, x_{\xi^+}) \in H(Q, F).
\]
\end{proposition}

Note that in the above result we may have $x_{\xi^+}=x_{\xi^-}$. In such a case we have that $H(P, F)\cap H(Q, F)\ne\emptyset$ (in that case both classes are nonhyperbolic).

\begin{proof}[Proof of Proposition~\ref{prol.extremalpoints}]
We prove the proposition for the point $(\xi ,x_{\xi^+})$, the proof for $(\xi,x_{\xi^-})$ is similar considering negative iterates. First let $[a_k,1]=I_{[\xi_0\ldots\, \xi_k]}$, where $a_k\eqdef a_{[\xi_0\ldots\, \xi_k]}$ and recall (see Remark \ref{rem:notationnn}) that
\begin{equation}\label{eq:fomrulaaa}
	f_{[\xi_0\ldots \,\xi_k]} (a_k)=0
\end{equation}	 
and that $(a_k)_k$ converges monotonically to $x_{\xi^+}$. Also recall that $a_k\le x_{\xi^+}$. It is enough to show the following.

\begin{claim*} 
	There are infinitely many pairs $(m,k)$, both arguments unbounded, such that 
\begin{equation}\label{eq:fomrulaaaa}
	f_{[\xi_{-m}\ldots\,\xi_{-1}.]}(a_k)
	\in[0,1).
\end{equation}	 
\end{claim*}

Assuming the above claim, there is a sequence $(m_\ell,k_\ell)_\ell$, $m_\ell\to\infty$, $k_\ell\to\infty$, satisfying \eqref{eq:fomrulaaaa}. Lemma~\ref{l.charHQ} implies that  
\[
	A_\ell
	\eqdef \big((0^{-\bN}\xi_{-m_\ell}\ldots\xi_{-1}.\xi_0\ldots\xi_{k_\ell}0^\bN),a_{k_\ell}\big).
\]
is a homoclinic point of $Q$ and hence belongs to $H(Q, F)$. As $A_\ell\to (\xi,x_{\xi^+})$, the proposition will follow. As \eqref{eq:fomrulaaa} is always valid, it only remains to prove the claim.

\begin{proof}[Proof of Claim]
There are the following cases according to the number of $1$s in $\xi$.

\smallskip\noindent\textbf{Case 1) $\xi^+$ has infinitely many $1$s:}
We list the positions of $1$s by ${i_k}$.
Note that $\xi^-$ is admissible for $x_{\xi^+}$, hence for every point in $[0,x_{\xi^+}]$, and in particular for $a_{i_k}$. Note that in this case we have $a_{i_k}<x_{\xi^+}$.
We claim that 
\[
	f_{[\xi_{-m}\ldots\,\xi_{-1}.]}(a_{i_k})
	\in[0,1)
\] 
for all $m,k\ge1$. Otherwise condition $a_{i_k}<x_{\xi^+}$ together with monotonicity would imply $1<f_{[\xi_{-m}\ldots\,\xi_{-1}.]}(x_{\xi^+})$, in contradiction with the admissibility for $x_{\xi^+}$. 

\smallskip\noindent\textbf{Case 2) $\xi^+$ has finitely many $1$s:}
Let $\ell\ge0$ such that $\xi_\ell=1$ and $\xi_k=0$ for all $k>\ell$.
By Remark \ref{rem:constseqs} this implies that 
\begin{equation}\label{eq:ellfixed}
	a_k=x_{\xi^+}
	\quad\text{ for every }\quad k>\ell.
\end{equation}	 
We consider two subcases according to the number of $1$s in $\xi^-$:

\smallskip\noindent\textbf{Case 2.a) $\xi^-$ has infinitely many $1$s:}
By Remark \ref{rem:numones}, the number of consecutive $1$'s is bounded. List by $j_k$ the positions such that $\xi_{-j_k}=1$ and $\xi_{-j_k-1}=0$. By the above, $(j_k)_k$ defines an increasing sequence. For each $k$ there are two possibilities:
\begin{itemize}
\item $f_{[\xi_{-j_k}\ldots\,\xi_{-1}.]}(x_{\xi^+})
	=f_{[\xi_{-j_k}\ldots\,\xi_{-1}.]}(a_k) \in[0,1)$,
\item $f_{[\xi_{-j_k}\ldots\,\xi_{-1}.]}(x_{\xi^+})=1$ and 
\[
	f_{[\xi_{-j_k-1}\xi_{-j_k}\ldots\,\xi_{-1}.]}(x_{\xi^+}) 
	= f_{[0\,\xi_{-j_k}\ldots\,\xi_{-1}.]}(x_{\xi^+})
	=f_{[0\,\xi_{-j_k}\ldots\,\xi_{-1}.]}(a_k)
	\in [0,1).
\]	
\end{itemize}
This yields \eqref{eq:fomrulaaaa} considering the sequences $(j_k,k)_k$ and $(j_k+1,k)_k$, respectively.

\smallskip\noindent\textbf{Case 2.b) $\xi^-$ has finitely many $1$s:}
Let $n\ge0$ such that $\xi_{-n}=1$ and $\xi_{-m}=0$ for all $m>n$.  
Note that we cannot have 
\[
	f_{[\xi_{-n}\ldots\,\xi_{-1}]}(x_{\xi^+})=1.
\]
Indeed, by \eqref{eq:ellfixed} we would have 
\[\begin{split}
	f_{[\xi_{-n}\ldots\,\xi_{-1}\xi_0\ldots\xi_\ell\,0]}(1)
	&= f_{[\xi_0\ldots\,\xi_\ell\,0]}\circ f_{[\xi_{-n}\ldots\,\xi_{-1}]}(1)\\
	&= f_{[\xi_0\ldots\,\xi_\ell\,0]}(x_{\xi^+})
	= f_{[\xi_0\ldots\,\xi_\ell\,0]}(a_{\ell+1})
	=0.
\end{split}\]
which would imply (recalling the definition of $\Sigma^{\rm het}$ in \eqref{eq:xiPQ})
\[
	\xi
	= (\xi^-.\xi^+)
	= (0^{-\bN}\xi_{-n}\ldots\xi_{-1}.\xi_0\ldots\xi_\ell0^\bN)
	\in \Sigma^{\rm het}
\]
and hence $\xi\not\in\Sigma^{\rm cod}$, contradiction. 
Thus, we have  
\[
	f_{[\xi_{-n}\ldots\,\xi_{-1}.]}(x_{\xi^+})\in[0,1)
\]
	 and hence for every $r\ge1$ we obtain
\[
	f_{[\xi_{-n-r}\ldots\xi_{-n}\ldots\,\xi_{-1}.]}(x_{\xi^+}) 
	= f_{[0^r\xi_{-n}\ldots\,\xi_{-1}.]}(x_{\xi^+})
	= f_{[0^r\xi_{-n}\ldots\,\xi_{-1}.]}(a_k)
	\in [0,1)
\]	 
This then yields \eqref{eq:fomrulaaaa} taking the sequence $(k,k)_k$. 

This proves the claim.
\end{proof}

The proof of the proposition is now complete.
\end{proof}

%------------------------------------------------------------------------------------------------------
\subsection{Homoclinic relations for parabolic periodic points}\label{ss:parabolic}
%------------------------------------------------------------------------------------------------------

We now consider homoclinic relations including parabolic points. Note that in our concave setting a parabolic periodic point behaves as an attracting periodic point (to its right) and a repelling periodic point (to its left). In this way, to each such a point we will associate two ``homoclinic classes'', one taking into account its contracting nature and the other one its expanding one. Let us provide the details.

Consider the set $\Per_{\le0}\subset\Gamma$ of periodic points which are either parabolic or of contracting typ. Define $\Per_{\ge0}$ analogously. We say that a pair of points $A=((\xi_0\ldots\xi_{m-1})^\bZ,a)$ and $B=((\eta_0\ldots\eta_{n-1})^\bZ,b)$ in $\Per_{\le0}$ are \emph{homoclinically${}_{\le0}$ related} if:
\begin{itemize}
\item either we have $\cO(A)=\cO(B)$,
\item or we have $\cO(A)\ne\cO(B)$ and there are words $(\alpha_1\ldots\alpha_k)$ and $(\beta_1\ldots\beta_\ell)$ such that
\[
	f_{[\beta_1\ldots\beta_\ell]}(a)
		\in\interior\big(\cW^\s_{\rm loc}(b,f_{[\eta_0\ldots\eta_{n-1}]})\big)
	\quad\text{ and }\quad
	f_{[\alpha_1\ldots\alpha_k]}(b)
		\in\interior\big(\cW^\s_{\rm loc}(a,f_{[\xi_0\ldots\xi_{m-1}]})\big).
\]
\end{itemize}
Analogously, we define being \emph{homoclinically${}_{\ge0}$ related} on the set $\Per_{\ge0}$. 

\begin{remark}[(Parabolic) homoclinic relations] 
Comparing the above definition with Remark \ref{rem:homoclinrel}, we note that there are some subtle differences considering the interior of the stable sets. Let us observe that if in the above definition $A$ is parabolic, then 
\[
	\interior\big(\cW^\dag_{\rm loc}(a,f_{[\xi_0\ldots\xi_{m-1}]})\big)
	=\big(\cW^\dag_{\rm loc}(a,f_{[\xi_0\ldots\xi_{m-1}]})\big)\setminus\{a\},
	\quad
	\dag\in\{\s,\u\}.
\]	
If $A$ is of contracting type, then 
\[
	\interior\big(\cW^\s_{\rm loc}(a,f_{[\xi_0\ldots\xi_{m-1}]})\big)	
	=\cW^\s_{\rm loc}(a,f_{[\xi_0\ldots\xi_{m-1}]}).
\]	
Finally, if $A$ is of expanding type, then 
\[
	\interior\big(\cW^\u_{\rm loc}(a,f_{[\xi_0\ldots\xi_{m-1}]})\big) 			
	=\cW^\u_{\rm loc}(a,f_{[\xi_0\ldots\xi_{m-1}]}).
\]	
Hence, if $A$ and $B$ are both of contracting type (expanding type), we recover the usual homoclinic relation (recall the characterization in \eqref{eq:sepjoint1}, \eqref{eq:sepjoint2}). 
\end{remark}

\begin{lemma}
	To be homoclinically${}_{\le0}$ (homoclinically${}_{\ge0}$) related is an equivalence relation on $\Per_{\le0}$ (on $\Per_{\ge0}$).
\end{lemma}	

\begin{proof}
	We only consider the case $\le0$. The only fact which remains to prove is transitivity. Let us consider the case $A=((\xi_0\ldots\xi_{m-1})^\bZ,a), B=((\eta_0\ldots\eta_{n-1})^\bZ,b), C=((\tau_0\ldots\tau_{r-1})^\bZ,c)\in\Per_{\le0}$, where $A$ and $B$ are related and $B$ and $C$ are related.  By definition, there are words $(\beta_1\ldots\beta_\ell)$ and $(\gamma_1\ldots\gamma_i)$ satisfying
\[
	f_{[\beta_1\ldots\beta_\ell]}(a)\in\interior\big(\cW^\s(b,f_{[\eta_0\ldots\eta_{n-1}]})\big)
	\quad\text{ and }\quad
	f_{[\gamma_1\ldots\gamma_i]}(b)
		\in\interior\big(\cW^\s_{\rm loc}(c,f_{[\tau_0\ldots\tau_{r-1}]})\big).
\]
Hence, for every $x$ sufficiently close to $b$ we have 
\[
	f_{[\gamma_1\ldots\gamma_i]}(x)
		\in\interior\big(\cW^\s_{\rm loc}(b,f_{[\tau_0\ldots\tau_{r-1}]})\big).
\]
Since 
\[
	f_{[\beta_1\ldots\beta_\ell(\eta_0\ldots\eta_{n-1})^j]}(a)\to b
\]
as $j\to\infty$, for every $j$ big enough we have that
\[
	f_{[\beta_1\ldots\beta_\ell(\eta_0\ldots\eta_{n-1})^j\gamma_1\ldots\gamma_i]}(a)
		\in \interior\big(\cW^\s_{\rm loc}(c,f_{[\tau_0\ldots\tau_{r-1}]})\big).	
\]
This implies the first condition in the relation. The other one is completely analogous.
\end{proof}

\begin{remark}
	Note that in the above definition for $\le0$ ($\ge0$) we needed to consider the interior of the stable (unstable) manifolds. Without this hypothesis such a relation may fail to be transitive. Note that for parabolic points the un-/stable manifolds are half-open intervals and intersections may occur in the boundary and this may cause non-transitivity.   
\end{remark}

Let $H_{\le0}(A, F)$ be the closure of the set of points in $\Per_{\le0}$ which are homoclinically${}_{\le0}$ related to $A$ and note that if $A$ is of contracting type then $H_{\le0}(A, F)=H(A, F)$. Analogously for $H_{\ge0}(A, F)$. If $A$ is a parabolic point then the sets $H_{\le0}(A, F)$ and $H_{\ge0}(A, F)$ necessarily intersect through the orbit of $A$, but they may be different. 

\begin{proposition}\label{pl.parabolichomoclinic}
Every parabolic periodic point $S\in \Gamma$ is homoclinically${}_{\le0}$ related to $P$ and homoclinically${}_{\ge0}$ related $Q$.
Hence 
$$
H_{\le0}(S, F)=H(P, F) \quad \mbox{and}
\quad
H_{\ge0}(S, F)=H(Q, F).
$$
In particular, if $ F$ has a parabolic periodic point then $H(P, F) \cap H(Q, F)\ne\emptyset$.
\end{proposition}

\begin{proof}
We only show that a parabolic periodic point $S=((\xi_0\ldots\xi_{m-1})^\bZ, s)$ is homo\-clinically${}_{\le0}$ related to $P$. By Remark \ref{rem:notationnn}, we have $I_{[\xi_0\dots \xi_{n-1}]} = [a,1]$, where $a=a_{[\xi_0\dots \xi_{n-1}]}$. By concavity, we have that $1\in \cW^\s(f_{[\xi_0\dots \xi_{n-1}],}s)$ and thus $f_{[\xi_0\dots \xi_{n-1}]}(1)\in\interior(\cW^\s(f_{[\xi_0\dots \xi_{n-1}],}s))$. As $s\in(0,1)\subset\interior( \cW^\s(f_0,1))$, it follows that $S$ are $P$ are homoclinically${}_{\le0}$ related.   As $H_{\le0}(P, F)=H(P, F)$ we are done. 
\end{proof}

%------------------------------------------------------------------------------------------------------
\section{Concave one-dimensional maps}\label{sec:concavemaps}
%------------------------------------------------------------------------------------------------------

In this section we collect some auxiliary results. Throughout, we assume (H1)--(H2+) and let $M$ be a constant as in (H2+). Similar arguments, in particular those in Lemma \ref{lem:monoto}, can also be found in \cite{AlsMis:15}. Recall the notation $f_\xi^n= f_{[\xi_0\ldots\,\xi_{n-1}]}$.

%------------------------------------------------------------------------------------------------------
\subsection{Distortion control}
%------------------------------------------------------------------------------------------------------

We prove two distortion results. 

\begin{lemma}[Controlled distortion]\label{lem:distortion}
	For every $\xi\in\Sigma$, every $x,y\in I_\xi$, $x<y$, and every $n\ge1$ we have
\[
	M^{-1}
	\le \frac{\log(f_\xi^{n})'(x)-\log(f_\xi^{n})'(y)}
		{\sum_{k=0}^{n-1} (f_\xi^{k}(y)-f_\xi^{k}(x))}
	\le M.
\]	
\end{lemma}

\begin{proof}
By hypothesis (H2+),
\[\begin{split}
	\log(f_\xi^{n})'(x)-\log(f_\xi^{n})'(y)
	&= \sum_{k=0}^{n-1}\big(\log f_{\xi_k}'(f_\xi^{k}(x)) - \log f_{\xi_k}'(f_\xi^{k}(y))\big)\\
	&\ge M^{-1}\sum_{ k=0}^{n-1} (f_\xi^{k}(y)-f_\xi^{k}(x)).
\end{split}\]
The lower bound is analogous.
\end{proof}

\begin{lemma}\label{lem:obv}
	There are positive increasing functions $C_1, C_2\colon(0,1]\to\bR$ with $C_i(x)\to0$ as $x\to0$ such that for every interval $I=[x,y]\subset[0,1]$ satisfying $\lvert I\rvert\ge a$ we have
\[
	\frac{\lvert f_i(I)\rvert}{\lvert I\rvert}e^{ C_1(a)}
	\le f_i'(x)
	\le \frac{\lvert f_i(I)\rvert}{\lvert I\rvert}e^{ C_2(a)},
	\quad
	i=0,1.
\]	
\end{lemma}

\begin{proof}
For the first inequality, using (H2+), we have
\[
	\lvert f_i(I)\rvert
	= \int_If_i'(z)\,dz
	\le f_i'(x)\int_Ie^{-M^{-1}(z-x)}\,dz
	= f_i'(x) \frac{1-e^{-M^{-1}\lvert I\rvert}}{M^{-1}}.
\]
Thus, if $\lvert I\rvert\ge a$ then
\[
	f_i'(x)
	\ge \frac{\lvert f_i(I)\rvert}{\lvert I\rvert } 
		\frac{\lvert I\rvert M^{-1}}{1-e^{-M^{-1}\lvert I\rvert}}
	\ge  \frac{\lvert f_i(I)\rvert}{\lvert I\rvert } 
		e^{ C_1(\lvert I\rvert)},
\]
where $ C_1(a)\eqdef \log(a M^{-1}/(1-e^{-a M^{-1}}))$ has the claimed properties. 

The other inequality follows analogously, taking $ C_2(a)\eqdef \log(a M/(1-e^{-a M}))$.
\end{proof}
%------------------------------------------------------------------------------------------------------
\subsection{Rescaling of moving intervals}
%------------------------------------------------------------------------------------------------------

The following scheme will be used several times. Consider $\xi\in\Sigma$ such that $I_{\xi^+}$ is not a singleton and points $x_1,x_2,x_3\in I_{\xi^+}$ satisfying $x_1<x_2<x_3$. Let $I^0\eqdef [x_1,x_3]$. For $n\ge0$ consider the ``moving intervals"
\[
	I^n
	\eqdef f_\xi^n(I^0)
	= [x_1^n,x_3^n],
	\quad\text{ where }\quad
	x_1^n
	\eqdef f_\xi^n(x_1),\quad
	x_3^n
	\eqdef f_\xi^n(x_3).
\]	
Given $k$, for $\ell > k$ consider the rescaled map $g_k^\ell\colon I^k\to[0,\infty]$ defined by
\begin{equation}\label{eq:defgkl}
	g_k^\ell(x)
	\eqdef \frac{\lvert I^k\rvert}{\lvert I^\ell\rvert} (f_{\xi_{\ell-1}}\circ\ldots\circ f_{\xi_k})(x).
\end{equation}
Observe that
\[
	(g_k^\ell)'(x)
	= (g_{\ell-1}^\ell)'(f_\xi^{\ell-1}(x))\cdot\ldots\cdot (g_k^{k+1})'(f_\xi^k(x)).
\]
Hence, for every $x\in I^0$ we have
\begin{equation}\label{eq:equal}
	(g_0^n)'(x)
	= \frac{\lvert I^0\rvert}{\lvert I^n\rvert}(f_\xi^n)'(x).
\end{equation}
Further, observe that
\[
	\lvert g_k^{k+1}(I^k)\rvert
	= \lvert I^k\rvert.
\]
Note that concavity also implies
\begin{equation}\label{eq:ofcourse}
	(g_k^{k+1})'(x_1^k)
	\ge 1.
\end{equation}

Hence, arguing inductively, for every $n\ge0$ we have
\begin{equation}\label{eq:good}
	\lvert g_0^n(I^0)\rvert
	= \lvert I^0\rvert.
\end{equation}

\begin{lemma}\label{lem:monoto}
	If $\sum_{k=0}^\infty(x_3^k-x_1^k)=\infty$, then for every choice of $x_2\in(x_1,x_3)$ we have 
\[
	\lim_{n\to\infty}(x_3^n-x_2^n)
	=0.
\]
\end{lemma} 

\begin{proof}
	Note that the sequence $(x_3^n-x_2^n)_n$ is not necessarily monotone, so convergence is not immediate. 	We start by proving a slightly stronger fact.
	
\begin{claim*}
\[
	d
	\eqdef\lim_{n\to\infty}\frac{x_3^n-x_2^n}{x_3^n-x_1^n}
	= 0.
\]	
\end{claim*}	

As  $0< \lvert x_3^n-x_1^n\rvert\le 1$, the lemma is then an immediate consequence of the above claim.

\begin{proof}[Proof of the claim]
	Note that concavity of the maps  implies that the sequence in the claim is monotonically decreasing, hence its limit exists.	 By contradiction, suppose that $d>0$. 
	
	 By monotonicity of the derivatives of $f_0,f_1$ and the choice of $x_2$, by \eqref{eq:equal} we have
\begin{equation}\label{eq:meud}\begin{split}
	\lvert (g_0^n)'(x_2)\rvert
	&= \frac{\lvert I^0\rvert}{\lvert I^n\rvert} (f_\xi^n)'(x_2)
	= \frac{x_3-x_1}{x_3^n-x_1^n} (f_\xi^n)'(x_2)
	\ge \frac{x_3-x_1}{x_3^n-x_1^n} \frac{x_3^n-x_2^n}{x_3-x_2}
	>d.
\end{split}\end{equation}	
Fix some $y\in(x_1,x_2)$ and let $\alpha\in(0,1)$ such that $x_2=\alpha y+(1-\alpha)x_3$. By concavity of $g_0^n$, letting $y^n=f_\xi^n(y)$, we get
\[\begin{split}
	\frac{\lvert I^0\rvert}{\lvert I^n\rvert}(\alpha y^n+(1-\alpha)x_3^n)
	&= \alpha g_0^n(y)+(1-\alpha)g_0^n(x_3)\\
	&\le g_0^n(\alpha y+(1-\alpha)x_3)
	= \frac{\lvert I^0\rvert}{\lvert I^n\rvert}x_2^n
\end{split}\]
and hence 
\[
	\alpha y^n+(1-\alpha)x_3^n\le x_2^n,
\]	 
which implies
\[
	x_3^n-x_1^n
	\ge x_2^n-y^n
	\ge (1-\alpha)(x_3^n-y^n)
	\ge (1-\alpha)(x_3^n-x_2^n).
\]
Hence
\[
	\frac{x_2^n-y^n}{x_3^n-x_1^n}
	\ge (1-\alpha)\frac{x_3^n-x_2^n}{x_3^n-x_1^n}
	\ge (1-\alpha)d>0.
\]
Further, again by monotonicity of the derivatives, \eqref{eq:equal}, and by Lemma \ref{lem:distortion} together with the above we obtain
\[\begin{split}
	\max_{z_1,z_2\in[y,x_2]}\frac{(g_0^n)'(z_1)}{(g_0^n)'(z_2)}
	&= \frac{(g_0^n)'(y)}{(g_0^n)'(x_2)} 
	= \frac{(f_\xi^n)'(y)}{(f_\xi^n)'(x_2)} 
	\ge	\exp\Big({M^{-1}\sum_{k=0}^{n-1}(x_2^k-y^k)}\Big)\\
	&\ge \exp\Big({M^{-1}(1-\alpha)d\sum_{k=0}^{n-1}(x_3^k-x_1^k)}\Big).
\end{split}\]
By hypothesis, the latter diverges as $n\to\infty$. By \eqref{eq:meud}, we obtain $(g_0^n)'(y)\to\infty$. By concavity, $\min_{z\in[x_1,y]}(g_0^n)'(z)\to\infty$ as $n\to\infty$. But this implies that 
\[
		g_0^n(y)-g_0^n(x_1) 
	\to\infty,
\] 
which contradicts \eqref{eq:good}, proving the claim. 
\end{proof}

This proves the lemma.
\end{proof}

%------------------------------------------------------------------------------------------------------
\subsection{Distance to fixed points}
%------------------------------------------------------------------------------------------------------

The following lemma will be instrumental in the proof of Theorem \ref{teo:homoclinicclasses}.

\begin{lemma} \label{lem:close0}
There exists a function $h\colon(0,\infty)\to(0,\infty)$ with $\lim_{t\to 0} h(t)=0$ satisfying the following property. Let $(\omega_1\ldots\omega_n)\in \{0,1\}^n$, $n\ge1$, be a word such that the map $g=f_{[\omega_1\ldots\,\omega_n]}$ has some fixed point in $[0,1]$. If $\lvert g(x)-x\rvert<\varepsilon$ for some $x\in[a_{[\omega_1\ldots\,\omega_n]},1]$, then the distance of $x$ to the closest fixed point of $g$ is not larger than $h(\varepsilon)$.
\end{lemma}

\begin{proof}
Let $z$ be the (unique) point satisfying $g'(z)=1$ (that is, the maximum point for $t\mapsto g(t)-t$). Denote by $x^\pm$ the fixed point(s) of $g$, $x^+\le z\le x^-$. Note that we may have $x^+=x^-=z$. 

We only study the case $x\le z$, the case $x>z$ is analogous.

\begin{claim*} 
For every $y\le z-\sqrt\varepsilon$ we have $g'(y)\ge e^{M^{-1}\sqrt\varepsilon}$.
\end{claim*}

\begin{proof}
 By (H2+), we have $\log (f_{\omega_1}'(y)/ f_{\omega_1}'(z)) \ge M^{-1}\sqrt\varepsilon$. Hence
\[\begin{split}
	g'(y)
	&= f_{\omega_1}'(y) \cdot\ldots\cdot  f_{\omega_n}'(f_{[\omega_1\ldots \,\omega_{n-1}]}(y)) 
	 > \frac {f_{\omega_1}'(y)} {f_{\omega_1}'(z)} \cdot f_{\omega_1}'(z) \cdot\ldots\cdot  f_{\omega_n}'(f_{[\omega_1\ldots \,\omega_{n-1}]}(z))    \\
	 &\geq e^{M^{-1}\varepsilon^{1/2}}g'(z)
	 = e^{M^{-1}\varepsilon^{1/2}},
\end{split}\] 
proving the claim.
\end{proof}

Define 
\[
	h(\varepsilon)
	\eqdef (e^{M^{-1}\sqrt\varepsilon}-1)^{-1}\varepsilon + \sqrt\varepsilon.
\]
Given $\varepsilon>0$ such that $\lvert g(x)-x\rvert<\varepsilon$, denote
\[
	y= z-\sqrt\varepsilon.
\]
There are two cases:

\smallskip
\noindent \textbf{Case $y < x^+\le z$:}
If $y \le x\le z$, then clearly we have $\lvert x-x^+\rvert\le h(\varepsilon)$. 
Assume now $x< y$. By hypothesis and since $x\le x^+$, we have 
\[
	\varepsilon> \lvert g(x)-x\rvert = -(g(x)-x).
\]	 
Since $g(y)-y <0$, by concavity we have
\[\begin{split}
	\varepsilon
	&> 0-(g(x)-x)
	= (g(x^+)-x^+)-(g(y)-y) +(g(y)-y)-(g(x)-x)\\
	&> 0 +(g-{\rm id})'(y)(y-x)
	\ge (e^{M^{-1}\sqrt\varepsilon}-1)(y-x),
\end{split}\]
where for the latter we used the above claim.
Thus, by the above, we obtain 
\[
	\lvert x-x^+\rvert
	\le \lvert x-y\rvert+\lvert y-x^+\rvert
	\le \varepsilon(e^{M^{-1}\sqrt\varepsilon}-1)^{-1} +\sqrt\varepsilon
	\le h(\varepsilon).
\]	

\smallskip
\noindent \textbf{Case $x^+\le y$:}
If $y \le x\le z$, by hypothesis and as $x\ge x^+$, we get 
\[
	\varepsilon> \lvert g(x)-x\rvert = g(x)-x.
\]	 
By concavity, applying the above claim, we have
\[\begin{split}
	\varepsilon
	&> (g(x)-x)-(g(y)-y) +(g(y)-y)-(g(x^+)-x^+)\\
	&> 0 +(g-{\rm id})'(y)(y-x^+)
	\ge (e^{M^{-1}\sqrt\varepsilon}-1)(y-x^+),
\end{split}\]
which implies, as above,
\[
	\lvert x-x^+\rvert
	\le \lvert x-y\rvert+\lvert y+x^+\rvert
	\le h(\varepsilon).
\]

Finally, if $x< y$, we have
\[
	\varepsilon
	> \lvert g(x)-x\rvert 
	= \lvert (g(x)-x)-(g(x^+)-x^+)\rvert.
\]
Letting now $w=\max\{x^+,x\}$, using again the claim, we have
\[
	\varepsilon
	>\lvert (g(x)-x)-(g(x^+)-x^+)\rvert
	\ge (g-{\rm id})'(w)\lvert x-x^+\rvert
	\ge (e^{M^{-1}\sqrt\varepsilon}-1)\lvert x-x^+\rvert,
\]
obtaining $\lvert x-x^+\rvert<h(\varepsilon)$. This proves the lemma.
\end{proof}

%------------------------------------------------------------------------------------------------------
\section{Density of periodic points -- Proof of Theorem \ref{teo:homoclinicclasses}}
\label{sec:proofofteo:homoclinicclasses}
%------------------------------------------------------------------------------------------------------

%------------------------------------------------------------------------------------------------------
\subsection{Approximation by (hyperbolic) periodic points}	
%------------------------------------------------------------------------------------------------------

First observe that, as our maps are local diffeomorphisms, every periodic point has positive derivative which might be equal to $1$ (parabolic) or different from $1$ (hyperbolic, either of contracting or of expanding type). The next lemma deals with the approximation of parabolic periodic points -- in case such points do exist -- by hyperbolic ones.

\begin{lemma}\label{lem:parpoiacc}
	Assume (H1)--(H2). Every parabolic periodic point is accumulated by hyperbolic periodic points of either type of hyperbolicity.
	Moreover, every parabolic periodic measure is weak$\ast$ accumulated by hyperbolic periodic measures of either type of hyperbolicity.
\end{lemma}

\begin{proof}
We only prove the lemma for periodic points of contracting type, the other case of expanding type is similar and hence omitted.

	Let $X=F^n(X)=((\omega_0\ldots\omega_{n-1})^\bZ,x)$ be a parabolic periodic point. Abbreviate the word $(\omega_0\ldots\omega_{n-1})$ simply by $\omega$. Hence  
\[
	f_{[\omega]}(x)=x
	\quad\text{ and }\quad	
	(f_{[\omega]})'(x)=1
\]	 
and therefore $\omega\ne0^n$ and $x\ne0$.
Notice also that $(f_{[\omega]})'<1$ in $(x,1]$. Thus, $f_{[\omega]}((x,1])\subset(x,1)$. As $x\ne0$, there is $k\ge1$ such that $(f_0^k)'(x)<1$. Hence, we have $(f_0^k)'<1$ in $[x,1]$. Note also that $f_0^k((x,1))\subset(x,1)$.
	
	Noting that $1$ is in the basin of attraction of $x$ with respect to $f_{[\omega]}$,  given $\varepsilon>0$ sufficiently small, for every $\ell\ge1$ sufficiently large we have $f_{[\omega^\ell]}(1)\in(x,x+\varepsilon)$. 
	Observe that 
\[
	f_{[\omega^\ell0^k\omega^\ell]}(1)
	\le f_{[\omega^\ell]}(1)
	<x+\varepsilon.
\]	
By the above, we have $f_{[\omega^\ell0^k\omega^\ell]}((x,1])\subset(x,x+\varepsilon)$ and hence  there is a periodic point $p^{(\ell)}$ for $f_{[\omega^\ell0^k\omega^\ell]}$ in $(x,x+\varepsilon)$. Moreover $(f_{[\omega^\ell0^k\omega^\ell]})'<1$ on $(x,1]$.
 Since the periodic sequences 
 \[
 	\eta^{(\ell)}\eqdef(\omega^\ell0^k\omega^\ell)^\bZ 
	= ((\omega^\ell0^k\omega^\ell)^{-\bN}.(\omega^\ell0^k\omega^\ell)^\bN),
\]
satisfy $\eta^{(\ell)}\to (\omega^{-\bN}.\omega^\bN)=\omega^\bZ$ as $\ell\to\infty$, the corresponding $F$-periodic points 
$Y_\ell=(\eta^{(\ell)},p^{(\ell)})$
converge to $X=(\omega^\bZ,x)$. Moreover, each $Y_\ell$ is hyperbolic of contracting type.

The claim about weak$\ast$ approximation of the parabolic measure by hyperbolic periodic ones is immediate by construction.
\end{proof}

%------------------------------------------------------------------------------------------------------
\subsection{Proof of Theorem \ref{teo:homoclinicclasses}}
%------------------------------------------------------------------------------------------------------

Throughout this section, we assume (H1)--(H2+).
By Proposition \ref{pro.l.homoclinicallyrelated}, we have
\[
	{\rm closure}\{A\in\Gamma\colon A\text{ hyperbolic and periodic}\}
	= H(P,F)\cup H(Q,F).	
\]
Hence, the first claim in the theorem is then a consequence of the following lemma.

\begin{lemma}
We have $\Omega(\Gamma,F)={\rm closure}\{A\in\Gamma\colon A\text{ hyperbolic and periodic}\}$.
\end{lemma}

\begin{proof}
First recall that, by Theorem \ref{teo:coded}, every point in $\Gamma^{\rm het}$ is isolated in $\Gamma$ and non-periodic, hence it is wandering. Hence, to prove the lemma, it is enough to see that every nonwandering point in $\Gamma^{\rm cod}$ is accumulated by hyperbolic periodic points. 
By Lemma \ref{lem:parpoiacc}, every parabolic periodic point is accumulated by hyperbolic periodic ones. Hence it remains to show approximation by (either hyperbolic or parabolic) periodic points.

 Let $X=(\xi,x)\in\Gamma^{\rm cod}$ be a nonwandering point, $X\not\in\{P,Q\}$. For $k\in\bZ$ denote 
 \[
 	x_k=f_\xi^k(x).
\]	 
As the set of nonwandering points $\Omega(\Gamma,F)$ is $F$-invariant, for every $k\in\bZ$ the point $X_k=F^k(X)=(\sigma^k(\xi),x_k)$ is also nonwandering. There are the following cases:
\begin{itemize}
\item[1)] there exists a smallest $m_0\ge1$ so that $x_{m_0}=0$,
\item[2)] there exists a smallest $m_0\ge1$ so that $x_{m_0}=1$.
\item[3)] for every $m\ge1$ we have $x_m\in(0,1)$.
\end{itemize}
 
\noindent \textbf{Case 1):}  Given $n\ge1$, observe that $f_{[\xi_{-n}\ldots\,\xi_{m_0}]}(x_{-n})=0$, hence the only forward admissible continuation at 0 is $0^\bN$, that is, we have $(\xi_{-n}\ldots\xi_m)=(\xi_{-n}\ldots\xi_{m_0}0^{m-m_0})$ for all $m>m_0$. 
Thus, for every $m>m_0$ we have $f_{[\xi_{-n}\ldots\,\xi_m]}(x_{-n})=0$. At the same time, given $\varepsilon>0$ small, $f_{[\xi_{-n}\ldots\,\xi_m]}(x_{-n}+\varepsilon)\to 1$ as $m\to\infty$. 
Thus, we can find $m(\varepsilon)>\max\{0, 2n-m_0\}$ such that for every $m>m(\varepsilon)$ it holds
\[
	f_{[\xi_{-n}\ldots\,\xi_m]}(x_{-n}+\varepsilon)>x_{-n}+\varepsilon,
\]	 
and hence this map has a fixed point in the interval $[x_{-n}, x_{-n}+\varepsilon]$. Note that each such point corresponds to a $F$-periodic point of period $m$, $Y=((\eta_0\ldots\eta_{m-1})^\bZ,y)$, where $(\eta_0\ldots\eta_{m-1})=(\xi_{-n}\ldots\xi_{-1}\xi_0\ldots\xi_{n-1}\xi_n\ldots\xi_m)$ and $y\in[x_{-n},x_{-n}+\varepsilon]$. Hence,
\[\begin{split}
	F^n(Y)
	&\in [\xi_{-n}\ldots\xi_{-1}.\xi_0\ldots\xi_{n-1}]\times 
		f_{[\xi_{-n}\ldots\,\xi_{-1}.]}([x_{-n},x_{-n}+\varepsilon])\\
	&= [\xi_{-n}\ldots\xi_{-1}\xi_0\ldots\xi_{n-1}]\times 
		[x_0,x_0+\tau(\varepsilon)],
\end{split}\]	
where $\tau(\varepsilon)\to0$ as $\varepsilon\to0$.
Passing first with $\varepsilon$ to $0$ and then with $n$ to $\infty$, we see that $X$ is accumulated by $F$-periodic points.

\noindent \textbf{Case 2):} This case is analogous to Case 1 considering backward iterates.
 
\smallskip\noindent \textbf{Case 3):}
Given $n\ge1$ choose $\varepsilon>0$ sufficiently small such that $x_{-n}+\varepsilon<1$ and that the word $(\xi_{-n}\ldots \xi_{n-1})$ is backward admissible for all $t>x_{-n}-\varepsilon$. Indeed, such $\varepsilon$ exists because $f_{[\xi_{-n}\ldots\,\xi_m]}(x_{-n})\in(0,1)$ for every $m=1,\ldots,n$. Consider the neighborhood 
\[
	U
	= [\xi_{-n}\ldots\xi_{n-1}] \times [x_{-n}-\varepsilon, x_{-n}+\varepsilon]  
\]	 
of the nonwandering point $F^{-n}(X)$. There exists a point $Y=(\eta,y)\in U$ and a number $m>2n$ such that $F^m(Y)=(\sigma^m(\eta),y_m)\in U$. 
Note that as $m>2n$, we have  
\begin{equation}\label{eq:diegleichen}
	(\eta_0\ldots\eta_{2n-1})=(\xi_{-n}\ldots\xi_{n-1}).
\end{equation}	  
Note also that $I_{[\eta_0\ldots\,\eta_{2n-1}]}\supset I_{[\eta_0\ldots\,\eta_{m-1}]}$. 

Let $g=f_{[\eta_0\ldots\,\eta_{m-1}]}$. Observe that $y,g(y)\in[x_{-n}-\varepsilon,x_{-n}+\varepsilon]$ implies $\lvert g(y)-y\rvert\le2\varepsilon$. There are two subcases to consider.

\smallskip\noindent \textbf{Case 3a) The map $g$ has a fixed point:} 
By Lemma \ref{lem:close0}, there is a fixed point of $g$  which is $h(2\varepsilon)$-close to $y$. 

\smallskip\noindent \textbf{Case 3b) The map $g$ has no fixed point:} As there is no fixed point for $g$, we have 
\begin{equation}\label{eq:decreasing}
	g(t)<t
	\quad\text{ for all }\quad
	t\in I_{[\eta_0\ldots\,\eta_{m-1}]}.
\end{equation} 
In particular, $y_m\eqdef g(y)$ satisfies $y_m<y$. Moreover, observe that the number of admissible concatenations of $g$ is bounded from above by some number $\ell_0$.
Given any $\delta\in(0,\varepsilon)$, we now find a fixed point within the interval $[y_m,y+\delta)$.
Since $y_m\ge x_{-n}-\varepsilon$, by the choice of $\varepsilon$ and with \eqref{eq:diegleichen}, we have  $y_m\in I_{[\xi_{-n}\ldots\,\xi_{n-1}]}=I_{[\eta_0\ldots\,\eta_{2n-1}]}$, that is, $(\eta_0\ldots\eta_{2n-1})$ is forward admissible at $y_m$. Note that for every $k\ge1$ sufficiently large we have $\tilde y_k\eqdef (f_0^k\circ f_{[\eta_0\ldots\,\eta_{2n-1}]})(y_m)>y$, hence $\tilde y_k\in I_{[\eta_0\ldots\,\eta_{m-1}]}$ and we can apply $g$ to $\tilde y_k$. Observing again \eqref{eq:decreasing} and recalling that $[y_m,y)$ is a fundamental domain for $g$, for every $k$, there is a unique number $\ell=\ell(k)\in\{1,\ldots,\ell_0\}$ such that 
\[
	z_k
	\eqdef h_{(k)} (y_m) \in[y_m,y)
	,\quad\text{ where }\quad
	h_{(k)}\eqdef g^\ell \circ f_0^k\circ f_{[\eta_0\ldots\,\eta_{2n-1}]}.
\]
Note that, by construction, the derivative of $h_{(k)}$ in $[y_m,1]$ tends to $0$ as $k\to\infty$ (here we use the fact that $\ell\le\ell_0$ and that $f_0$ is contracting at $1$). Hence, $h_{(k)}$ has a unique fixed point $q_k\in[y_m,1]$. Consider now the sequence $(q_k)_k$. Taking a subsequence, we can assume that it converges to some point $q_\infty\in[y_m,y]$. In the case when $q_\infty<y$, then $q_k\in[y_m,y)$ for every large enough $k$. If $q_\infty=y$, then $q_k\to y$ and hence $q_k\in[y_m,y+\delta)$ for every large enough $k$. 
And again we can apply Lemma \ref{lem:close0} to find a fixed point of $h_{(k)}$ which is $h(3\varepsilon)$-close to $y_m$. 

In both cases, as in Case 1), it follows that $X$ is accumulated by $F$-periodic points.

The proof of the lemma is now complete.
\end{proof}

What remains to show are the properties related to hyperbolicity. For that we will use Theorem  \ref{teo:accum} whose proof is postponed but is independent of what comes next.
 
\begin{lemma}
	If the sets $H(P,F)$ and $H(Q,F)$ both are hyperbolic, then they are disjoint. 
\end{lemma}

\begin{proof}
Note that $P\in H(P,F)$ implies that $F$ is uniformly contracting on $H(P,F)$ and $Q\in H(Q,F)$ implies that $F$ is uniformly contracting on $H(Q,F)$. Thus, both sets are disjoint.
\end{proof}

To conclude the proof of the theorem, it is enough to prove the following lemma.

\begin{lemma}
	Assume that $H(P,F)\cap H(Q,F)=\emptyset$. Then the sets $H(P,F)$ and $H(Q,F)$ are both hyperbolic. 
\end{lemma}

\begin{proof}The lemma is a consequence of the following claim.

\begin{claim*}
	Every measure supported on $H(P,F)$ is hyperbolic (of contracting type) and every measure supported on $H(Q,F)$ is hyperbolic (of expanding type).
\end{claim*}

By the above claim, the set $H(P,F)$ is compact $F$-invariant and hyperbolic of contracting type on a set of total probability%
\footnote{Recall that a set $A$ is \emph{of total probability} if $\mu(A)=0$ for every $F$-invariant probability measure $\mu$.}. Thus, we can invoke \cite[Corollary E]{AlvAraSau:03} and obtain that $H(P,F)$ is hyperbolic of contracting type. 
Analogously, $H(Q,F)$ is hyperbolic of expanding type. This ends the proof of the lemma.

\begin{proof}[Proof of the claim]
First recall that, by Proposition \ref{pl.parabolichomoclinic}, parabolic periodic points are simultaneously in both classes $H(P,F)$ and $H(Q,F)$. As, by hypothesis, these classes are disjoint, there are no such points. Also observe that, by Proposition \ref{pro.l.homoclinicallyrelated}, every periodic point of expanding type is in $H(Q,F)$. As a consequence, all periodic points in $H(P,F)$ are hyperbolic of contracting type. Analogously, all periodic points in $H(Q,F)$ are hyperbolic of expanding type.

It remains to see that all ergodic measures  are hyperbolic. By contradiction, assume that there is some ergodic nonhyperbolic measure. Then, by Theorem \ref{teo:accum} such measure is simultaneously weak$\ast$ accumulated by  hyperbolic periodic measures of contracting type (hence supported on $H(P,F)$) and expanding type (hence on $H(Q,F)$). The latter contradicts the disjointness of the homoclinic classes. 
\end{proof}

This proves the lemma.
\end{proof}
This completes the proof of Theorem \ref{teo:homoclinicclasses}.
\qed

%------------------------------------------------------------------------------------------------------
\section{Partition of the spaces of ergodic measures -- Proofs of Theorems \ref{teo:1} and \ref{teo:2}}
\label{sec:splitt}
%------------------------------------------------------------------------------------------------------

In this section, we show that every ergodic measure in $\Sigma$ lifts to at least one (Section \ref{sec:lifmea}) and at most two (Section \ref{sec:prodis}) ergodic measures in $\Gamma$. In the latter case we call those \emph{twin-measures} and study their distance (Section \ref{sec:distwimea}). Thereafter, we will conclude the proofs of Theorems  \ref{teo:1} and \ref{teo:2}. We close this section discussing the frequencies of $0$'s and $1$'s, which will be used in the bifurcation analysis in Section \ref{ss.explosionofentropyand}.

Unless otherwise stated, in this section we will assume hypotheses (H1)--(H2+). Recall that hypothesis comes with a constant $M$.
 
%------------------------------------------------------------------------------------------------------
\subsection{Lifting measures}\label{sec:lifmea}
%------------------------------------------------------------------------------------------------------

Observe that in Theorem \ref{teo:1}, given $\nu \in \cM_{\rm erg}(\Sigma)$, the existence of ergodic measures $\mu\in \cM_{\rm erg}(\Gamma)$ projecting to $\nu$ is a simple fact which does not depend on (H2+). The main point of Theorem \ref{teo:1} is the fact that either there is precisely one hyperbolic measure of each type projecting to $\nu$  or there is only one nonhyperbolic one; our proof requires (H2+). Indeed, we have the following lemma.

\begin{lemma}\label{lem:existsemihypmeas}
	Assume (H1)--(H2). Given any $\nu\in\cM_{\rm erg}(\Sigma)$, there exist measures $\mu_1\in\cM_{\rm erg,\le0}(\Gamma)$ and $\mu_2\in\cM_{\rm erg,\ge0}(\Gamma)$ such that $\pi_\ast\mu_1=\nu=\pi_\ast\mu_2$.
\end{lemma}

\begin{proof}
In the case when $\nu$ is the Dirac measure supported on $0^\bZ$ then the lemma follows immediately taking $\mu_1$ and $\mu_2$ being the Dirac measures supported on $P$ and $Q$, respectively.

Assume now that $\nu\in M_{\rm erg}(\Sigma)$ is different from such a Dirac measure.  Consider a $\nu$-generic point $\xi=\xi^-.\xi^+$. Observe that $\xi^+\ne0^\bN$. Hence the forward admissible interval is $I_{\xi^+}=[a_{\xi^+},1]$, where $a_{\xi^+}>0$. Take any $\xi$-admissible point $x$. 
Denote $x_n=f_\xi^n(x)$. Let $y\eqdef\liminf_n x_n$, and observe that hence $y>0$.

\begin{claim*}
	There exists a sequence $(n_k)_k$ such that $x_{n_k} \geq x_{n_0}$.
\end{claim*}

\begin{proof}
	 If $y=1$ then $y=\lim_nx_n$ and the claim follows. 
	
	 Assume now $y\in(0,1)$. Recall that the graph of $f_1$ is below the diagonal. Let $\varepsilon>0$ small so that $f_1(x)<x-\varepsilon$ for every $x\in[d,1]$ and $f_0(x)>y+\varepsilon$ for every $\lvert x-y\rvert<\varepsilon$ (which is possible because $f_0(y)>y$).	 
Choose $n_0\ge1$ so that $x_{n_0}<y+\varepsilon$. Consider a subsequence $(n_k)_k$ such that $\lvert x_{n_k}-y\rvert<\varepsilon$ for every $k$. There are two cases:
First, if $x_{n_k} \ge x_{n_0}$ for infinitely many $k$'s, then we are done taking the sequence $(x_{n_k})_k$.
Otherwise, we can assume that $x_{n_k}<x_{n_0}$ for every $k$.  Then we must have $\xi_{n_k+1}=0$ eventually. Indeed, otherwise the sequence $(x_{n_k})_k$ satisfies  $x_{n_k}<x_{n_0}$ and, by our choice of $\varepsilon$, 
\[
	x_{n_k+1} 
	= f_1(x_{n_k})
	< x_{n_k}-\varepsilon
	< x_{n_0}-\varepsilon
	< y,
\]	 
contradicting the definition of $y$. Hence, by the choice of $\varepsilon$, in this case it follows 
\[
	x_{n_k+1}
	= f_0(x_{n_k})
	> y+\varepsilon
	> x_{n_0}
\]	
and we are done taking the sequence $(x_{n_k+1})_k$.
\end{proof}

For the following, for simplicity, assume that $n_0=0$.

Consider now for each $k$ the periodic measure $\nu_k$ supported on $\eta^{(k)}=(\xi_0\ldots\xi_{n_k-1})^\bZ$. By the above claim, we have 
\[
	f_{[\xi_0\ldots\,\xi_{n_k-1}]}([x_0,1])
	\subset [x_{n_k},1]
	\subset [x_0,1].
\]
Hence, $\eta^{(k)}$ is forward admissible at $x_0$. By Remark \ref{rem:magenta},  there exists a point 
\[
	y_k
	= f_\xi^{n_k}(y_k)
	= f_{\eta^{(k)}}^{n_k}(y_k)
\]	 
with nonnegative Lyapunov exponent. Consider now the periodic measure $\mu_k^+$ supported on the periodic orbit of $(\eta^{(k)},y_k)$ with nonnegative fiber Lyapunov exponent. By construction, $\nu_k=\pi_\ast\mu_k^+$. Analogously, using the nonpositive exponent-periodic points provided by Remark \ref{rem:magenta}, we get an ergodic measure $\mu_k^-$ with nonpositive Lyapunov exponent and satisfying $\nu_k=\pi_\ast\mu_k^-$. Note that both measures $\mu_k^\pm$ may coincide, in which case the exponent is zero.
Consider weak$\ast$ accumulation measures $\mu^\pm$ of $\mu_k^\pm$ as $k\to\infty$. Observe that, by construction, $\pi_\ast\mu^\pm=\nu$. Clearly, $\mu^+$ has a nonnegative ($\mu^-$ has a nonpositive) Lyapunov exponent. 

\emph{A priori}, the measures $\mu^\pm$ are not ergodic, but each one will have some ergodic component having the claimed properties (recall Remark \ref{movethere}). This proves the lemma.
\end{proof}

%------------------------------------------------------------------------------------------------------
\subsection{Projection, disintegration, and twin-measures}\label{sec:prodis}
%------------------------------------------------------------------------------------------------------

\begin{lemma}\label{lem:atmost2}
	For every $\nu\in\cM_{\rm erg}(\Sigma)$, there are at least one and at most two ergodic measures
$\mu_1,\mu_2\in\cM_{\rm erg}(\Gamma)$ with $\pi_\ast (\mu_1)= \nu =\pi_\ast (\mu_2)$.
\end{lemma}

\begin{proof}
Given $\nu\in\cM_{\rm erg}(\Sigma)$, let $\xi$ be some $\nu$-generic sequence. Given $x\in I_\xi$, let $X=(\xi,x)$ and consider a weak$\ast$ accumulation measure 
\[
	\mu=\lim_{n_\ell\to\infty}
		\frac{1}{n_\ell}(\delta_{X}+\delta_{F(X)}+\ldots+\delta_{F^{n_\ell-1}(X)}), 
\]
which is an $F$-invariant probability measure. By Remark \ref{movethere}, every ergodic component $\mu'$ of $\mu$  satisfies $\nu=\pi_\ast\mu'$. This proves that there is at least one ergodic $F$-invariant measure projecting to $\nu$. 

To show that there are no more than two ergodic measure  projecting to $\nu$, assume by contradiction that there are (at least) three such measures, say $\mu_i$ satisfying $\pi_\ast\mu_i=\nu$, $i=1,2,3$. Observe that we can choose points $(\xi^{(i)},x_i)$ which are generic for $\mu_i$, $i=1,2,3$ respectively, and satisfy $\xi^{(1)}=\xi^{(2)}=\xi^{(3)}\eqdef\xi$. Hence
\[
	\lim_{n\to\infty}\frac1n\sum_{k=0}^{n-1}\delta_{F^k(\xi, x_i)} 
	= \mu_i.
\]
Up to relabelling, we can assume that $x_1<x_2<x_3$.
Denote $x_i^n\eqdef f_\xi^n(x_i)$, $i=1,2,3$. By monotonicity, we have $x_1^n<x_2^n<x_3^n$ for every $n$. The following two cases can occur:
\begin{enumerate}
\item[(1)]	$\sum_{k=0}^\infty (x_3^k - x_1^k)=\infty$,
\item[(2)]	$\sum_{k=0}^\infty (x_3^k - x_1^k)<\infty$.
\end{enumerate}

In Case (1), by Lemma \ref{lem:monoto}, we have $\lim_n(x_3^n-x_2^n)=0$. Since $(\xi,x_2)$ and $(\xi,x_3)$ are generic points within the common fiber, it follows $\mu_2=\mu_3$.

In Case (2), the sum of the series being finite immediately implies that $\lim_n(x_3^n- x_1^n)=0$ and thus $\lim_n(x_2^n- x_1^n)=0$. Hence, arguing as before, we get $\mu_3=\mu_1$ and $\mu_2=\mu_1$ and hence the three measures $\mu_1,\mu_2$, and $\mu_3$ coincide.
\end{proof}

\begin{lemma}\label{pro:simple}
	Assume that $(\mu_n)_n\subset\cM(\Gamma)$ is a sequence satisfying 
\begin{itemize}
\item $\lim_n\pi_\ast\mu_n=\nu\in\cM_{\rm erg}(\Sigma)$,
\item $\pi^{-1}_\ast\nu$ has just one element $\mu\in\cM_{\rm erg}(\Gamma)$,
\end{itemize}
then $\lim_n\mu_n=\mu$.
\end{lemma}

\begin{proof}
	Consider some subsequence $(n_i)_i$ such that $(\mu_{n_i})_i$ weak$\ast$ converges to some measure $\tilde\mu\in\cM(\Gamma)$. Since  $(\pi_\ast\mu_n)_n$ converges to $\nu$, this is also true for its subsequence $(\pi_\ast\mu_{n_i})_i$.  By continuity of $\pi_\ast$, we obtain $\pi_\ast\tilde\mu=\nu$. As $\pi_\ast^{-1}\nu$ is just one element, we have $\tilde\mu=\mu$. Since the subsequence was arbitrary, we conclude weak$\ast$ convergence.
\end{proof}

%------------------------------------------------------------------------------------------------------
\subsection{``Distance" between twin-measures}\label{sec:distwimea}
%------------------------------------------------------------------------------------------------------

\begin{lemma}
	For every $\nu\in\cM_{\rm erg}(\Sigma)$ and every ergodic measures $\mu_1,\mu_2$ satisfying $\pi_\ast\mu_i=\nu$, $i=1,2$, we have
\[
	M^{-1}
	\le \frac{\chi(\mu_1)-\chi(\mu_2)}{-\int x\,d\mu_1+\int x\,d\mu_2}
	\le M.
\]
\end{lemma}

\begin{proof}
As in the proof of Lemma \ref{lem:atmost2}, we choose $\mu_i$-generic points $(\xi,x_i)$, $i=1,2$,  in a common fiber. Up to relabelling, we can assume $x_1<x_2$. Then, by ergodicity, we obtain
\[
	\lim_{n\to\infty}
	\frac{\frac1n\log(f_\xi^{n})'(x_1)-\frac1n\log(f_\xi^{n})'(x_2)}
		{\frac1n\sum_{k=0}^{n-1} f_\xi^{k}(x_2)
		-\frac1n\sum_{k=0}^{n-1} f_\xi^{k}(x_1)}
	= \frac{\chi(\mu_1)-\chi(\mu_2)}
		{\int x\,d\mu_2-\int x\,d\mu_1}.
\]
By applying Lemma \ref{lem:distortion}, we conclude the proof.
\end{proof}

%------------------------------------------------------------------------------------------------------
\subsection{Hyperbolicity of measures -- Proof of Theorem \ref{teo:1}}	\label{sec:proofTheorem1}
%------------------------------------------------------------------------------------------------------

Let $\nu\in\cM_{\rm erg}(\Sigma)$.
By Lemma \ref{lem:atmost2} there are at least one and at most two ergodic measures projecting to $\nu$ that are precisely Cases a) and b), respectively, claimed in the theorem: 

\medskip
\noindent \textbf{Case a).}
There are two measures $\mu_1,\mu_2\in\cM_{\rm erg}(\Gamma)$ such that 
$\pi_\ast\mu_1 = \nu = \pi_\ast\mu_2$.
As in the proof of Lemma \ref{lem:atmost2}, we choose $\mu_i$-generic points $(\xi,x_i)$, $i=1,2$,  in a common fiber. Up to relabelling, we can assume $x_1<x_2$. As $\mu_1\ne\mu_2$, their Wasserstein distance $D\eqdef W_1(\mu_1,\mu_2)$ is positive. Note that, by monotonicity of the fiber maps, we can invoke Lemma \ref{lem:Wassx} and obtain 
\[
	D
	= \int x \,d\mu_2(\xi,x) - \int x \,d\mu_1(\xi,x).
\]

Let $I^0=[x_1,x_2]$ and
\[
	I^n
	\eqdef f_\xi^n(I^0)
	= [x_1^n,x_2^n],
	\quad\text{ where }\quad
	x_1^n
	\eqdef f_\xi^n(x_1),\quad
	x_2^n
	\eqdef f_\xi^n(x_2).
\]	
As we have chosen generic points within the same fiber, there is an infinite sequence of positive integers $n$ such that $\lvert I^n\rvert\ge D/3>0$. Denote by $J=J(D)$ the infinite set of such indices, 
\[
	J(D)
	\eqdef \Big\{n\ge1\colon\lvert I^n\rvert\ge \frac D3\Big\}.
\]

\begin{claim}\label{cla:posdens}
	The set $J(D)$ has density at least $D/3$, that is,
\[
	\liminf_{n\to\infty}\frac1n\card\{k\in\{0,\ldots,n-1\}\colon k\in J(D)\}
	\ge \frac D3.
\]	
\end{claim}

\begin{proof}
	Since $(\xi,x_i)$, $i=1,2$, are generic, we have 
\[
	D
	= W_1(\mu_1,\mu_2)
	= \lim_{n\to\infty}\frac1n\sum_{k=0}^{n-1}(x_2^k-x_1^k)
	= \lim_{n\to\infty}\frac1n\sum_{k=0}^{n-1}\lvert I^k\rvert.
\]	
Hence, for every $n$ sufficiently large we have
\[
	 \frac1n\sum_{k=0}^{n-1}\lvert I^k\rvert
	\ge \frac23D.
\]
Hence, writing $J_n\eqdef J(D)\cap[0,n]$ and denoting by $\card J_n$ its cardinal, we can conclude that for every such $n$
\[
	n\frac23D
	\le \sum_{k=0}^{n-1}\lvert I^k\rvert
	< \card J_n\cdot 1+ (n-\card J_n) \frac D3.
\]
Therefore,
\[
	\frac 1n\card J_n
	> \frac13D,
\]
proving the claim.
\end{proof}

To estimate the exponent of $\mu_1$, recalling that $(\xi,x_1)$ is $\mu_1$-generic (guaranteeing the existence of the first limit below) and recalling the definition of the rescaled maps $g_k^\ell$ in \eqref{eq:defgkl} and using \eqref{eq:equal}, we have
\begin{equation}\label{eq:thishere}
	\chi(\mu_1)
	= \lim_{n\to\infty}\frac1n\log(f_\xi^n)'(x_1)
	= \lim_{n\to\infty}\frac1n\log\left( (g_0^n)'(x_1)
		\frac{\lvert I^n\rvert}{\lvert I^0\rvert}\right).
\end{equation}

\begin{claim}\label{cla:rescalgest}
Taking $C_1( D/3)>0$ as in Lemma \ref{lem:obv}, we have
\[
	(g_k^{k+1})'(x_1^k)
	\begin{cases}
	\ge e^{ C_1(D/3)}&\text{ if }k\in J(D),\\
	\ge 1&\text{ otherwise}.
	\end{cases}
\]
\end{claim}

\begin{proof}
	If $k\in J(D)$, then using $\lvert I^k\rvert\ge D/3$ and applying Lemma \ref{lem:obv} to $I^k$, we obtain
\[
	e^{ C_1(D/3)}
	\le \frac{\lvert I^k\rvert}{\lvert I^{k+1}\rvert} f_{\xi_{k+1}}'(x_1^k)
	= (g_k^{k+1})'(x_1^k).
\]	
Note that for every $k$ we have $(g_k^{k+1})'(x_1^k)\ge1$, recall \eqref{eq:ofcourse}.
\end{proof}

Recall that $(g_0^n)'(x_1)=(g_0^1)'(x_1)(g_1^2)'(x_1^1)\cdots(g_{n-1}^n)'(x_1^{n-1})$. Thus, Claims \ref{cla:posdens} and \ref{cla:rescalgest} together imply 
\[
	\lim_{n\to\infty}
	\frac1n\log(g_0^n)'(x_1)
	\ge \lim_{n\to\infty}\frac1n\log\left((e^{C_1(D/3)})^{\card J_n}\cdot 1^{n-\card J_n}\right)
	\ge  C_1(\frac D3)\cdot \frac D3. 
\]
With \eqref{eq:thishere},  we have
\[\begin{split}
	\chi(\mu_1)
	&= \lim_{n\to\infty} \frac1n\log\,\left( (g_0^n)'(x_1)
		\frac{\lvert I^n\rvert}{\lvert I^0\rvert}\right)
	\ge \lim_{n\to\infty, n\in J}\frac1n\log\,\left( (g_0^n)'(x_1)
		\frac{D/3}{\lvert I^0\rvert}\right)\\
	&\ge \frac D3\cdot  C_1(\frac D 3).
\end{split}\]
Now we take $\kappa_1(D)\eqdef  D/3\cdot  C_1(D/3)$.

The analogous argument for the inverse fiber maps implies that $\chi(\mu_2)\le -\kappa_2(D)<0$ where $\kappa_2(D)= D/3 C_2( D/3)$, and hence completing the proof of Case a). 

\medskip
\noindent \textbf{Case b).}
There is only one ergodic measure $\mu$ such that $\pi_\ast\mu=\nu$. This implies $\nu\ne\delta_{0^\bZ}$ and hence any $\nu$-generic sequence $\xi=\xi^-.\xi^+$ is such that $\xi^+$ contains infinitely many $1$'s.  
To prove the claim, arguing by contradiction, assume that $\chi(\mu)\ne0$, say $\chi(\mu)>0$ (the case $\chi(\mu)>0$ is analogous).
Given $\varepsilon\in(0,\chi(\mu))$, there is a (forward) generic point $(\xi,x)$ so that there exists $n_0\ge1$ such that for every $n\ge n_0$ we have
\begin{equation}\label{eq:whichimplies}
	(f_\xi^n)'(x)\ge e^{n\varepsilon}.
\end{equation}
Note that by the previous comment, we can assume $x\in(0,1)$.
Take any $y\in(x,1]$. Notice that $y\in I_{\xi^+}$ (see Remark \ref{rem:onesidspin}). Considering the sequence of orbital measures $\mu_n$ uniformly distributed on $\{(\xi,y),F(\xi,y),\ldots, F^{n-1}(\xi,y)\}$, there is some subsequence which weak$\ast$ converges to some $F$-invariant measure $\mu'$. Notice that $\pi_\ast\mu'=\pi_\ast\mu$. Hence, by hypothesis, $\mu'=\mu$. In particular, it is unnecessary to consider a subsequence and thus $(\xi,y)$ is in fact $\mu$-forward generic. Thus, taking the limit of Birkhoff sums of the function $(\xi,s)\mapsto s$, we obtain
\[
	\lim_{n\to\infty}\frac1n\sum_{k=0}^{n-1}(f_\xi^k(y)-f_\xi^k(x))
	=0.
\] 
For sufficiently large $n$, by using \eqref{eq:whichimplies} we obtain
\[
	\sum_{k=0}^{n-1}(f_\xi^k(y)-f_\xi^k(x))
	\le n\frac{\varepsilon}{2M}.
\]
Thus, with Lemma \ref{lem:distortion} 
\[
	\max_{u,w\in[x,y]}\lvert \log (f_\xi^n)'(u)-\log (f_\xi^n)'(w)\rvert
	\le M\sum_{k=0}^{n-1}(f_\xi^k(y)-f_\xi^k(x))
	\le Mn\frac{\varepsilon}{2M}
	=n\frac\varepsilon2.
\]
Therefore, for every $z\in[x,y]$, with \eqref{eq:whichimplies} we obtain $(f_\xi^n)'(z) >e^{n\varepsilon/2}$. Hence,
\[
	\lvert f_\xi^n([x,y])\rvert
	 \to\infty
\]
as $n\to\infty$, which is a contradiction. 

This completes the proof of Case b) and hence the proof of Theorem \ref{teo:1}.
\qed

\begin{remark}
	Notice that in the proof of Theorem \ref{teo:1}  Case b), the $\mu$-generic point is of the form $(\xi,x)$ with $x\ne1$. This allowed as to take $y\in(x,1]$. Note also that the only measure having a generic point of the form $(\xi,1)$ is the Dirac measure at $P=(0^\bZ,1)$. This will no longer be true in the setting of Section \ref{sec:bif} studying bifurcation scenarios. This is precisely the place where another measure will  ``appear", see Proposition \ref{p.FSTmeasures} Case c).
\end{remark}	 

%------------------------------------------------------------------------------------------------------
\subsection{Proof of Theorem \ref{teo:2}} \label{sec:proof:teo2}
%------------------------------------------------------------------------------------------------------

Let $\mu\in\cM_{\rm erg}(\Gamma)$ and $\nu=\pi_\ast\mu$. Observe that, by Remark \ref{movethere}, $\nu$ is ergodic. By invariance of the disjoint subsets $\Sigma^{\rm sing}$ and $\Sigma^{\rm spine}$ in \eqref{eq:singspine}, $\nu$ is supported  on one of them only.  

If $\nu(\Sigma^{\rm sing})=1$ then $I_\xi=\{x_\xi\}$ is a singleton $\nu$-almost everywhere. Hence, by disintegration, $\mu_\xi=\delta_{x_\xi}$ $\nu$-almost everywhere. In particular, there is no other measure projecting to $\nu$. Hence, by Theorem \ref{teo:1} b), $\mu$ is nonhyperbolic.  
	
Otherwise, $\nu(\Sigma^{\rm spine})=1$ and  $I_\xi=[x_{\xi^+},x_{\xi^-}]$, where $x_{\xi^+}<x_{\xi^-}$ $\nu$-almost everywhere.  Consider the measures
\[
	\mu^\pm
	\eqdef \int_\Sigma\delta_{x_{\xi^\pm}}\,d\nu(\xi).
\]
It is clear that both are $F$-invariant. Since $x_{\xi^+}<x_{\xi^-}$ almost everywhere, we have $\mu^+\ne\mu^-$.  What remains to see is that these measures are ergodic and of the claimed type of hyperbolicity. Observe that for every $\mu'\in\cM(\Gamma)$ satisfying $\pi_\ast\mu'=\nu$ we have
\[\begin{split}
	\int x\,d\mu^+(\xi,x)
	&= \int x_{\xi^+}\,d\nu(\xi)
	\le \int x\,d\nu(\xi)
	\le \int x\,d\mu'(\xi,x)\\
	&\le \int x_{\xi^-}\,d\nu(\xi)
	= \int x\,d\mu^-(\xi,x).
\end{split}\]
Hence, the measures $\mu^\pm$ are extremal points in the subspace $\{\mu'\colon\mu'\in\cM(\Gamma),\pi_\ast\mu'=\nu\}$. By Remark \ref{movethere}, any ergodic component of $\mu^\pm$ also projects to $\nu$.  Thus $\mu^\pm$ cannot have a nontrivial ergodic decomposition. Hence $\mu^\pm$ both are $F$-ergodic. 
By Theorem \ref{teo:1} a), $\mu^\pm$ are hyperbolic with opposite type of hyperbolicity and there are no further ergodic measures projecting to $\nu$. Hence, $\mu\in\{\mu^+,\mu^-\}$. 
This proves Theorem \ref{teo:2}.
\qed

%------------------------------------------------------------------------------------------------------
\subsection{Frequencies}
%------------------------------------------------------------------------------------------------------

We conclude with some consequences of Theorem~\ref{teo:1}. They will be used in Section \ref{sec:bif} when analyzing explosion of entropy and of the space of ergodic measure in bifurcation scenarios. 

\begin{lemma}\label{lem:exponentextreme}
Assume  (H1)--(H2). For every $\nu\in\cM_{\rm erg}(\Sigma)$  we have
\[
	\nu([0]) \, \log f_0'(1) + \nu([1]) \, \log f_1'(1) 
	\le 0.
\]	
\end{lemma}

\begin{proof}
Let $\nu\in\cM_{\rm erg}(\Sigma)$. Arguing by contradiction, suppose that the statement is false. Any ergodic measure $\mu$ projecting to $\nu$ satisfies
\[
	\begin{split}
	\chi(\mu)
	&= \int \log f_{\xi_0} '(x) d\mu (\xi,x) 
	\ge  \int \log f_{\xi_0}' (1) d\mu (\xi,x)\\
	&=\nu([0]) \, \log f_0'(1) + \nu([1]) \, \log f_1'(1) >0.
\end{split}
\]
But this contradicts Lemma \ref{lem:existsemihypmeas} which guarantees the existence of an ergodic measure projecting to $\nu$ with nonpositive Lyapunov exponent. 
\end{proof}

Given $\xi\in\Sigma$ and $a\in\{0,1\}$, for natural numbers $n< m$ we define 
\[
	\freq_n^m(\xi,a)
	\eqdef\frac{1}{m-n+1}\card\{k\in\{n,\ldots,m\}\colon \xi_k=a\}.
\]
Let
\begin{equation}\label{eq:deffreq}
	\overline\freq(\xi,a)
	\eqdef \limsup_{n\to-\infty,m\to\infty}\freq_n^m(\xi,a)
\end{equation}
and define $\underline\freq(\xi,a)$ analogously taking $\liminf$ instead of $\limsup$.
Observe that those functions are measurable and $\sigma$-invariant. Hence, for every $\nu\in\cM_{\rm erg}(\Sigma)$ for $\nu$-almost every $\xi$ we have 
\[
	\overline\freq(\xi,a)
	= \underline\freq(\xi,a)
	= \nu([a]).
\]	

\begin{corollary}\label{cor:exponentextreme-xi}
Assume (H1)--(H2).	For every $\xi\in\Sigma$ we have
\[
	\underline\freq(\xi,0)\log f_0'(1)+\overline\freq(\xi,1)\log f_1'(1)\le 0.
\]	
\end{corollary}

\begin{proof}
	Given $\xi$, consider subsequences $(n_i)_i$ and $(m_i)_i$ such that
\[
	\lim_{i\to\infty}\freq_{n_i}^{m_i}(\xi,0)
	= \underline\freq(\xi,0)
\]	
and consider the probability measures 
\[
	\frac{1}{m_i-n_i+1}\sum_{k=n_i}^{m_i}\delta_{\sigma^k(\xi)}.
\]
Then any weak$\ast$ accumulation measure $\nu$ of those measures is $\sigma$-invariant and satisfies $\nu([0])=\underline{\freq}(\xi,0)$. Observe that 
\[
	\nu([1])
	= \overline\freq(\xi,1)
	= 1- \underline\freq(\xi,0).
\]	
Now it suffices to consider the ergodic decomposition of $\nu$ and apply Lemma \ref{lem:exponentextreme}.
\end{proof}

%------------------------------------------------------------------------------------------------------
\section{Accumulations of ergodic measures in $\cM(\Gamma)$: Proof of Theorem \ref{teo:accum}} \label{sec:accum}
%------------------------------------------------------------------------------------------------------

In the entire section we will assume (H1)--(H2). 
We will indicate separately places which require also (H2+). Indeed, we need (H2+) when invoking Lemma \ref{pro:simple} at the very end of the proof of Theorem  \ref{teo:accum}.  Section \ref{sec:perapperg} is essentially only about weak$\ast$ approximation, the remainder of this section discusses approximation also in entropy. One essential step in the proof are so-called skeletons, defined and discussed in Section \ref{sec:skeletons}.

%------------------------------------------------------------------------------------------------------
\subsection{Periodic approximation of ergodic measures}\label{sec:perapperg}
%------------------------------------------------------------------------------------------------------

\begin{proposition}[Density of hyperbolic periodic measures]\label{ppp.periodicapproximation}
	Assume (H1)--(H2+). Every measure in $\cM_{\rm erg}(\Gamma)$ is weak$\ast$ accumulated by hyperbolic periodic measures in $\cM_{\rm erg}(\Gamma)$. \end{proposition}
	
For the above result in the case when the measure is hyperbolic see Remark \ref{Katokforevery}. We will provide a proof, building upon the concavity hypotheses, which has intrinsic interest. It applies to any (also nonhyperbolic) ergodic measure $\mu\not\in\{\delta_P,\delta_Q\}$.  The simple cases $\mu\in\{\delta_P,\delta_Q\}$ we check separately in Lemma \ref{lem:diractmeasuress}. 

\begin{corollary}\label{corppp.periodicapproximation}
	Assume (H1)--(H2).
	Every measure in $\cM_{\rm erg}(\Sigma)$ is weak$\ast$ accumulated by periodic measures in $\cM_{\rm erg}(\Sigma)$.
\end{corollary}	

Assuming (H1)--(H2+), the above corollary is a consequence of Proposition \ref{ppp.periodicapproximation} together with the fact that for every $\nu\in\cM_{\rm erg}(\Sigma)$ there exists $\mu\in \cM_{\rm erg}(\Gamma)$ with $\pi_\ast\mu=\nu$, see Lemma \ref{lem:atmost2}.
The general case, assuming only (H1)--(H2), we prove at the end of this section.
 
\begin{proof}[Proof of Proposition \ref{ppp.periodicapproximation}]
We first deal with the simplest case.

\begin{lemma}\label{lem:diractmeasuress}
	The measure $\delta_Q$ (the measure $\delta_P$) is weak$\ast$ accumulated by periodic measures.
\end{lemma}

\begin{proof}
	Note that the sequence $\xi^{(n)}=(0^n1)^\bZ$ is admissible for any $n$ sufficiently large (see Remark \ref{rem:someseq}). Moreover, we can apply Lemma \ref{lem:fixpots} to obtain two fixed points $p_n^+<p_n^-$ for the map $f_{[0^n1]}$.  Observe that for every $z\in(0,f_1(1))$ there exists $n_0$ such that for every $n\ge n_0$ we have $(f_1\circ f_0^n)(z)>z$ and hence $p_n^+\le z\le p_n^-$. This implies that $\lim_{n\to\infty}p_n^+=0$. Moreover,  we have  
\[
	q_n^+
	\eqdef f_0^n(p_n^+)
	= f_1^{-1}(p_n^+)
	\to d=f_1^{-1}(0).
\]	 
 As the periodic orbit of the point $((0^n1)^\bZ,p_n^+)$ (projected to $[0,1]$) consists of the points $q_n^+$, $f_0^{-1}(q_n^+)$, $\ldots$, $f_0^{-n}(q_n^+)$, almost all of them stay close to $0$. More precisely, given any $z\in(0,1)$, there exists $n_1=n_1(z)$ such that for every $n\ge n_1$ we have $f_0^{-n}(d)<z$.  As $z>0$ can be chosen arbitrarily close to $0$, in this way we construct a sequence $Q_n=(\xi^{(n)},p_n^+)$ of periodic points whose $F$-invariant probability measures supported on its orbit converge weak$\ast$ to $\delta_Q$.
	
	The analogous arguments apply to $\delta_P$ considering $F^{-1}$.
\end{proof}

We collect some preparatory results. Recall that $\varrho(\xi,x)=x$ is the canonical projection to the second coordinate. 

\begin{remark}\label{rem:projmu0}
	For every $\mu\in\cM_{\rm erg}(\Gamma)$, $\mu\not\in\{\delta_P,\delta_Q\}$, there exists $a\in(0,1)$ such that $\varrho_\ast\mu(J)>0$ for both intervals $J=(0,a)$ and $J=(a,1)$. 
\end{remark}

\begin{lemma}\label{lemcla:13}
	For every $\mu\in\cM_{\rm erg}(\Gamma)$, $\mu\not\in\{\delta_P,\delta_Q\}$, there exist $\mu$-generic points $(\xi,x)$ such that there exists infinitely many times $n\ge1$ with $f_\xi^n(x)>x$.
\end{lemma}	

\begin{proof}
	By Remark \ref{rem:projmu0}, there exists $a\in(0,1)$ such that $\varrho_\ast\mu(J)>0$ for $J=(0,a)$ and $J=(a,1)$. Hence, there exists a $\mu$-generic point $R=(\xi,x)$ in $\Sigma\times(0,a)$. Then, by Poincar\'e recurrence, the orbit of $R$ by $F$ has infinitely many return times $n\ge1$ to $\Sigma\times(a,1)$ and therefore satisfies $f_\xi^n(x)>a>x$, proving the lemma. 
\end{proof}

\begin{lemma}\label{lem:dontknow}
	Let $\mu\in\cM_{\rm erg}(\Gamma)$ and $(\xi,x)$ be any $\mu$-generic point  such that there are infinitely many times $n_i\ge1$ with $f_\xi^{n_i}(x)>x$. Then for every $i\ge1$ the sequence $\xi^{(n_i)}=(\xi_0\ldots\xi_{n_i-1})^\bZ$ is admissible and there exist repelling and contracting points
\[
	p^+_{[\xi_0\ldots\,\xi_{n_i-1}]}
	= f_{[\xi_0\ldots\,\xi_{n_i-1}]}(p^+_{[\xi_0\ldots\,\xi_{n_i-1}]}) 
	<x
	< p^-_{[\xi_0\ldots\,\xi_{n_i-1}]}
	= f_{[\xi_0\ldots\,\xi_{n_i-1}]}(p^-_{[\xi_0\ldots\,\xi_{n_i-1}]}).
\] 
Let 
\[
	P^\pm_{n_i}
	\eqdef ((\xi_0\ldots\xi_{n_i-1})^\bZ,p^\pm_{[\xi_0\ldots\,\xi_{n_i-1}]})
\]	 
and consider the sequence of ergodic measures $(\mu_{n_i}^\pm)_{n_i}$ supported on their orbits 
\[
	\mu^\pm_{n_i}
	= \frac1{n_i}(\delta_{P^\pm_n}+\delta_{F(P^\pm_n)}+\ldots+\delta_{F^{n-1}(P^\pm_n)}).
\]	 
Then we have that 
\[
	\lim_{i\to\infty}\pi_\ast\mu^\pm_{n_i}
	= \nu
	\eqdef \pi_\ast\mu.
\]
Moreover, any weak$\ast$ accumulation point $\mu^+$ of $(\mu_{n_i}^+)_i$ and any weak$\ast$ accumulation point $\mu^-$ of $(\mu_{n_i}^-)_i$ satisfies
\begin{equation}\label{eq:numeeee}
	\chi(\mu^+)
	\ge 0
	\ge \chi(\mu^-)
	\quad\text{ and }\quad
	\chi(\mu^+)
	\ge \chi(\mu)
	\ge \chi(\mu^-).
\end{equation}
Finally, there are ergodic components $\tilde\mu^\pm$ of $\mu^\pm$, respectively, such that $\pi_\ast\tilde\mu^\pm=\nu$ and 
\[
	\chi(\tilde\mu^+)
	\ge 0
	\ge \chi(\tilde\mu^-)
	\quad\text{ and }\quad
	\chi(\tilde\mu^+)
	\ge \chi(\mu)
	\ge \chi(\tilde\mu^-).
\]
\end{lemma}

\begin{proof}
Let $[a_i,1]$ be the admissible interval for the map $f_{[\xi_0\ldots\,\xi_{n_i-1}]}$. We have $f_{[\xi_0\ldots\,\xi_{n_i-1}]}(a_i)=0$ and $f_{[\xi_0\ldots\,\xi_{n_i-1}]}(x)>x$. Hence this map has at least one fixed point, which cannot be parabolic as its graph crosses the diagonal. Hence Lemma \ref{lem:fixpots} case (1a) proves the first claim providing the repelling point $p^+_{[\xi_0\ldots\,\xi_{n_i-1}]}$ and the contracting point $p^-_{[\xi_0\ldots\,\xi_{n_i-1}]}$. Moreover, the choices of the times $n_i$ prove the inequalities for $x$.

	 Let $\nu=\pi_\ast\mu$. Observe that $\xi$ is $\nu$-generic and that the sequence of ergodic measures $(\nu_{n_i})_i$ supported on the periodic sequences $\xi^{(n_i)}$ converges in the weak$\ast$ topology to $\nu$. 
	Consider the measures $\mu^\pm_{n_i}$ in the statement of the lemma. Observe that $\pi_\ast\mu^+_{n_i}=\nu_{n_i}^+$ converges in the weak$\ast$ topology to $\nu$. 
By the first part of the lemma, $P^+_{n_i}$ is of expanding type and $P^-_{n_i}$ of contracting type. This implies the first part of \eqref{eq:numeeee}. The second part  of \eqref{eq:numeeee} follows from concavity and the relative position of $x$. 

Finally, by Remark \ref{movethere}, $\nu$ is ergodic and any ergodic component $\tilde\mu^+$ of $\mu^+$ also satisfies $\pi_\ast\tilde\mu^+=\nu$. By the ergodic decomposition of $\mu^+$, there is one component with nonnegative exponent and exponent not smaller than $\chi(\mu)$. The argument for $\mu^-$ is analogous.
\end{proof}

\begin{lemma}\label{lempro:0denseperio}
	Assume (H1)--(H2+).
	Every $\mu\in\cM_{\rm erg,0}(\Gamma)$ is the weak$\ast$ limit of a sequence of periodic measures in $\cM_{\rm erg,>0}(\Gamma)$. Analogously, there is a sequence of periodic measures in $\cM_{\rm erg,<0}(\Gamma)$ converging weak$\ast$ to $\mu$. 
\end{lemma}

\begin{proof}
	Given $\mu\in\cM_{\rm erg,0}(\Gamma)$, clearly $\mu\not\in\{\delta_P,\delta_Q\}$. Let $\nu=\pi_\ast\mu\in\cM_{\rm erg}(\Sigma)$. By applying Lemma \ref{lemcla:13}, we have that the hypotheses of Lemma \ref{lem:dontknow} are satisfied and hence there is a sequence $(\mu^\pm_{n_i})_i$ of periodic measures such that $\pi_\ast\mu^\pm_{n_i}=\nu_{n_i}^\pm$ converges to $\nu$. By Theorem \ref{teo:1} b), $\mu$ is the only $F$-invariant (ergodic) measure projecting to $\nu$.  Hence, by Lemma \ref{pro:simple} the sequences $(\mu^\pm_{n_i})_i$ both weak$\ast$ converge to $\mu$. 
\end{proof}

\begin{lemma}\label{lem:finalpart}
	Every $\mu\in\cM_{\rm erg,>0}(\Gamma)$ is the weak$\ast$ accumulation point of a sequence of periodic measures $(\mu_n)_n\subset\cM_{\rm erg,>0}(\Gamma)$. Analogously for measures in $\cM_{\rm erg,<0}(\Gamma)$.
\end{lemma}

\begin{proof}
The case $\mu=\delta_Q$ is just Lemma \ref{lem:diractmeasuress}. For $\mu\ne\delta_Q$, by Lemma \ref{lemcla:13} the hypotheses of Lemma \ref{lem:dontknow} are satisfied and hence there is a sequence of periodic measures $(\mu^+_{n_i})_i$ converging to some $F$-invariant measure $\tilde\mu$. 
Clearly $\chi(\tilde\mu)\ge0$ and $\pi_\ast\tilde\mu=\pi_\ast\mu=\nu$. 
By Lemma \ref{lem:dontknow}, indeed we have $\chi(\tilde\mu)\ge\chi(\mu)>0$. 
We claim that $\tilde\mu=\mu$.

If $\tilde\mu$ is ergodic, then by Theorem \ref{teo:1} a) we have $\tilde\mu=\mu$. 
	Otherwise, if $\tilde\mu$ is not ergodic, by Theorem \ref{teo:1} a), then there is exactly one further measure $\mu'\in\cM_{\rm erg,<0}(\Gamma)$ also projecting to $\nu$ such that $\tilde\mu=\alpha\mu+(1-\alpha)\mu'$ for some $\alpha\in(0,1)$ and hence $\chi(\tilde\mu)<\chi(\mu)$, a contradiction.
\end{proof}

  The proposition now follows from Lemmas \ref{lem:diractmeasuress}, \ref{lempro:0denseperio}, and \ref{lem:finalpart}.
\end{proof}

\begin{proof}[Proof of Corollary \ref{corppp.periodicapproximation} assuming only (H1)--(H2)]
Recall again that for every $\nu\in\cM_{\rm erg}(\Sigma)$ there exists $\mu\in \cM_{\rm erg}(\Gamma)$ with $\pi_\ast\mu=\nu$, see Lemma \ref{lem:atmost2}.
Replacing Lemma \ref{lempro:0denseperio} by Theorem \ref{teopro:intermediate} (which only requires (H1)--(H2)), we conclude the proof.
\end{proof}

%------------------------------------------------------------------------------------------------------
\subsection{Skeletons}\label{sec:skeletons}	
%------------------------------------------------------------------------------------------------------

In this section we first collect some ingredients necessary to prove Theorem \ref{teopro:intermediate}. They are somewhat similar to the ones in \cite[Section 3.1]{DiaGelRam:17}, though here we essentially rely on concavity arguments and neither on minimality nor one expanding/contracting itineraries. The focus in this section is on nonhyperbolic measures, only.  

In this section we adopt the following notation. Given $k\ge1$ and $\xi\in\Sigma$, consider the following probability measure put on the orbit of $\xi$
\begin{equation}\label{eq:defcA}
	 \cA_k\xi
	 \eqdef \frac1k(\delta_{\xi}+\delta_{\sigma(\xi)}+\ldots+\delta_{\sigma^{k-1}(\xi)}).
\end{equation}
The analogous notation we use for one-sided sequences.

\begin{definition}[Skeleton$\ast$ property]\label{def:skeleton}
	By a measure $\mu\in\cM_{\rm erg,0}(\Gamma)$ \emph{having the skeleton$\ast$ property} we mean that for every $\varepsilon\in(0,h(\mu))$ there exist numbers $a\in(0,1)$, $C>0$, and $n_0\ge1$ such that for every $n\ge n_0$ there is $k\in\{n(1-\varepsilon),\ldots,n\}$ and there is a finite set $\sX=\sX(h,\varepsilon,a,C)=\{X_i\}$ of points $X_i=(\xi^{(i)},x_i)$ (a \emph{Skeleton} relative to $F,\mu,\varepsilon$) so that:
\begin{itemize}
\item[(i)] the set $\sX$ has cardinality 
\[
	\card\sX
	\ge \frac{C}{n\varepsilon}e^{n(1-\varepsilon)(h-\varepsilon)}, 
\]
\item[(ii)] the words $(\xi^{(i)}_0\ldots\xi^{(i)}_{k-1})$ are all different,
\item[(iii)]
$\displaystyle
	-k\varepsilon<\log\, (f_{\xi^{(i)}}^k)'(x_i)<k\varepsilon,
$
\item[(iv)] $x_i\in(0,a)$ and $f_{\xi^{(i)}}^k(x_i)\in(a,1)$,
\item[(v)] for $\nu=\pi_\ast\mu$ the Wasserstein distance $W_1$ in $\cM(\Sigma)$ satisfies  
\begin{equation}\label{eq:weakstar}
	W_1\Big(\cA_k\xi^{(i)},\nu\Big)
	<\varepsilon.
\end{equation}
\end{itemize}
\end{definition}

\begin{remark}[Relation with skeletons in other contexts]
All previously used definitions of so-called skeletons (see \cite[Section 4]{DiaGelRam:17} and \cite[Section 5.2]{DiaGelSan:19} have the following essential ingredients: (i) cardinality governed by entropy, (ii) spanning property, (iii) finite-time Lyapunov exponents close to zero, (iv) connecting times, and (v) proximity in the weak$\ast$ topology. In our context, item (iv) is clearly different as we are not using minimality properties but obtain connecting times from concavity. Moreover, item (v) is stated in terms of the Wasserstein distance. 
\end{remark}

\begin{proposition}\label{pro:skeleton}
	Assuming (H1), every $\mu\in\cM_{\rm erg,0}(\Gamma)$ has the skeleton$\ast$ property.
\end{proposition}

\begin{proof}
The proof will use some preliminary standard arguments that we borrow from \cite[Section 3.1]{DiaGelRam:17}. First note that, as by hypothesis we have $\mu\not\in\{\delta_P,\delta_Q\}$, by Remark \ref{rem:projmu0} there is $a\in(0,1)$ such that $\varrho_\ast\mu(J)>0$ for $J=(0,a)$ and $J=(a,1)$, where $\varrho$ is the projection to the second coordinate. Let $A\eqdef\Sigma\times(0,a)$ and $B\eqdef\Sigma\times(a,1]$. By the above, both sets have positive measure $\mu$. Choose $\varepsilon>0$ satisfying
\[
	\varepsilon
	<\min\{\mu(A)/4,\mu(B)/4\}.
\]

\begin{lemma}
	There are a set $A_1\subset\Gamma$ satisfying $\mu(A_1)>1-\varepsilon$ and a number $N_1\ge1$ such that for every $n\ge N_1$ and every $X\in A_1$ we have $F^\ell(X)\in \Sigma\times(a,1]$ for some $\ell\in\{n(1-\varepsilon)-1,\ldots, n\}$.
\end{lemma}

\begin{proof}
By ergodicity, recurrence, and Egorov's theorem, there are a set $A_1\subset\Gamma$ satisfying $\mu(A_1)>1-\varepsilon$ and a number $N_1'\ge1$ such that for every $n\ge N_1'$ and every $X\in A_1$ we have 
\[
	\left\lvert\frac1n\card\{\ell\in\{0,\ldots,n-1\}\colon F^k(X)\in B\}-\mu(B)\right\rvert
	\le \varepsilon^2.
\]
Now let $N_1\ge N_1'$ such that for every $n$ with $n(1-\varepsilon)\ge N_1$ we have
\begin{equation}\label{eq:stupoi}
	n\varepsilon(\mu(B)-2\varepsilon)
	>1.
\end{equation}
Hence, by the above, for every $X\in A_1$ and every $n$ with $n(1-\varepsilon)\ge N_1$
\[\begin{split}
	\card&\{\ell\in\{n(1-\varepsilon),\ldots,n\}\colon F^\ell(X)\in B\}\\
	&=\card\{\ell\in[0,n]\colon F^\ell(X)\in B\}
	-\card\{\ell\in[0,n(1-\varepsilon)-1]\colon F^\ell(X)\in B\}\\
	&\ge(\mu(B)-\varepsilon^2)n-(n(1-\varepsilon)-1)(\mu(B)+\varepsilon^2)\\
	&\ge n\varepsilon(\mu(B)-2\varepsilon)
	>1,
\end{split}\]
where in the last inequality we used \eqref{eq:stupoi}. This proves the lemma.
\end{proof}

\begin{lemma}\label{lcla:waassee}
	There are a set $A_2\subset\Gamma$ satisfying $\mu(A_2)> 1-\varepsilon$ and $N_2\ge1$ such that for every $n\ge N_2$ and every $X=(\xi,x)\in A_2$ and for every $\ell\in\{n(1-\varepsilon),\ldots, n\}$ it holds $W_1(\cA_\ell\xi,\pi_\ast\mu)<\varepsilon$.
\end{lemma}

\begin{proof}
	First, by Remark \ref{movethere}, $\pi_\ast\mu$ is ergodic.
	We now consider a countable dense set of continuous functions and to each of them apply Birkhoff's ergodic theorem to $\pi_\ast\mu$. Using again Egorov's theorem, we conclude the proof.
\end{proof}

\begin{lemma}\label{lcleq:genericderivative}
	There are $A_3\subset\Gamma$ satisfying $\mu(A_3)> 1-\varepsilon$ and  $N_3\ge1$ such that for every $X=(\xi,x)\in A_3$ and every $\ell\ge N_3$ we have 
\[
	- \ell\varepsilon
	<
	\log\,(f_\xi^\ell)'(x) 
	< \ell\varepsilon.
\]
\end{lemma}

\begin{proof}
	The lemma follows from Birkhoff's theorem using 
\[	
	\chi(\mu)
		=\int\log\, f_{\xi_0}'(x) \,d\mu(\xi,x)
		=0
\]	 
and then applying Egorov's theorem.
\end{proof}

\begin{lemma}\label{lcla:BriKat}
	There are $A_4\subset\Gamma$ satisfying $\mu(A_4)>1-\varepsilon$ and $N_4\ge1$ such that for every $n\ge N_4$ and every $X=(\xi,x)\in A_4$ it holds
\[
	e^{-n(h(\mu)+\varepsilon)}
	\le\nu([\xi_0\ldots\xi_{n-1}])
	\le e^{n(h(\mu)-\varepsilon)}.
\]	
\end{lemma}	

\begin{proof}
Note that by Remarks \ref{movethere} and \ref{eqlem:projent}, $\nu\eqdef\pi_\ast\mu$ is ergodic and $h(\nu)=h(\mu)$.
The lemma now follows from Brin-Katok's theorem \cite{BriKat:83} applied to $\mu$ and $\nu$  and then applying Egorov's theorem.
\end{proof}

By construction, the set $A'\eqdef A\cap A_1\cap A_2\cap A_3\cap A_4$ satisfies $\mu(A')>\mu(A)-4\varepsilon>0$. Now let
\[
	n_0
	\eqdef \max\{N_1,N_2,N_3,N_4\}.
\]	
Take $n\ge n_0$ and for every $\ell\in\{n(1-\varepsilon),\ldots,n\}$ denote by $A'(\ell)\subset A'$ the set of points $X\in A'$ such that $F^\ell(X)\in B$. 
Let $k$ be the index of a set $A'(k)$ with maximal measure. Hence, by the pigeonhole principle
\[
	\mu( A'(k))
	\ge \frac{\mu( A')}{n\varepsilon+1}
	>0.
\] 

Let $S'\eqdef\pi(A'(k))$ and observe $\nu(S')\ge\mu(A'(k))>0$. Choose any point $X_1=(\xi^1,x_1)\in A'(k)$. Let $S_1\eqdef S'\setminus[\xi^1_0\ldots\xi^1_{k-1}]$. We continue inductively: choose any $\xi^\ell\in S_{\ell-1}$, let $S_\ell\eqdef S_{\ell-1}\setminus[\xi^\ell_0\ldots\xi^\ell_{k-1}]$ and $X_\ell=(\xi^\ell,x_\ell)$ and repeat. As by Lemma \ref{lcla:BriKat} we have
\[
	\nu(S_\ell)
	\ge \nu(S')-\ell e^{-k(h(\mu)-\varepsilon)},
\]
we can repeat this process at least $m$ times, where
\[\begin{split}
	m
	&\ge \nu(S') e^{k(h(\mu)-\varepsilon)}
	\ge \mu(A'(k)) e^{k(h(\mu)-\varepsilon)}\\
	&\ge \frac{\mu(A')}{n\varepsilon+1} e^{k(h(\mu)-\varepsilon)}
	\ge \frac{\mu(A)-4\varepsilon}{n\varepsilon+1}
		e^{n(1-\varepsilon)(h(\mu)-\varepsilon)}
\end{split}\]
Taking $\sX\eqdef\{ X_i\}_{i=1}^M$, we obtain item (i) taking $C>0$ appropriately. Since, by construction, the cylinders $[\xi^{(i)}_0\ldots\xi^{(i)}_{k-1}]$ are pairwise disjoint, proving item (ii). As $k\ge N_3$, item (iii) now follows from Lemma \ref{lcleq:genericderivative}. Analogously, we have item (v) from Lemma \ref{lcla:waassee}. Finally, item (iv) follows from the very definition of $\sX$. 
\end{proof}

%------------------------------------------------------------------------------------------------------
\subsection{Weak$\ast$ and entropy approximation of ergodic measures: Proof of Theorem \ref{teopro:intermediate}}	
%------------------------------------------------------------------------------------------------------

	For any hyperbolic measure $\mu\in\cM_{\rm erg,<0}(\Gamma)\cup\cM_{\rm erg,>0}(\Gamma)$ the weak$\ast$ and entropy approximation of $\mu$ is a well known consequence of  \cite{Cro:11,Gel:16}. Thus, in what follows, we will assume $\mu\in\cM_{\rm erg,0}(\Gamma)$ and $\nu=\pi_\ast\mu$. Denote $h=h(\mu)$ and let $\varepsilon\in (0,h)$. 
	
	The steps of this proof are the following. First, we choose an appropriate skeleton with $n$ sufficiently large and $k\in\{n(1-\varepsilon),\ldots,n\}$. Based on that, we construct a certain SFT $S^+$ contained in $\Sigma^+$ which associates a  IFS of contracting interval maps. Thereafter, we estimate the topological entropy of this SFT. And we also prove that every measure supported on $S^+$ is close to $\nu^+$ in the correspondingly defined Wasserstein distance. Here $\nu^+$ is the measure $\nu$ projected to $\Sigma^+$. Finally, we argue that the convergence (in entropy and Wasserstein distance) on $\Sigma^+$ implies the one on $\Sigma$, proving the proposition.
	
\medskip\noindent\textbf{1. Choice of skeletons.}
By Proposition \ref{pro:skeleton}, there are numbers $a\in(0,1)$, $C>1$, and $n_0\ge1$ providing the  claimed skeletons $\sX=\sX(h,\varepsilon,a,C)=\{X_i\}_i$ of finitely many points $X_i=(\xi^i,x_i)$. Below, we will choose $n\ge n_0$ sufficiently large. To that end, fix  
\[
	L\in(0,\lvert\log f_0'(1)\rvert)
\]
 and consider   
\begin{equation}\label{eqdef:Z}
	Z
	\eqdef \{z\in[0,1]\colon f_0'(z)\le e^{-L}\},
\end{equation}
which is a nontrivial closed interval containing $1$.
Let $\ell\ge0$ be the smallest integer such that $f_0^\ell(a)\in Z$.		
Now choose 
\begin{equation}\label{eq:choicek}
	n
	> \max\left\{n_0,\frac{\ell}{(1-\varepsilon)\varepsilon}\cdot\log f_0'(0)\right\}
\end{equation}
so that
\begin{equation}\label{eqdef:n}
	\frac{1}{n+\ell+2n\varepsilon/L}
	\log\left(\frac{C}{n\varepsilon}e^{n(1-\varepsilon)(h-\varepsilon)}\right)
	\ge h-2\varepsilon.
\end{equation}
Let
\begin{equation}\label{eq:matter}
	\varepsilon'
	\eqdef \max_{k\in\{n(1-\varepsilon),\ldots,n\}}
	\frac{1}{k+\ell+2k\varepsilon/L}
		\left( \left(1+  k\varepsilon\right)
			+\left(\ell+\frac{2k\varepsilon}{L}\right)\right).
\end{equation}
Note that 
\[
	\varepsilon'
	= O(\varepsilon)
\]
as $n\to\infty$.

Now take $k\in\{n(1-\varepsilon),\ldots,n\}$ and the skeleton $\sX=\{X_i\}_i$, $X_i=(\xi^{(i)},x_i)\in\Gamma$, as provided by  Proposition \ref{pro:skeleton}. By item (i), it holds
\begin{equation}\label{eqdef:skeletonproof}
	\card\sX
	\ge \frac{C}{n\varepsilon}e^{n(1-\varepsilon)(h-\varepsilon)}.
\end{equation}
Let
\begin{equation}\label{eq:choices}
	s
	\eqdef \ell + \ell',
	\quad \text{ where }\quad
	\ell'\eqdef \left\lceil \frac{2k\varepsilon}{L}\right\rceil .
\end{equation}
For notational simplicity, let us assume that 
\begin{equation}\label{eq:defLLL}
	\ell'=\frac{2k\varepsilon}{L}.
\end{equation}

By property (iv) of the skeleton, for every $i$ we have 
\[
	x_i\in(0,a)
	\quad \text{ and }\quad
	x_i'\eqdef f_{\xi^{(i)}}^k(x_i)
	= f_{[\xi_0^{(i)}\ldots\,\xi_{k-1}^{(i)}]} (x_i)
	\in(a,1)
\]	 
and hence $x_i<a<x_i'$. Note that applying any concatenation of $f_0$ to any point $x\in(0,1)$ moves this point further to the right towards $1$, and hence
\begin{equation}\label{eq:furtherr}
	f_0^s\circ f_{[\xi_0^{(i)}\ldots\,\xi_{k-1}^{(i)}]}(a)
	> f_{[\xi_0^{(i)}\ldots\,\xi_{k-1}^{(i)}]} (a)
	> f_{[\xi_0^{(i)}\ldots\,\xi_{k-1}^{(i)}]}(x_i)
	> a
	> x_i.
\end{equation}
Observe that, in particular, by the choice of $\ell$, we have
\begin{equation}\label{eqLchoiceZell}
	(f_0^\ell\circ f_{[\xi_0^{(i)}\ldots\,\xi_{k-1}^{(i)}]})(x_i)
	\in Z.
\end{equation}

\noindent\textbf{2. Construction of an admissible SFT.}
We now introduce a certain SFT. Instead of describing its set of forbidden words of length $r=k+s$ (with $s$ as in \eqref{eq:choices}), we define its complement.
First, given the set $\{X_i\}_i$ of points $X_i=(\xi^{(i)},x_i)$ from the skeleton, define 
\[
	\eW'
	\eqdef \{(\xi_0^{(i)}\ldots\xi_{k-1}^{(i)})\}_i
\]	 
and prolong each word in $\eW'$ to an allowed word of length $r\eqdef k+s$ in the following way
\begin{equation}\label{eq:defSFTwords}
	\eW
	\eqdef \{(\eta^{(i)}_0\ldots\eta^{(i)}_{r-1})=(\xi_0^{(i)}\ldots\xi_{k-1}^{(i)}0^s)
		\colon (\xi_0^{(i)}\ldots\xi_{k-1}^{(i)})\in \eW'\}.
\end{equation}
Define
\begin{equation}\label{eqdef:Splus}
	S^+
	\eqdef \bigcup_{j=0}^{r-1}(\sigma^+)^j(\eW^\bN),
\end{equation}
where
\[
	\eW^\bN
	\eqdef \{\omega^+=(\eta^{(i_1)}\eta^{(i_2)}\ldots)\colon\eta^{(i_m)}\in\eW, m\ge1\}
	\subset \Sigma_2^+
\]	 
is the one-sided subshift of all one-sided infinite concatenations of such words of lenght $r$. The next result asserts that it is a subshift of forward admissible (one-sided) sequences.

\begin{lemma}\label{lclaim:subshift}
	$\eW^\bN\subset\Sigma^+$.
\end{lemma} 

\begin{proof}
By \eqref{eq:furtherr}, $(\eta^{(i)}_0\ldots\eta^{(i)}_{r-1})$ is forward admissible for $a$, that is,  $[a,1]\subset I_{[\eta^{(i)}_0\ldots\,\eta^{(i)}_{r-1}]}$, and
\begin{equation}\label{eq:continuo}
	f_{[\eta^{(i)}_0\ldots\,\eta^{(i)}_{r-1}]}([a,1]) 
	\subset (a,1].
\end{equation}
Thus, any pair of maps of the type $f_{[\eta^{(i)}_0\ldots\,\eta^{(i)}_{r-1}]}$, $(\eta^{(i)}_0\ldots\eta^{(i)}_{r-1})\in \eW$, can be concatenated. Hence any one-sided infinite concatenations of words in $\eW$ is forward admissible.
\end{proof}

\smallskip\noindent\textbf{3. A contracting IFS.}
The following is immediate from \eqref{eq:continuo}.

\begin{lemma}
	For every $(\eta^{(i)}_0\ldots\eta^{(i)}_{r-1})\in \eW$ there exists $p_i^-=p_{[\eta^{(i)}_0\ldots\eta^{(i)}_{r-1}]}^-$ satisfying
\[
	p_i^-
	= f_{[\eta^{(i)}_0\ldots\,\eta^{(i)}_{r-1}]}(p_i^-)
	\in(a,1].
\]		
\end{lemma}

For every $i$, define 
\begin{equation}\label{eq:rem:ffpts-g}
	g_i
	\eqdef f_{[\eta^{(i)}_0\ldots\,\eta^{(i)}_{r-1}]}
	\quad\text{ and let }\quad
	p^-
	\eqdef \min_ip^-_i
	>a.
\end{equation}
The above ingredients will define the announced IFS.

\begin{lemma}\label{lclaeq:rem:ffpts-maps}
	For every $i$, it holds $g_i([p^-,1])
	\subset [p^-,1]$.
\end{lemma}

\begin{proof}
By \eqref{eq:furtherr}, for all $i$ it holds $a<g_i(a)$. This together with $p_i^-=g_i(p_i^-)$ implies that the graph of $g_i$ on $(a,p_i^-)$ is above the diagonal.
As $a< p^-\le p_i^-$, we have $p^-\le g_i(p^-)$ for all $i$, proving the lemma. 
\end{proof}

By Lemma \ref{lclaeq:rem:ffpts-maps}, we can consider the IFS generated by the restriction of the maps $g_i$ to the interval $[p^-,1]$. 
		
\begin{lemma}[Uniform contractions]\label{lclaimlem:unicon}
	The IFS $\{g_i\}_i$ is contracting on $[p^-,1]$.
\end{lemma}		

\begin{proof}
By \eqref{eqLchoiceZell}, for every $i$, we have $z_i\eqdef (f_0^\ell\circ f_{[\xi^{(i)}_0\ldots\,\xi^{(i)}_{k-1}]})(x_i)\in Z$. Therefore, by \eqref{eqdef:Z} and \eqref{eq:defLLL}, we have
\begin{equation}\label{eq:sacccoch}
	(f_0^{\ell'})'(z_i)
	\le e^{-L\ell'}
	=e^{-2k\varepsilon}.
\end{equation}
Hence,
\[\begin{split}
	(g_i)'(x_i)
	&= (f_0^{\ell'}\circ f_0^\ell\circ f_{[\xi^{(i)}_0\ldots\,\xi^{(i)}_{k-1}]})'(x_i)\\
	&\le ( f_{[\xi^{(i)}_0\ldots\,\xi^{(i)}_{k-1}]})'(x_i) \cdot (\max f_0')^\ell \cdot
		(f_0^{\ell'})'(z_i)\\
	\tiny{\text{(by item (iii) of the skeleton, \eqref{eq:sacccoch})}}\quad
	&\le e^{k\varepsilon}\cdot (f_0'(0))^\ell \cdot e^{-L2k\varepsilon/L}
	= e^{-k\varepsilon}(f_0'(0))^\ell\\
	{\tiny \text{(using $n(1-\varepsilon)\le k$)}}\quad
	&\le e^{-n(1-\varepsilon)\varepsilon}(f_0'(0))^\ell\\
	\tiny{\text{(by \eqref{eq:choicek})}}\quad
	&< 1.
\end{split}\]
Recall again that, by \eqref{eq:furtherr} and \eqref{eq:rem:ffpts-g}, $x_i<a<p^-$. Hence, concavity of the fiber maps implies that the above upper bounds also hold for any point in the interval $[p^-,1]$.
\end{proof}

\begin{lemma}\label{lclaimjustonepoint}
	For every $\omega^+=(\eta^{(i_1)}\eta^{(i_2)}\ldots)\in \eW^\bN$, $\eta^{(i_m)}\in\eW$ for every $m\ge1$, the set
\[
	z(\omega^+)
	\eqdef \bigcap_{j\ge1}z_j(\omega^+),
\quad\text{ where }\quad
	z_j(\omega^+)
	\eqdef \bigcap_{m=1}^j (g_{i_1}\circ\ldots\circ g_{i_m})([p^-,1]),
\]
is just one point.
\end{lemma}

\begin{proof}
By Lemma \ref{lclaimlem:unicon}, the IFS is uniformly contracting.
 Note that $z_j(\omega^+)$ is a nested sequence of compact intervals and hence $z(\omega)$ is either a point or a compact interval. By contraction, it is just one point.
\end{proof}

\begin{lemma}\label{lcla:Splus}
	$S^+\subset\Sigma^+$.
\end{lemma}

\begin{proof}
	By Lemma \ref{lclaim:subshift}, every $\omega^+\in\eW^\bN$ is forward admissible and $z(\omega^+)$ defined in Lemma \ref{lclaimjustonepoint} is an admissible point. It follows that for every $i=0,\ldots,r-1$ the map $f_{\omega^+}^i$ is well defined at $z(\omega^+)$  and, in particular, $(\sigma^+)^i(\omega^+)$ is forward admissible. 
\end{proof}

Denote by $\sigma^+$ the one-sided shift on $\Sigma_2^+$. 		
Note that $S^+$ is a $\sigma^+$-invariant subshift which is a SFT. Clearly, the natural extension of $\sigma^+\colon S^+\to S^+$ is the SFT $\sigma\colon S\to S$, where 
\[
	S
	= \bigcup_{i=0}^{r-1}\sigma^i(\eW^\bZ)
\]
are all bi-infinite concatenations of words in $\eW$. It is clear that, by construction, $\sigma\colon S\to S$ is topologically transitive. At the end of the proof, we will use the following fact.

\begin{lemma}
	$S\subset \Sigma$.
\end{lemma} 

\begin{proof}
Let $\xi=(\ldots\xi_{-1}.\xi_0\xi_1\ldots)\in S$. By definition, for every $j\ge1$ its ``one-sided infinite remainder" $(\xi_{-j}\ldots\xi_{-1}\xi^+)$  belongs to $S^+$. By Lemma \ref{lcla:Splus}, it is forward admissible. Thus, observing that the two-sided sequence $\eta^{(j)}\eqdef(0^{-\bN}.\xi_{-j}\ldots\xi_{-1}\xi^+)$ is (forward and backward) admissible, we can conclude that the shifted sequence $\sigma^{-j}(\eta^{(j)}) = (\ldots\xi_{-j}\ldots\xi_{-1}.\xi_0\ldots\xi_j\ldots)$ is (forward and backward) admissible. Therefore, we have
\[
	C_j\cap\Sigma\ne\emptyset,
	\quad\text{ where }\quad
	C_j\eqdef [\xi_{-j}\ldots\xi_{-1}.\xi_0\ldots\xi_j].
\]
Observe that $\{C_j\}_j$ is a compact decreasing sequence of compact nonempty sets each intersecting the compact set $\Sigma$. Hence, also $\bigcap_jC_j$ intersects $\Sigma$ and therefore $\xi\in\Sigma$.
\end{proof}

\smallskip\noindent\textbf{4. Estimate of entropy.}

\begin{lemma}\label{lcla:entropy}
	$h_{\rm top}(\sigma^+, S^+)\ge h(\mu)-2\varepsilon$.
\end{lemma}

\begin{proof}
Take $k\le n$ from the skeleton $\sX$, for $n$ as in \eqref{eq:choicek}, and $s$ as in \eqref{eq:choices}. Recall that, by our choice \eqref{eq:defSFTwords}, we have $\card \eW=\card \sX$. Then 
\[\begin{split}
	h_{\rm top}(\sigma^+, S^+)
	&= \frac{1}{k+s} h_{\rm top}((\sigma^+)^{k+s},S^+)
	= \frac{1}{k+s} \log \card \eW\\
	\tiny{\text{(by \eqref{eq:choices} and \eqref{eqdef:skeletonproof})}}\quad
	&\ge \frac{1}{k+\ell+2k\varepsilon/L}
		\log\left(\frac{C}{n\varepsilon}e^{n(1-\varepsilon)(h-\varepsilon)}\right)\\
	\tiny{\text{(using that $k\le n$)}}\quad
	&\ge \frac{1}{n+\ell+2n\varepsilon/L}
		\log\left(\frac{C}{n\varepsilon}e^{n(1-\varepsilon)(h-\varepsilon)}\right)\\
	\tiny{\text{(by \eqref{eqdef:n})}}\quad
	&\ge h-2\varepsilon 
	= h(\mu)-2\varepsilon	,
\end{split}\]	
proving the lemma.
\end{proof}

\noindent\textbf{5. Weak$\ast$ approximation in $\cM(\Sigma^+)$.}
In the space $\Sigma_2^+$ we adopt the definition of the distance $d_1$ defined in \eqref{eqdef:distanceSigma} in the analogous way and consider the Wasserstein distance $W_1$ on $\cM(\Sigma^+_2)$. Consider the natural projection $\pi^+\colon\Sigma_2\to\Sigma^+_2$ to the space of one-sided sequences and let $\nu^+\eqdef(\pi^+)_\ast\nu$. Recall the notation $\cA_k$ for the finite averages of shifted measures in \eqref{eq:defcA}.

\begin{lemma}\label{lclaimsubclaim}
	For every $\xi^+\in\Sigma_2^+$ and every $\zeta^+\in[\xi_0\ldots\xi_{k-1}]^+\subset\Sigma_2^+$ we have 
\[
	W_1(\cA_k\zeta^+,\cA_k\xi^+)
	<\frac1k.
\]
\end{lemma}

\begin{proof}
We have 
\[
	W_1(\delta_{\sigma^\ell(\zeta^+)},\delta_{\sigma^\ell(\xi^+)})
	\le e^{-(k-\ell)}.
\]	 
For every $m\ge k$ and every $\zeta^+\in[\xi_0\ldots\xi_{k-1}\xi_k\ldots\xi_{m-1}]^+$ we hence have 
\[
	W_1(\cA_k\zeta^+,\cA_k\xi^+)
	\le \frac1k\sum_{\ell=0}^{k-1}e^{-(k-\ell)},
\]
which implies the claim of the lemma.
\end{proof}

\begin{lemma}\label{lclaimlem:pqp}
	For every $\tilde\nu\in\cM( S^+)$ it holds $W_1(\tilde\nu,\nu^+) \le \varepsilon'$, where $\varepsilon'$ is as in \eqref{eq:matter}.
\end{lemma}

\begin{proof}
Every sequences $\xi^{(i)}$ in the skeleton $\sX=\{X_i\}$ of points $X_i=(\xi^{(i)},x_i)\}$ satisfies
\begin{equation}\label{eq:weakstarbis}\begin{split}
	W_1(\cA_k(\xi^{(i)})^+,\nu^+)
	&= 	W_1((\pi^+)_\ast\cA_k(\xi^{(i)}),(\pi^+)_\ast\nu)\\
	\tiny{\text{by item (v) of the skeleton property, see \eqref{eq:weakstar}}}\quad
	&\le 	W_1(\cA_k(\xi^{(i)}),\nu) 
	<\varepsilon .
\end{split}\end{equation}
Recall the choice of words $(\xi^{(i)}_0\ldots\xi^{(i)}_{k-1}0\ldots 0)$ in \eqref{eq:defSFTwords} which define $S^+$. 
 Lemma \ref{lclaimsubclaim} and \eqref{eq:weakstarbis} then imply that for every $\eta^+\in[\xi^{(i)}_0\ldots\xi^{(i)}_{k-1}]$ we have
\begin{equation}\label{eq:splitbis}
	W_1(\eta^+,\nu^+)
	\le W_1(\eta^+,\cA_k(\xi^{(i)})^+)
		+W_1(\cA_k(\xi^{(i)})^+,\nu^+)
	\le \frac1k +\varepsilon.
\end{equation}

Recall that, by \cite[Corollary to Theorem 4]{Sig:74}, every invariant measure has a generic point. Let now $\tilde\nu\in\cM_{\rm erg}(S^+)$ and take a $\tilde\nu$-generic point $\zeta^+\in S^+$. Recall that $(\sigma^+)^j(\zeta^+)$ is also $\tilde\nu$-generic for every $j\ge1$. By construction of $S^+$ in \eqref{eqdef:Splus}, we can assume that there is some index $i$ so that $\zeta^+\in[\xi^{(i)}_0\ldots\xi^{(i)}_{k-1}0\ldots 0]$. Observe that, by construction of $S^+$, for every $m\ge1$ we then also have
\begin{equation}\label{eq:consequence}
	(\sigma^+)^{mr}(\zeta^+)
	\in [\xi^{(i_m)}_0\ldots\xi^{(i_m)}_{k-1}0\ldots 0]
\end{equation}
for some $i_m$ in the index set of the skeleton.
Recall that $r=k+s=k+\ell+\ell'$. Consider the decomposition
\begin{equation}\label{eq:split}
	\cA_{r}\zeta^+
	= \frac{1}{k+\ell+\ell'} 
	\left( k\cA_k\zeta^+
		+ \left(\ell+\ell'\right) R(k,\zeta^+)\right),
\end{equation}
for some probability measures $R(k,\zeta^+)$. 

By Lemma \ref{lem:Wassx}, the  Wasserstein distance satisfies 
\begin{equation}\label{eq:again}
	W_1(s\nu_1+(1-s)\nu_2,\nu^+)
	= sW_1(\nu_1,\nu^+)+(1-s)W_1(\nu_2,\nu^+)
\end{equation}
for arbitrary  $\nu_1,\nu_2$  probability measures and $s\in[0,1]$.
Hence, with \eqref{eq:split} we conclude 
\[\begin{split}
	W_1(\cA_{r}\zeta^+,\nu^+)
	&= \frac{1}{k+\ell+\ell'}
		\left( k\cdot W_1(\cA_k\zeta^+,\nu^+) 
			+\left(\ell+\ell'\right)W_1(R(k,\zeta^+),\nu^+)\right)\\
	&\le \frac{1}{k+\ell+\ell'}
		\left( k\cdot W_1(\cA_k\zeta^+,\nu^+) 
			+\left(\ell+\ell'\right)\cdot 1\right)\\
			\tiny{\text{(using \eqref{eq:splitbis})}}\quad
	&\le \frac{1}{k+\ell+\ell'}
		\left( \left(1+  k\varepsilon\right)
			+\left(\ell+\ell'\right)\right)
	\\
				\tiny{\text{(using \eqref{eq:choices})}}\quad	
	&= \frac{1}{k+\ell+2k\varepsilon/L}
		\left( \left(1+  k\varepsilon\right) 
			+\left(\ell+2k\varepsilon/L\right)\right)	\\
	&\le \varepsilon',		
\end{split}\]
where $\varepsilon'$ was defined in \eqref{eq:matter}.

Using \eqref{eq:consequence}, we can repeat the above arguments replacing $\zeta^+$ by $\eta^+=(\sigma^+)^{mr}(\zeta^+)$ and $[\xi^{(i)}_0\ldots\xi^{(i)}_{k-1}0\ldots0]$ by $[\xi^{(i_m)}_0\ldots\xi^{(i_m)}_{k-1}0\ldots0]$ and obtain
\[
	W_1(\cA_{r}(\sigma^+)^{mr}(\zeta^+),\nu^+)
\]
 for any $m\ge0$
Hence, since 
\[
	\cA_{jr}\zeta^+
	=\frac{1}{j}\left( \cA_r\zeta^++\cA_r(\sigma^+)^r(\zeta^+)
		+\ldots+\cA_r(\sigma^+)^{(j-1)r}(\zeta^+)\right)
\]	
and using again \eqref{eq:again}, we obtain
\[
	W_1(\cA_{jr}\zeta^+,\nu^+)
	\le \frac{1}{j}\cdot j \varepsilon'	.
\]
By the triangle inequality,
\[
	W_1(\tilde\nu,\nu^+)
	\le W_1(\tilde\nu,\cA_{jr}\zeta^+)+ W_1(\cA_{jr}\zeta^+,\nu^+)
	\le W_1(\tilde\nu,\cA_{jr}\zeta^+)+\varepsilon'.
\]
As $\zeta^+$ is generic, $\cA_{jr}\zeta^+\to\tilde\nu$ as $r\to\infty$ in the weak$\ast$ topology, proving the lemma.
\end{proof}

\begin{lemma}\label{lclaimlem:natext}
	There is a function $\varepsilon\mapsto\delta(\varepsilon)$, $\delta(\varepsilon)\to0$ as $\varepsilon\to0$, such that for every $\nu^+_1,\nu^+_2\in\cM(\Sigma^+)$ satisfying $W_1(\nu^+_1,\nu^+_2)<\varepsilon$ then their natural extensions $\nu_1,\nu_2\in\cM(\Sigma)$ satisfy $W_1(\nu_1,\nu_2)<\delta$.
\end{lemma}

\begin{proof}
	It suffices to observe that taking natural extension is a homeomorphism between the two compact metric spaces $\cM(\Sigma^+)$ and $\cM(\Sigma)$.
\end{proof}
	
Now the following is an immediate consequence of Lemmas \ref{lclaimlem:pqp} and \ref{lclaimlem:natext}.
	
\begin{lemma}\label{lclaimlem:pqpbis}
	For every 	$\tilde\nu\in\cM( S)$ it holds $W_1(\tilde\nu,\nu)\le \delta(\varepsilon')$, where $\delta(\cdot)$ is as in Lemma \ref{lclaimlem:natext}.
\end{lemma}

\medskip\noindent\textbf{6. Construction of a basic set.}
Recall $z\colon \eW^\bN\to[0,1]$ defined in Lemma \ref{lclaimjustonepoint} and note that,  letting $z((\sigma^+)^i(\omega^+))\eqdef f_{\omega^+}^i(z(\omega^+))$, it extends to a map  $z\colon S^+\to[0,1]$. Let
\[
	\Upsilon^+
	\eqdef \{(\omega^+,z(\omega^+))\colon\omega^+\in S^+\}
	\subset \Sigma^+\times[0,1].
\]	
Observe that $\Upsilon^+$ is compact and $F^+$-invariant, for $F^+\colon \Sigma^+\times[0,1]\to \Sigma^+\times[0,1]$  defined by
\[
	F^+(\xi^+,x)
	\eqdef (\sigma^+(\xi^+),f_{\xi_0}(x)).
\]
and take the natural extension $\Upsilon$ of $\Upsilon^+$ relative to $F^+$. Note that
$\Upsilon$ is a compact $F$-invariant set is semi-conjugate to the two-sided transitive SFT $\sigma\colon S\to S$.
By construction, on $\Upsilon$ the map $F^r$ is hyperbolic of contracting type in the fiber direction. The set $\Upsilon$ is the claimed basic set.  

\medskip\noindent\textbf{7. Weak$\ast$ and in entropy approximation in $\cM(\Sigma)$.}
The above construction depends on $\varepsilon$ and on the hence chosen quantifiers. Given $\varepsilon=1/n$, let us now denote the correspondingly constructed (two-sided) SFT by $S_n$ and the corresponding basic set by $\Upsilon_n$.
Clearly, any $\mu_n\in\cM(\Upsilon_n)$ projects to a measure $\nu_n=\pi_\ast\mu_n\in\cM(S_n)$. By Lemma \ref{lclaimlem:pqpbis}, it holds $\lim_{n\to\infty}\nu_n=\nu$ in the weak$\ast$ topology in $\cM(\Sigma)$. Recalling that $\nu=\pi_\ast\mu$ proves the first assertion of the proposition.

For the second assertion, note that $h_{\rm top}(F,\Upsilon_n)=h_{\rm top}(\sigma,S_n)$.  By Lemma \ref{lcla:entropy}, we have 
\[
	L
	\eqdef\limsup_nh_{\rm top}(\sigma,S_n)
	\ge h(\mu)
	=h(\nu).
\]	 
We claim that we have in fact equality. Arguing by contradiction, suppose that $L>h(\nu)$, consider the sequence of measures $\nu_n\in\cM(S_n)$ of maximal entropy $h(\nu_n)=h_{\rm top}(\sigma,S_n)$. By Lemma \ref{lclaimlem:pqp}, $\nu=\lim_n\nu_n$. Hence, by upper semi-continuity of the entropy map, we have $\limsup_nh(\nu_n)\le h(\nu)$, a contradiction.

This finishes the proof of Theorem \ref{teopro:intermediate}.
\qed

%------------------------------------------------------------------------------------------------------
\subsection{Proof of Theorem \ref{teo:accum}}
%------------------------------------------------------------------------------------------------------

By Theorem \ref{teopro:intermediate}, it remains to show that not only the projections converge to the projection, but also the measures themselves converge to.
Note again $\pi_\ast\mu_n\to\pi_\ast\mu=\nu$ in the weak$\ast$ topology. 

By our hypothesis, $\chi(\mu)=0$. As we assume (H2+), we can invoke Theorem \ref{teo:1} Case b). Hence,  $\mu$ is the only measure with $\pi_\ast\mu=\nu$. Thus, applying Lemma \ref{pro:simple}, the sequence $\mu_n$ weak$\ast$ converges to $\mu$.  
\qed

%------------------------------------------------------------------------------------------------------
\subsection{Weak$\ast$ and entropy approximation of linear combinations of hyperbolic ergodic measures with the same type of hyperbolicity}\label{sec:weakstarentropy}
%------------------------------------------------------------------------------------------------------

Let us first state  some fundamental fact about SFTs. 

\begin{lemma}\label{lem:weakstarentropySFT}
Let $(X,T)$ be a transitive SFT. Then every measure $\mu\in\cM(X)$ can be approximated, weak$\ast$ and in entropy, by a sequence of ergodic measures. That is, $\cM(X)$ is an entropy dense Poulsen simplex.
\end{lemma}

\begin{proof}
That $\cM(X)$ is a Poulsen simplex is well known \cite{Sig:74}. 
The entropy denseness appears first in \cite[Lemma 3]{JohJorOebPol:10} for full shifts, the version for SFT is a simple modification and appears in \cite[Proposition 3.3]{BarRamSim:}.
\end{proof}

We will apply Lemma \ref{lem:weakstarentropySFT} to the following situation.
\begin{remark}
	Given $\Upsilon\subset\Gamma$ a basic set, then $F|_\Upsilon$ is topologically conjugate to a topological Markov chain. In particular, this Markov chain is a SFT and it is topologically transitive (mixing) if $F|_\Upsilon$ is topologically transitive (mixing) (see, for example, \cite[Chapter 18.7]{KatHas:95}). 
\end{remark}

\begin{proposition}\label{cor:Poulsen}
	Given any finite number of ergodic measures $\mu_1,\ldots,\mu_k\in\cM_{\rm erg,\ge0}(\Gamma)$ and positive numbers $\lambda_1,\ldots,\lambda_k$ satisfying $\sum_{i=1}^k\lambda_i=1$, the measure $\mu'=\sum_{i=1}^k\lambda_i\mu_i$ is weak$\ast$ and in entropy approximated by ergodic measures in $\cM_{\rm erg,>0}(\Gamma)$.
\end{proposition}

\begin{proof}
	For every measure $\mu_i\in\cM_{\rm erg,\ge0}(\Gamma)$ there exists a sequence of horseshoes (with uniform fiber expansion) $(\Gamma_n^{(i)})_n$ such that $\cM(\Gamma_n^{(i)})$ converges weak$\ast$ to $\mu_i$ and $h_{\rm top}(F,\Gamma_n^{(i)})$ converges to $h(\mu_i)$. Indeed, apply Remark \ref{Katokforevery} if $\mu_i$ is hyperbolic and Theorem \ref{teo:accum} otherwise.
	In particular, the measure $\mu'$ is weak$\ast$ and in entropy approximated by measures $\mu'_n=\sum_{i=1}^k\lambda_i\mu_i'$, where $\mu_i'\in\cM_{\rm erg}(\Gamma^{(i)}_n)$ can be taken the measure of  maximal entropy for $F$ in $\Gamma_n^{(i)}$. By Corollary \ref{c.horseshoes}, for every $n$ there exists a horseshoe $\Gamma_n'\supset\bigcup_i\Gamma^{(i)}_n$.  In particular, $\mu'_n\in\cM(\Gamma_n')$.
		By Lemma \ref{lem:weakstarentropySFT}, $\mu_n'$ can be weak$\ast$ and in entropy approximated by ergodic measures in $\cM_{\rm erg}(\Gamma_n')$. Clearly, $\cM_{\rm erg}(\Gamma_n')\subset\cM_{\rm erg,>0}(\Gamma)$.
\end{proof}	

\begin{remark}\label{rem:BocAvi}
Given $\mu\in\cM(\Gamma)$ and $U$ some neighborhood of $\mu$ in the weak$\ast$ topology, there exist $\mu_1,\ldots,\mu_k\in\cM_{\rm erg,\ge0}(\Gamma)$ and positive numbers $\lambda_1,\ldots,\lambda_k$ satisfying $\sum_{i=1}^k\lambda_i=1$ such that $\mu'=\sum_{i=1}^k\lambda_i\mu_i\in U$ (see, for example, \cite[Lemma 2.1]{AviBoc:12}). 
\end{remark}

\begin{proof}[Proof of Corollary \ref{cor:accumoneside}]
The Proposition \ref{cor:Poulsen} implies the assertion in Corollary \ref{cor:accumoneside} for measures whose ergodic decomposition has only finitely many components. For the general case recall Remark \ref{rem:BocAvi} and note that for any $\mu'=\sum_i\lambda_i\mu_i$ we have $h(\mu')=\sum_i\lambda_ih(\mu_i)$. 
	The corollary now follows from Proposition \ref{cor:Poulsen}.
\end{proof}

%------------------------------------------------------------------------------------------------------
\section{Entropy-dense Poulsen structure of $\cM(\Sigma)$: Proof of Theorem \ref{teo:poulsen}} \label{sec:teo:poulsen}
%------------------------------------------------------------------------------------------------------

	By Remark \ref{rem:BocAvi}, it suffices to assume that $\nu=\sum_{i=1}^k\lambda_i\nu_i$ for some $\nu_1,\ldots,\nu_k\in\cM_{\rm erg}(\Sigma)$ and some positive numbers $\lambda_1,\ldots,\lambda_k$ satisfying $\sum_{i=1}^k\lambda_i=1$.
	
	For every $\nu_i$, by Lemma \ref{lem:existsemihypmeas} there exists some measure $\mu_i\in\cM_{\rm erg,\ge0}(\Gamma)$ with non-negative fiber Lyapunov exponent satisfying $\pi_\ast\mu_i=\nu_i$. 
	For each $i$, by Theorem \ref{teopro:intermediate}, there is a sequence of basic sets $\Upsilon^{(i)}_n\subset \Gamma$ such that their measure spaces project to measure spaces weak$\ast$ converging to $\nu_i$. Moreover, there is a sequence of measures $\mu^{(i)}_n$ such that $\pi_\ast\mu^{(i)}_n$ converge weak$\ast$ and in entropy to $\pi_\ast\mu_i=\nu_i$. 
	
	By Corollary \ref{c.horseshoes} for each $n$ there exists a horseshoe  
\[
	\Gamma_n
	\supset\bigcup_{i=1}^k\Upsilon^{(i)}_n.
\]	 
By  \cite[Lemma 2.14]{BarRamSim:}, we can arbitrarily approximate weak$\ast$ and in entropy the measure $\sum_{i=1}^k\lambda_i\mu^{(i)}_n$ by ergodic measures in $\cM_{\rm erg}(\Gamma_n)$,
\[
	\sum_{i=1}^k\lambda_i\mu^{(i)}_n
	= \underset{m\to\infty}{\text{we-}\lim\,}\mu_{n,m},
	\quad\mu_{n,m}\in\cM_{\rm erg}(\Gamma_n),
\]	
where $\welim$ denotes convergence in the weak$\ast$ topology and in entropy.
Hence
\[
	\sum_{i=1}^k\lambda_i\pi_\ast\mu^{(i)}_n
	= \pi_\ast\sum_{i=1}^k\lambda_i\mu^{(i)}_n
	= \pi_\ast\left(\underset{m\to\infty}{\text{we-}\lim\,}\mu_{n,m}\right)
	= \underset{m\to\infty}{\text{we-}\lim\,}\pi_\ast\mu_{n,m},
\]	
	 Hence, by a diagonal argument, we find a sequence of ergodic measures $(\mu_{n,m_n})_n$ such that 
\[
	\nu
	= \sum_i\lambda_i\nu_i
	= \lim_n\sum_i\lambda_i\pi_\ast\mu^{(i)}_n 
	= \lim_n\left(\underset{m\to\infty}{\text{we-}\lim\,}\pi_\ast\mu_{n,m}\right).
\]	 
This concludes the proof of the theorem.
\qed

%------------------------------------------------------------------------------------------------------
\section{Bifurcation exit scenarios}\label{sec:bif}
%------------------------------------------------------------------------------------------------------

In this section, we put our results in an extended ambient considering bifurcation scenarios. The skew-products lead to globally defined one-parameter families that may have ``explosions of entropy and  of the space of ergodic measures" at some bifurcation associated to a collision of a pair of homoclinic classes. This sort of question was studied in \cite{DiaSan:04,DiaRoc:07}, where the collision occurs through the orbit of a parabolic (saddle-node) point. In our setting the collision may be of different nature. We now proceed to explain the details, not aiming for complete generality but for giving the general ideas.

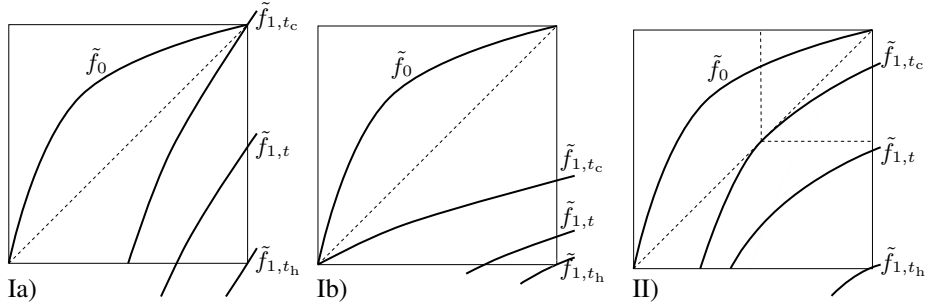
\begin{figure}[h] 
 \begin{overpic}[scale=.25]{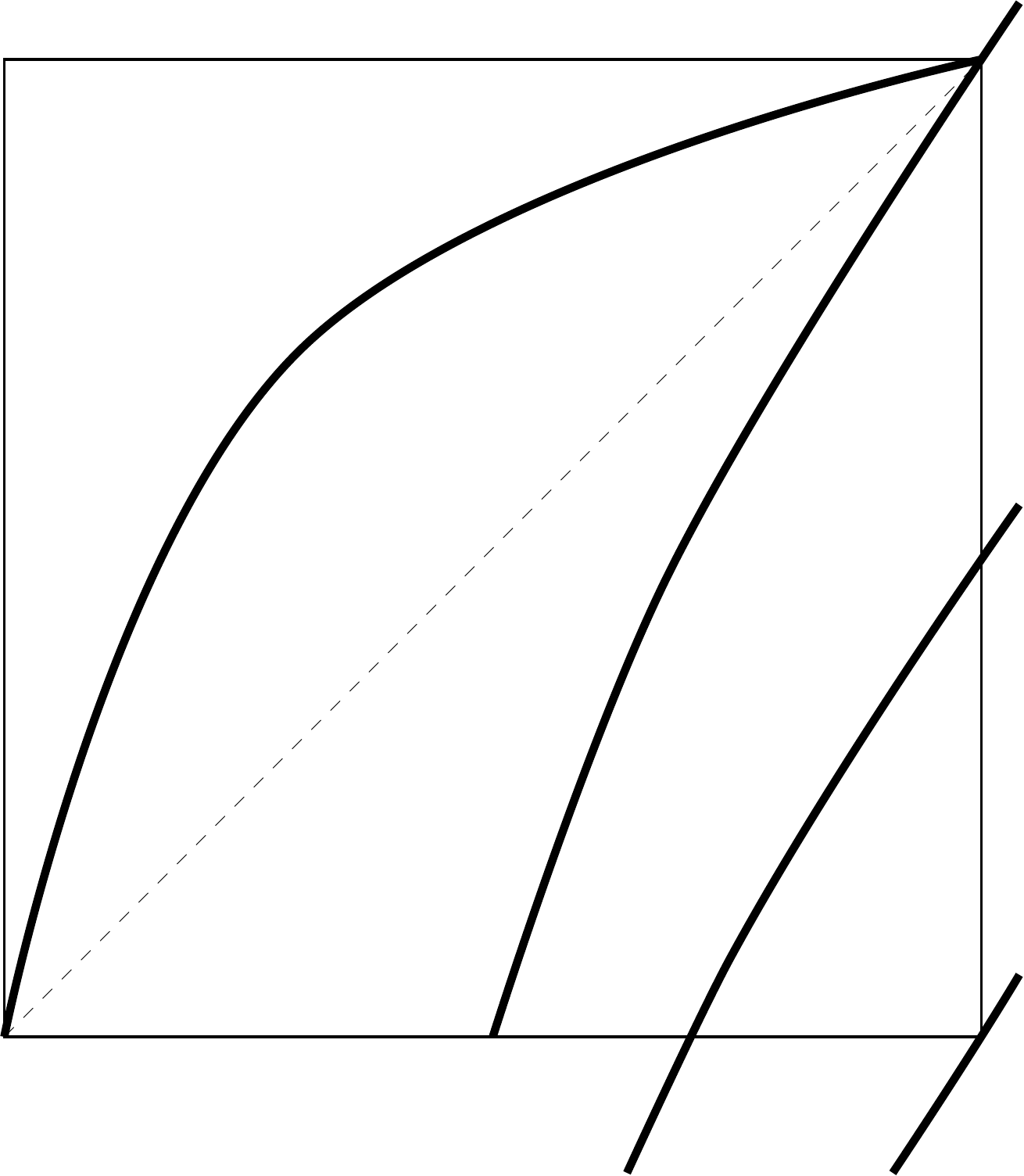}
	\put(0,0){Ia)}
	 \put(27,79){\small $\tilde f_0$}
 	\put(86,10){\small $\tilde f_{1,t_{\rm h}}$}
 	\put(86,50){\small $\tilde f_{1,t}$}
 	\put(86,95){\small $\tilde f_{1,t_{\rm c}}$}
 \end{overpic}\hspace{0.7cm}
 \begin{overpic}[scale=.25]{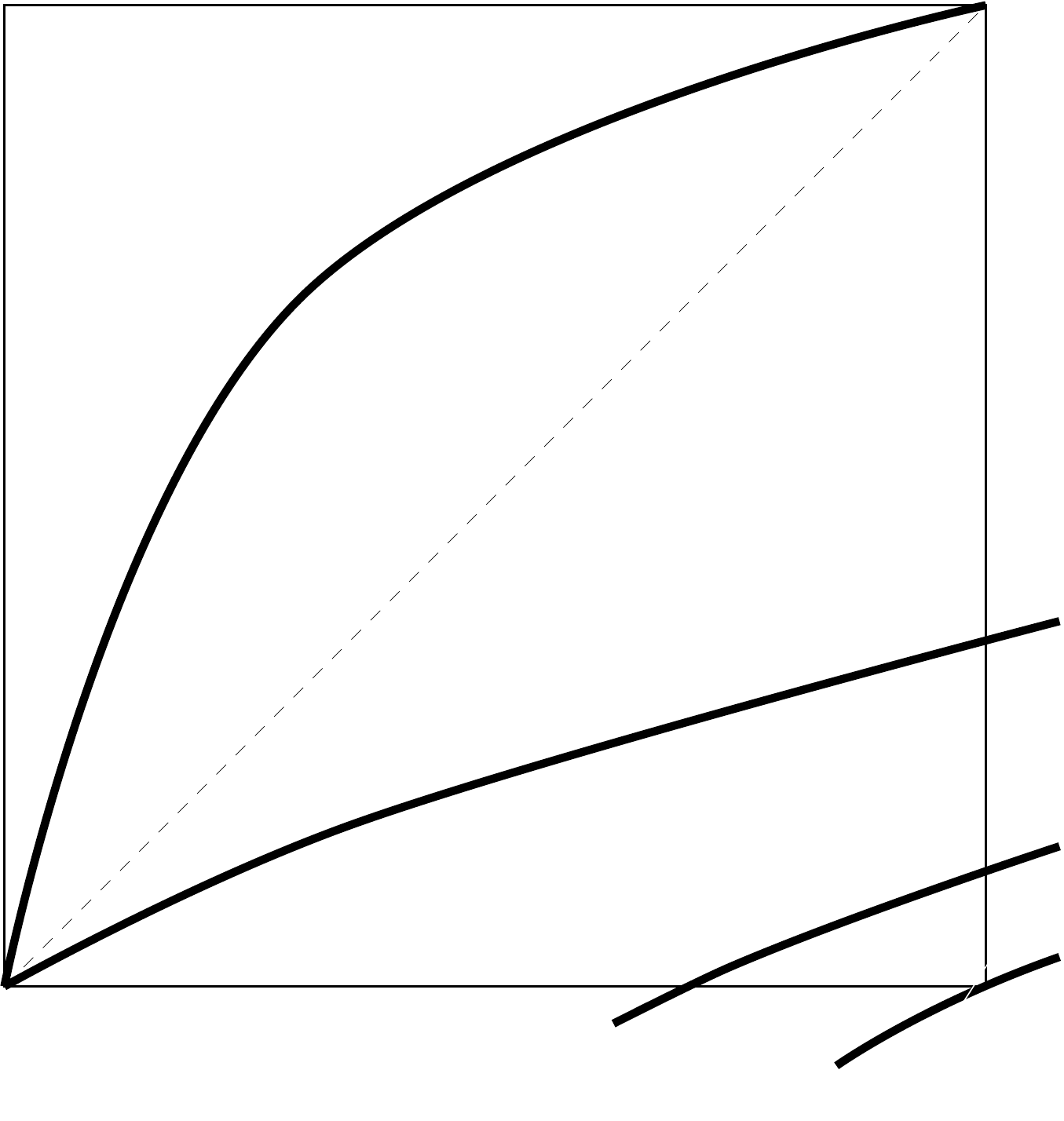}
 	\put(0,0){Ib)}
	\put(27,82){\small $\tilde f_0$}
 	\put(89,8){\small $\tilde f_{1,t_{\rm h}}$}
 	\put(89,27){\small $\tilde f_{1,t}$}
 	\put(89,48){\small $\tilde f_{1,t_{\rm c}}$}
 \end{overpic}\hspace{0.7cm}
 \begin{overpic}[scale=.25]{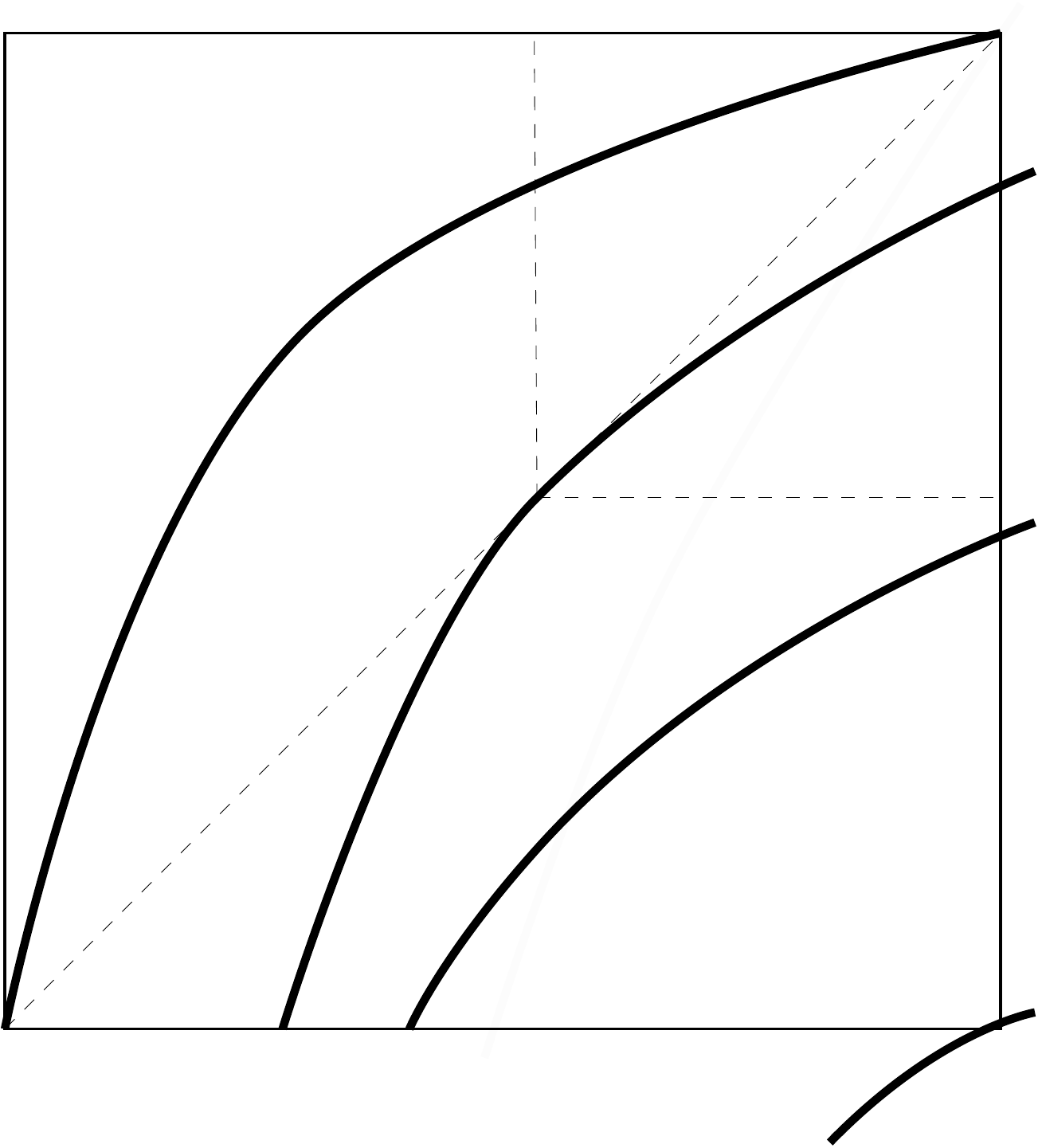}
 	\put(0,0){II)}
	\put(27,80){\small $\tilde f_0$}
 	\put(90,10){\small $\tilde f_{1,t_{\rm h}}$}
 	\put(90,50){\small $\tilde f_{1,t}$}
 	\put(90,85){\small $\tilde f_{1,t_{\rm c}}$}
 \end{overpic}
 \caption{The bifurcation scenarios}
 \label{fig.bifs}
\end{figure}

Returning to the globally defined skew-product map $\tilde F$ in \eqref{eq:parental}, consider one whose fiber maps $\tilde f_0$ and $\tilde f_1$  satisfy conditions (H1)--(H2) in the interval $[0,1]$. 
Our previous study applies to the set $\Gamma$ defined in \eqref{def:Gamma}, which is indeed the locally maximal  invariant set of $\tilde F$ in the open set $\Sigma_2 \times (-\varepsilon, 1+\varepsilon)$ for some small $\varepsilon>0$. Though the dynamics of $\tilde F$ may have other ``pieces" beyond $\Gamma$.%
\footnote{Giving an interpretation using the \emph{filtrations} in Conley theory \cite{Con:78} for the study of chain recurrence classes,  the  set  $\Gamma$ is the maximal invariant set of a filtrating neighbourhood. This means that there are two compact sets with nonempty interiors  $M_1$ and $M_2$, with  $M_1\subset \mathrm{int}(M_2)$ and such that $\tilde F(M_i)\subset \mathrm{int}(M_i)$, $i=1,2$. Then one studies the dynamics in $M_2 \setminus M_1$ which is a level of the filtration. Further levels may also be considered. In some cases, $\Gamma$ may split into two separated parts, and hence there is another filtrating set separating them. Indeed, this occurs when the homoclinic classes $H(P,\tilde F)$ and $H(Q,\tilde F)$ both are hyperbolic (see Theorem \ref{teo:homoclinicclasses}).  
}
To perform a bifurcation analysis, we embed $\tilde F$ in a one-parameter family $\tilde F_t$, $t\in[t_{\rm h},t_{\rm c}]$, defined by means of fiber maps $\{\tilde f_0,\tilde f_{1,t}\}$. Here the parameter $t_{\rm h}$ corresponds to a heterodimensional cycle and $t_{\rm c}$ to a ``collision''. Concerning the latter, there are two qualitatively different scenarios that will be studied below, compare Figure \ref{fig.bifs}. 
Topologically, in Cases Ia) and Ib), there occurs a collision of a homoclinic class in $\Gamma$ with one ``coming from outside" that does not involve parabolic points. In Case II, there occurs an internal collision of the classes of $P$ and $Q$ along  a parabolic orbit. In either case, the maximal invariant set that emerges at the bifurcation has full entropy $\log2$. The further analysis depends on the choice of the concave maps. On one hand, we may observe an explosion of the space of  admissible sequences (Proposition \ref{prop:explosionsequences}). On the other hand, in the ergodic level, we may observe an explosion in the space of measures and in entropy (Proposition \ref{pro:entropyjump}). Finally, we see how the twin-structure of the measures in Theorem \ref{teo:1} changes in Case I) (Proposition \ref{p.FSTmeasures}). In Sections \ref{sec:1134} and \ref{sec:1135} we will give an interpretation of the Cases I) and II), respectively.

\subsection{Formal setting}
Consider real functions $\tilde f_0,\tilde f_1$ whose restrictions to the interval $[0,1]$, that again we denote by $f_0,f_1$, satisfy the following hypothesis that is slightly more general than (H1) in the way that it allows $f_1$ to have a fixed point. 
\begin{itemize}
\item[(\~H1)] $f_0\colon [0,1]\to[0,1]$ is a differentiable increasing map such that $f_0$ is onto and satisfies $f_0(0)=0$, $f_0(1)=1$, $f_0'(0)>1$, $f_0'(1)\in(0,1)$, and $f_0(x)>x$ for every $x\in(0,1)$. 
Moreover, there is $d\in[0,1]$ such that $f_1\colon[d,1]\to[0,1]$ is a differentiable increasing map satisfying $f_1(x)\le x$ for every $x\in[d,1]$ such that $f_1(d)=0$, and $f_1$ has at most one fixed point in $[d,1]$. 
\end{itemize}

Unless stated otherwise, in this section, we fix a map $\tilde f_0$ and consider a one-parameter family of maps $\{\tilde f_{1,t}\}_{t\in[t_{\rm h},t_{\rm c}]}$ such that $\{f_0,f_{1,t}\}$ satisfy (\~H1) and (H2), with the corresponding (possibly degenerate) intervals $[d_t,1]\subset [0,1]$, for every $t\in[t_{\rm h},t_{\rm c}]$.%
\footnote{One could also study a more general case when $\tilde f_0$ also depends on the parameter $t$, hence changing its fixed points. This analysis just would require  straightforward modifications of the domains of the maps.}
We will invoke (H2+) only in Proposition \ref{p.FSTmeasures}.
We assume $f_{1,t}$ to be continuous in $t$ in the $C^1$ topology.  Moreover, the family $\{f_0,f_{1,t}\}$ satisfies (H1) for every $t\in(t_{\rm h},t_{\rm c})$ and at $t=t_{\rm h}$ and $t=t_{\rm c}$ complies the following ``exit bifurcation scenarios"  for a family of maps satisfying (\~H1):
\begin{itemize}
\item (\emph{bifurcation parameter} $t=t_{\rm h}$) $f_{1,t_{\rm h}}(1)=0$,
\item (\emph{bifurcation parameter} $t=t_{\rm c}$) there exists $a\in [0,1]$ such that $f_{1,t_{\rm c}}(a)=a$.
\end{itemize}
Note that for every $t\in(t_{\rm h},t_{\rm c}]$ the interval $[d_t,1]$ is nondegenerate.
By assumption, for every $t\in(t_{\rm h},t_{\rm c})$ the family $\{f_0,f_{1,t}\}$ satisfies (H1)--(H2) and therefore the corresponding results of previous sections hold. 

There are three cases of what can happen at the bifurcation point $t=t_{\rm c}$:
\begin{itemize}
\item[I)] (hyperbolic case)
$f_{1,t_{\rm c}}'(a)\ne1$
\begin{itemize}
\item[a)] $f_{1,t_{\rm c}}'(a)>1$ (and hence $a=1)$,
\item[b)] $f_{1,t_{\rm c}}'(a)<1$ (and hence $a=0$),
\end{itemize}
\item[II)] (parabolic case) $f_{1,t_{\rm c}}'(a)=1$ (in this case there is no \emph{a priori} restriction on the value $a\in[0,1]$).
\end{itemize}
Cases Ia) and Ib) are actually identical up to the time reversal (compare also Remark \ref{rem:time-reversal}), so we only consider Case Ia). 

\begin{remark}
Notice that the induced IFS $\{f_0,f_{1,t_{\rm c}}\}$ on $[0,1]$, corresponding to Case I, was also studied in \cite{FanSimTot:06}, though the focus there was on the stationary measures (and also assuming contraction on average) which represent a special subclass of invariant measures.  Closer to our approach is \cite{AlsMis:14} studying the so-called mystery of the vanishing attractor where concavity properties similar to the ones in Section \ref{sec:concavemaps} are used.
\end{remark}

\begin{example}\label{exa:simple}
One simple example is the family
\[
	 f_0, \quad
	\tilde f_{1,t}=\tilde f_1+t,
\]
where $f_0$ satisfies (H1)--(H2) and $\tilde f_1$ is a differentiable increasing map of $\bR$ such that $\tilde f_1'$ is not increasing. Now consider the associated maps $\{f_0,f_{1,t}\}$, where 
\[
	f_0 = \tilde f_0|_{[0,1]},\quad
	f_{1,t}=\tilde f_{1,t}|_{[\tilde f_{1,t}^{-1}(0),1]}.
\]	 
In this example, in either case  we have $t_{\rm h}=-\tilde f_1(1)$. The value $t_{\rm c}$ depends on the cases:
\begin{itemize}
\item if $\tilde f_1'(1)>1$ then $a=1$ and $t_{\rm c}=1-\tilde f_1(1)$,
\item if $\tilde f_1'(1)\le1$ and there exists $c\in[0,1]$ with $\tilde f_1'(c)=1$ then $a=c$ and $t_{\rm c}=c-\tilde f_1(c)$,  
\item if $\tilde f_1'(0)<1$  then $a=0$ and $t_{\rm c}=-\tilde f_1(0)$.
\end{itemize}
The first and the last case correspond to Ia) and Ib), respectively, while the second case corresponds to II).
\end{example}

Let $\tilde F_t$, $t\in[t_{\rm h}, t_{\rm c}]$, be the corresponding skew-product and denote by $\Gamma^{(t)}$ the analogously defined maximal invariant set of $\tilde F_t$ in $\Sigma_2 \times (-\varepsilon_t, 1+\varepsilon_t)$ for some small $\varepsilon_t>0$. Define 
\[
	F_t
	\eqdef {\tilde F_t}|_{\Gamma^{(t)}}
	\quad\text{ and let }\quad
 	\Sigma^{(t)}
	\eqdef \pi(\Gamma^{(t)})
\]
be the space of admissible sequences for $F_t$.  

\begin{remark}\label{rem:idiot}
Note that in Cases I) and II), we have $\Sigma^{(t_{\rm c})}=\{0,1\}^\bZ$, because for $t=t_{\rm c}$ all sequences are forward and backward admissible at $a$. 
In particular,
\[
	\Sigma_2\times\{a\}\subset\Gamma^{(t_{\rm c})}.
\]
Meanwhile, $\Sigma^{(t_{\rm h})}$ is a countable set consisting of sequences with at most one appearance of the symbol $1$. Thus, the topological entropy of the bifurcating maps $F_{t_{\rm c}}$ and $F_{t_{\rm h}}$ are $\log 2$ and 0, respectively.
\end{remark}

\begin{remark}[Discussion of Figure \ref{fig.bif}]\label{rem:fig.bif}
	Figure \ref{fig.bif} depicts the following points
\[\begin{split}	
	A
	&= ((0^{-\bN}1.0^\bN),0) 
	\in \cW^\s(Q,\tilde F_{t_{\rm h}})\cap\cW^\u(P,\tilde F_{t_{\rm h}}),\\
	B&= ((0^{-\bN}1.0^\bN),0) 
	\in \cW^\s(Q,\tilde F_{t_{\rm c}})\cap\cW^\u(Q,\tilde F_{t_{\rm c}}),
	\quad
	B\in H(Q,\tilde F_{t_{\rm c}}),\\
	C&= ((0^{-\bN}1.0^\bN),\tilde f_{1,t_{\rm c}}(1)) 
	\in \cW^\s(P,\tilde F_{t_{\rm c}})\cap\cW^\u(P,\tilde F_{t_{\rm c}}),
	\quad
	C\in H(P,\tilde F_{t_{\rm c}}),\\
	D&= ((0^{-\bN}1.0^\bN),0) 
	= \cW^\s(Q,\tilde F_{t_{\rm c}})\cap\cW^\u(P,\tilde F_{t_{\rm c}}),
	\quad
	D\in H(Q,\tilde F_{t_{\rm c}}),\\
	E&= ((0^{-\bN}1.0^\bN),1) 
	\in \cW^\s(P,\tilde F_{t_{\rm c}})\cap\cW^\u(P,\tilde F_{t_{\rm c}}),
	\quad
	E\in H(P,\tilde F_{t_{\rm c}}),\\
	S
	&=(1^\bZ,0),\\
	R
	&=(1^\bZ,1).
\end{split}\]	

The point $R$ is a fixed point of $\tilde F_{t_{\rm c}}$ of expanding type and hence cannot be homoclinically related to the fixed points $P$ of contracting type. Moreover, $\Sigma_2\times\{1\}\subset H(P,\tilde F_{t_{\rm c}})\cap H(R,\tilde F_{t_{\rm c}})$ and $R$ and $P$ are involved in a heterodimensional cycle. Analogous arguments apply to $S$ and $Q$ and the set $\Sigma_2\times\{0\}$.

Finally, the difference between the points $C$ and $E$ is that while both are homoclinic points of $P$,  the latter is a contained in the intersection of the ``strong stable'' set and the unstable set of $P$. Analogously for the homoclinic points $B$ and $D$ of $Q$.
\end{remark}

\begin{remark}[Nondecreasing complexity]
For the simple bifurcation family in Example \eqref{exa:simple}, the family of compact sets $\{\Sigma^{(t)}\}_t$ is nondecreasing in $t$. Hence, the topological entropy of $\tilde F_t$ on $\Gamma^{(t)}$ is nondecreasing and there is no ``annihilation" of periodic points as $t$ increases. This may not be the case in a general situation.
As observed before, we pass from zero entropy (for $t=t_{\rm h}$) to full entropy ($t=t_{\rm c}$).
These features resemble somewhat to the ones in the quadratic family of maps $g_\lambda(x)=\lambda x(1-x)$, $\lambda\in[1,4]$, see for example, \cite{MilThu:88}. In the quadratic family the creation of periodic points occurs through saddle-node and flip bifurcations as well as the creation of ``homoclinic tangencies" (in the sense that the critical point is pre-periodic). In our case, periodic points are created either through saddle-node bifurcations or heterodimensional cycles  (see Remark \ref{rem:herocycle}).
In some sense, one may regard this family as a partially hyperbolic version of the quadratic family showing a nondecreasing transition from trivial dynamics to full chaos. 
\end{remark}

%------------------------------------------------------------------------------------------------------
\subsection{Bifurcation at $t_{\rm h}$: heterodimensional cycles}
%------------------------------------------------------------------------------------------------------

By Remark \ref{rem:herocycle},  the map $\tilde F_{t_{\rm h}}$ has a heterodimensional cycle associated to $P=(0^\bZ,1)$ and $Q=(0^\bZ,0)$. 

%------------------------------------------------------------------------------------------------------
\subsection{Bifurcation at $t_{\rm c}$: collisions of sets}\label{ss.explosionofentropyand}
%------------------------------------------------------------------------------------------------------

Consider the function $C\colon[t_{\rm h},t_{\rm c}]\to [0,\infty]$,
\begin{equation}\label{eq:entjumpCt}
	C(t)\eqdef 
	\begin{cases}
		\displaystyle\frac{ \lvert \log f_{0}'(1)\rvert}{\log f_{1,t}'(1)}
			&\text{ if }f_{1,t}'(1)>1,\\
		\infty
			&\text{ otherwise}.
	\end{cases}	
\end{equation}
By hypothesis, $C$ is continuous.

%---------------------------------------------------------------------------------------------------
\subsubsection{Space of admissible sequences}
%---------------------------------------------------------------------------------------------------

Recall the definitions of the upper and lower frequencies $\overline\freq$ and $\underline\freq$ in \eqref{eq:deffreq}.

\begin{proposition}\label{prop:explosionsequences}
Given $C\ge0$, let 
\[
	S_C
	\eqdef \left\{\xi\in\Sigma_2\colon\frac{\overline\freq(\xi,1)}{\underline\freq(\xi,0)}
	\le 	C\right\}.
\]
Then for every $\varepsilon>0$ there exists $\delta>0$ such that for $t\in (t_{\rm c}-\delta,t_{\rm c})$ we have
\[
	\Sigma^{(t)}\subset S_{C(t_{\rm c})+\varepsilon}.
\]
In particular, if $C(t_{\rm c})<\infty$ then there is a jump in the space of admissible sequences.
Moreover, independently of the value $C(t_{\rm c})$, $\Sigma^{(t)}\to\Sigma_2$ in the Hausdorff distance as $t\to t_{\rm c}$.
\end{proposition}

\begin{proof}
The first statement is an immediate consequence of Corollary \ref{cor:exponentextreme-xi}. Indeed, the ratio of frequencies cannot be larger than $C(t)$ and $C(t)\to C(t_{\rm c})$ as $t\to t_{\rm c}$.

By Remark \ref{rem:idiot}, we have $\Sigma^{(t_{\rm c})}=\Sigma_2$, which implies the claimed jump if $C(t_{\rm c})<\infty$.

For the last assertion, observe that for any $n\ge1$ there is $\varepsilon_n$ such that for every $t\in(t_{\rm c}-\varepsilon_n,t_{\rm c})$ the word $1^n$ is forward admissible in $\Sigma^{(t)}$. Hence, every word of length $n$ is forward admissible in $\Sigma^{(t)}$, which means that $\Sigma^{(t)}$ intersects every (forward) $n$-th level cylinder for $t$ small enough. This, together with the fact that $\Sigma^{(t)}$ is shift-invariant, implies the convergence in Hausdorff distance.
\end{proof}

%---------------------------------------------------------------------------------------------------
\subsubsection{Spaces of measures and entropy}
%---------------------------------------------------------------------------------------------------

Let  
\[
	\mathcal H(p) 
	\eqdef -p\log p - (1-p)\log(1-p),
	\quad\text{ for }\quad
	p\in(0,1).
\]	

\begin{proposition}\label{pro:entropyjump}
	For every $t<t_{\rm c}$ and $C(t)$ as in \eqref{eq:entjumpCt} we have that
\[
	\cM(\Sigma^{(t)})
	\subset 
		\Big\{\nu\in\cM(\Sigma_2)\colon\frac{\nu([1])}{\nu([0])}\le C(t)\Big\}
\]	
is a closed proper subset of $\cM(\Sigma_2)$. Moreover, for every $t<t_{\rm c}$ we have
\begin{equation}\label{eq:blabla}
	\sup_{\nu\in\cM_{\rm erg}(\Sigma^{(t)})}h(\nu)
	= h_{\rm top}(\sigma,\Sigma^{(t)})
	\le 
	\mathcal H(p_t),
	\,\text{ where }\,
	p_t\eqdef \min\Big\{\frac12,\frac{C(t)}{1+C(t)}\Big\}.
\end{equation}

In particular, if $C(t_{\rm c})<1$ then the space $\cM(\Sigma_2)$ is not the weak$\ast$ limit of the subspaces $\cM(\Sigma^{(t)})$ as $t\to t_{\rm c}$ and there is a jump in entropy at $t_{\rm c}$, in the sense that
\[	
	\mathcal H(p_{t_{\rm c}})
	= \mathcal H(\frac{C(t_{\rm c})}{1+C(t_{\rm c})})
	< 
	\mathcal H(\frac12)= \log2
	= h_{\rm top}(\sigma,\Sigma_2)
	 = h_{\rm top}(\sigma,\Sigma^{({t_{\rm c}})}).
\]
\end{proposition}

\begin{proof}
First observe that by Lemma \ref{lem:exponentextreme} for any $t<{t_{\rm c}}$ and every $\nu\in\cM(\Sigma^{(t)})$, we have
\[	 \frac{\nu([1])}{\nu([0])}
	 \le 
	 \frac{ \lvert \log f_{0}'(1)\rvert}{\log f_{1,t}'(1)}
	=C(t).
\]
This implies the first statement. 

To prove \eqref{eq:blabla}, observe that $\eA=\{[0],[1]\}$ is a generating partition for $\Sigma^{(t)}$. Hence, by the Kolmogorov-Sinai theorem, for any ergodic measure $\nu\in\cM_{\rm erg}(\Sigma^{(t)})$, letting $p\eqdef\nu([1])$ and hence $1-p=\nu([0])$, we have 
\[
	h(\nu)
	= h(\nu,\eA)
	\le -p\log p-(1-p)\log(1-p)
	=\mathcal H(p).
\]	 
As $p\in[0,p_t]$ and $\frac12$ is the maximum of $\mathcal H$, equation \eqref{eq:blabla} follows from the variational principle for entropy, proving the proposition.
\end{proof}

%---------------------------------------------------------------------------------------------------
\subsubsection{Structure of the space of measures}
%---------------------------------------------------------------------------------------------------

In this section we will put Theorem~\ref{teo:1} in the context of the bifurcation scenario. 

\begin{proposition}\label{p.FSTmeasures}
Assume (\~H1)--(H2+).
There exist continuous functions $\kappa_1,\kappa_2\colon(0,\infty)\to(0,\infty)$ which are increasing and satisfy $\lim_{D\to0}\kappa_i(D)=0$, $i=1,2$, such that, given any $\nu\in\cM_{\rm erg}(\Sigma_2)$, one of the following three cases occurs:
\begin{enumerate}
\item[a)] There exist exactly two measures $\mu_1,\mu_2\in\cM_{\rm erg}(\Gamma^{(t_{\rm c})})$  such that $\pi_\ast\mu_1=\nu=\pi_\ast\mu_2$. In this case, both measures are hyperbolic and have fiber Lyapunov exponents with different signs. More precisely, 
if $\chi(\mu_1)>0>\chi(\mu_2)$ then
for the Wasserstein distance $D\eqdef W_1(\mu_1,\mu_2)$ between $\mu_1$ and $\mu_2$, we have	
\[
	D
	= \int x\,d\mu_2 (\xi,x)-\int x\,d\mu_1(\xi,x)
	>0
\]	
and 
\[
	-\kappa_2(D)
	<\chi(\mu_2)
	<-\kappa_1(D)
	<0<\kappa_1(D)
	<\chi(\mu_1)
	<\kappa_2(D).
\]
\item[b)] There exists exactly one measure $\mu\in\cM_{\rm erg}(\Gamma^{(t_{\rm c})})$ 
with $\pi_\ast\mu=\nu$ and $\chi(\mu)=0$. 
\item[c)] There exists exactly one measure $\mu\in\cM_{\rm erg}(\Gamma^{(t_{\rm c})})$ with $\pi_\ast\mu=\nu$ and $\chi(\mu)>0$. In this case we are in Case Ia) and this measure is supported on $\Sigma_2\times \{1\}$.
\end{enumerate}
\end{proposition}

Below we see that case c) in the above proposition indeed occurs (see Remark \ref{rem:seehere}).

\begin{remark}\label{rem:bifcases}
Given $\xi=(\xi_0\ldots\xi_{n-1})^\bZ\in\Sigma^{(t_{\rm c})}$, consider the map $g=f_{\xi_0,t}\circ\ldots\circ f_{\xi_{n-1},t}$ (writing $f_{0,t}=f_0$) defined on $I_\xi$. There are the following possibilities according to the hyperbolic case Ia) and the parabolic case II):
\begin{itemize}
\item[Ia)] 
\begin{itemize}
\item either $g$ has a unique fixed point that is parabolic,
\item  or $g$ has a pair of fixed points $p^+_{[\xi_0\ldots\,\xi_{n-1}]}<p^+_{[\xi_0\ldots\,\xi_{n-1}]}$ that are repelling and contracting, respectively.
\end{itemize}
\item[II)] 
\begin{itemize}
\item either $1$ is the unique fixed point of $g$ (which can be repelling or parabolic),
\item or $g$ has a pair of fixed points $p^+_{[\xi_0\ldots\,\xi_{n-1}]}<p^+_{[\xi_0\ldots\,\xi_{n-1}]}$ that are repelling and contracting, respectively.
\end{itemize}
\end{itemize}
\end{remark}

\begin{proof}[Proof of Proposition \ref{p.FSTmeasures}]
The arguments of the proof of Theorem~\ref{teo:1} essentially work.  Case a) is as before. It remains to consider the cases where there is only one measure $\mu \in\cM_{\rm erg}(\Gamma^{(t_{\rm c})})$  projecting to $\nu$. Let us first see that $\chi(\mu)\ge0$. By contradiction, if $\chi(\mu)<0$ then let $\nu=\pi_\ast\mu$. Observe that, by Corollary \ref{corppp.periodicapproximation}, $\nu$ is accumulated by periodic measures $\nu_n$ supported on periodic orbits of sequences $\xi^{(n)}$. 
By Remark \ref{rem:bifcases}, there is a sequence of periodic points $P_n=(\xi^{(n)},p_n)\in\Gamma^{(t_{\rm c})}$ that are either repelling or parabolic. Consider the sequence $(\mu_n)_n$ of periodic measures supported on those points. Taking,  if necessary, a subsequence, we can assume that $\mu_n\to \mu'$. By construction $\pi_\ast (\mu')=\nu$ and $\chi(\mu') \ge 0$.
Thus $\mu'\ne \mu$, a contradiction.

To see that case c) does not occur in the parabolic case II), we argue as above. Indeed, arguing again by contradiction, by Remark \ref{rem:bifcases}, there is a sequence of periodic points $Q_n$ of contracting or parabolic type whose corresponding measures converge to some measure $\mu''$ satisfying $\pi_\ast\mu''=\pi_\ast\mu$ and $\chi(\mu'')\le0$, which is a contradiction.  
\end{proof}

%------------------------------------------------------------------------------------------------------
\subsubsection{Illustration of the hyperbolic case Ia)}\label{sec:1134}
%------------------------------------------------------------------------------------------------------

We present an example of a skew-product with two parts leading to a collision of homoclinic classes. The map $\tilde F_t$ in $\Sigma_2 \times [0,1]$ has the dynamics discussed in the previous sections. The dynamics of $\tilde F_t$ in $\Sigma_2 \times [1+\delta_t, 2]$, for some small $\delta_t>0$, is a ``twisted twin copy" of the former one, see Figure~\ref{Figurebiggerpicture}. In this way, we get locally maximal invariant sets $\Gamma^{(t)}\subset \Sigma_2\times(\delta_t,1+\delta_t/2)$ and $\Upsilon^{(t)}\subset\Sigma_2\times(1+\delta_t/2,2+\delta_t)$, respectively, which have qualitatively  ``the same'' dynamics when interchanging the roles of the maps $f_0$ and $f_{1,t}$. The sets $\Gamma^{(t)}$ and $\Upsilon^{(t)}$ are disjoint for $t\in (t_{\rm h},t_{\rm c})$ and collide for $t=t_{\rm c}$. Since $\Sigma^{(t_{\rm c})}=\Sigma_2$, this collision is big as the sets $\Gamma^{(t_{\rm c})}$ and $\Upsilon^{(t_{\rm c})}$ intersect in the ``topological horseshoe" $\Sigma_2 \times \{1\}$ whose hyperbolic-like nature depends on the value $C(t_{\rm c})$ defined in \eqref{eq:entjumpCt}.
Assuming that  $C(t_{\rm c})<1$,
according to the results above, we observe at the bifurcation $t=t_{\rm c}$  an explosion of the symbolic space and of entropy. The example illustrates where a substantial part of the ``additional'' symbolic sequences, in particular periodic sequences, come from. Indeed, the hyperbolic periodic points do not simply appear ``out of thin air" but come from outside, that is, from $\Upsilon^{(t)}$.

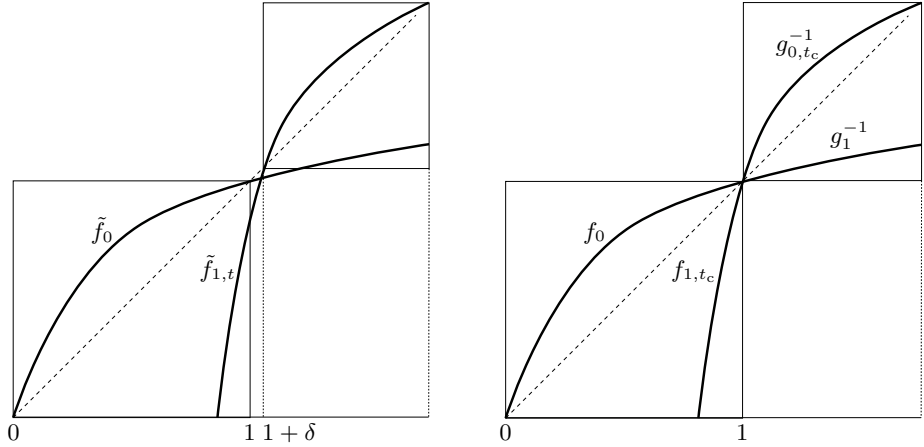
\begin{figure}[h] 
 \begin{overpic}[scale=.31]{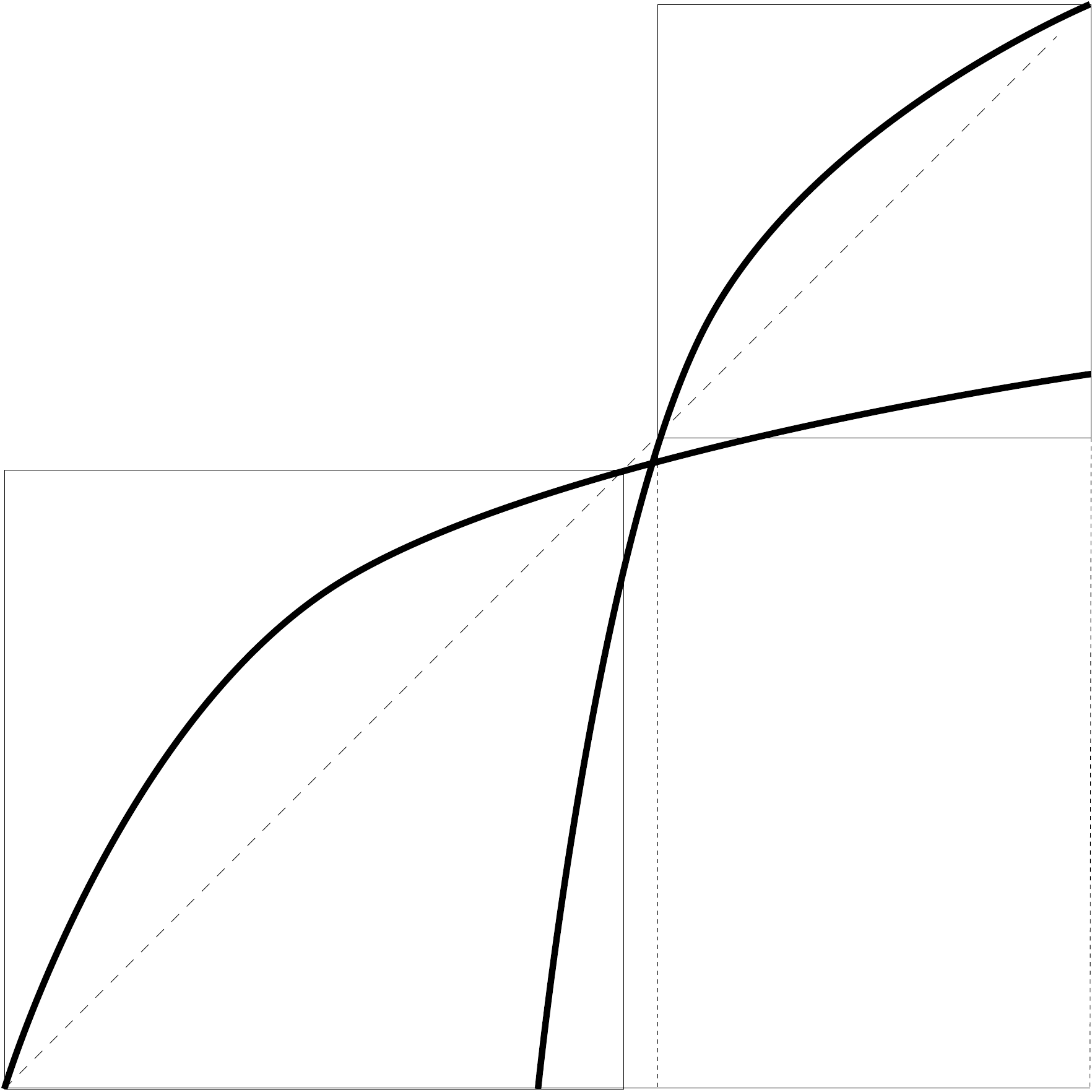}
 	\put(19,43){\small $\tilde f_0$}
 	\put(45,34){\small $\tilde f_{1,t}$}
 	\put(-1,-5){\small $0$}
 	\put(55.5,-5){\small $1$}
 	\put(60,-5){\small $1+\delta$}
 \end{overpic}\hspace{0.9cm}
 \begin{overpic}[scale=.31]{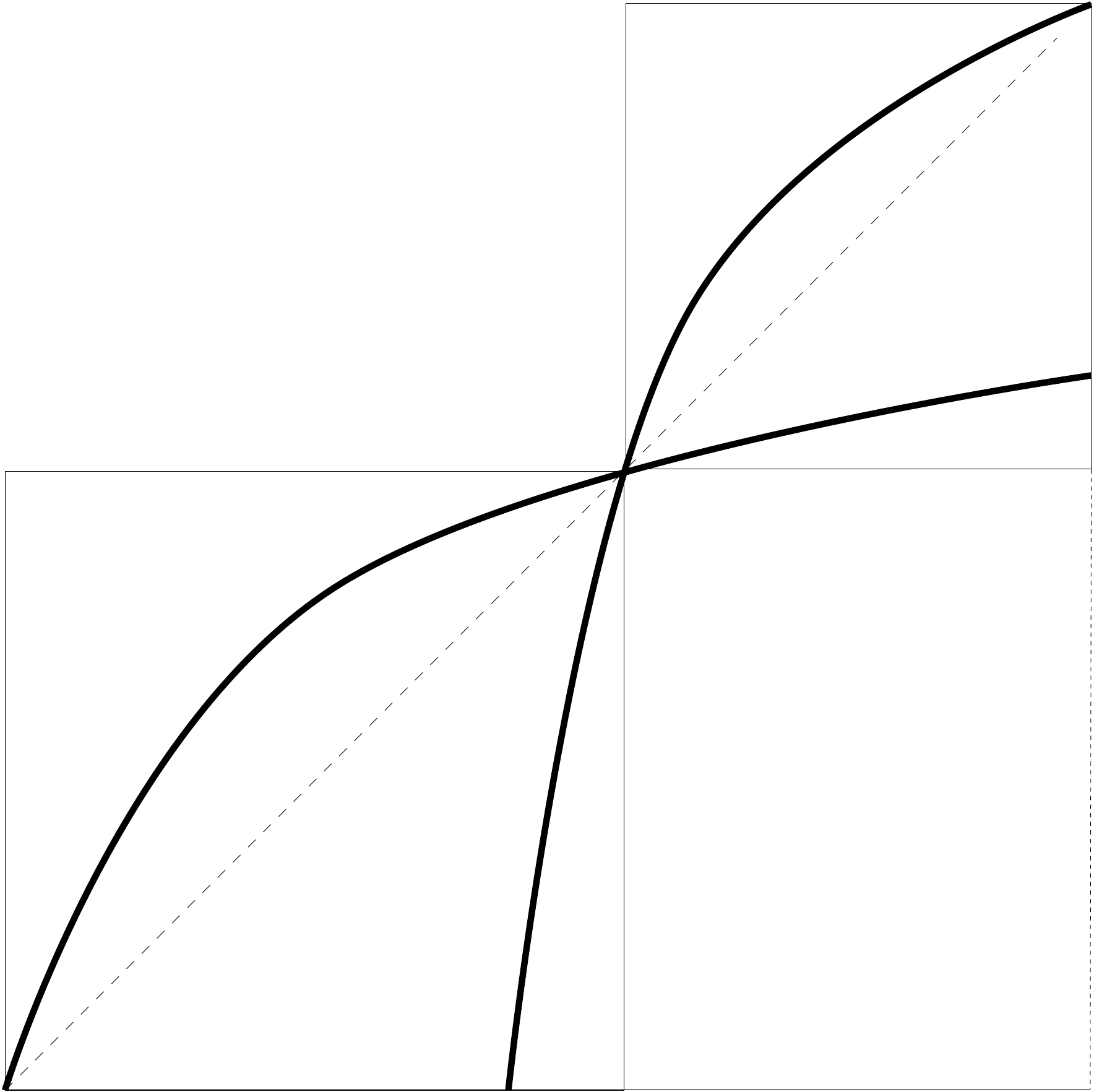}
 	\put(19,43){\small $ f_0$}
 	\put(40,34){\small $ f_{1,t_{\rm c}}$}
 	\put(65,88){\small $ g_{0,t_{\rm c}}^{-1}$}
 	\put(78,66){\small $ g_1^{-1}$}
 	\put(-1,-5){\small $0$}
 	\put(55.5,-5){\small $1$}
 \end{overpic}\hspace{2cm}
  \caption{Bifurcation with collision of homoclinic classes along a topological horseshoe}
 \label{Figurebiggerpicture}
\end{figure}

\begin{remark}[Entropy jump viewed from another side]
\label{rem:threedimensionalexplosions}
Recalling that there is no entropy in the fibers (see \eqref{eqlem:projent}), Proposition \ref{pro:entropyjump} implies that we have the corresponding jump in the topological entropy of $F_t$ at $\Gamma^{(t)}$ at $t=t_{\rm c}$. 

Observe that $\Upsilon^{(t)}$ has a correspondingly defined IFS with fiber maps $g_{0,t}$ and $g_{1}$ given by $g_{0,t}=\tilde f_{1,t}^{-1}$ and $g_1=\tilde f_0^{-1}$ (on appropriately defined domains), compare Figure \ref{Figurebiggerpicture}. 
Analogously, as in \eqref{eq:entjumpCt},  we can define a function $C'(t)$ for the maps $\{g_{0,t},g_1\}$. Note that for the collision parameter $t=t_{\rm c}$ we have $C(t_{\rm c})=C'(t_{\rm c})^{-1}$. Therefore, $C(t_{\rm c})<1$ implies $C'(t_{\rm c})>1$ and, as a consequence of Proposition \ref{pro:entropyjump}, there is no entropy jump in $\Upsilon^{(t_{\rm c})}$. Any entropy jump for $\Gamma^{(t)}$ in fact comes from entropy in $\Upsilon^{(t)}$. 
\end{remark}

\begin{remark}[Explosion of space of sequences and expanding measures without twins]\label{rem:seehere}
We also observe an explosion in the space of \emph{ergodic} measures $\cM_{\rm erg}(\Gamma^{(t)})$ at $t=t_{\rm c}$. The simplest example is the Dirac mass $\delta_R$ at $R=(1^\bZ,1)$ recalling that by hypothesis Ia),
\[
	\chi(\delta_R)
	=\log f_{1,t_{\rm c}}'(1)
	>0.
\]
Indeed, for every periodic sequence $\xi$ with large frequency of $1$s we also have that the measure uniformly distributed in the orbit of  $(\xi,1)$ is of expanding type.  Moreover, any ergodic measure of expanding type supported on $\Sigma_2\times\{1\}$ has no twin. All these measures provide examples of expanding measures without twins, that is, to which Proposition \ref{p.FSTmeasures} Case c) apply. 
None of them can be obtained as weak$\ast$ limits of the sets of invariant measures $\cM_{\rm erg}(\Gamma^{(t)})$ as $t\to t_{\rm c}$.

Arguments similar to the ones in Remark \ref{rem:threedimensionalexplosions} apply and imply that those ``new" measures come from $\cM_{\rm erg}(\Upsilon^{(t)})$.  
\end{remark}

%------------------------------------------------------------------------------------------------------
\subsubsection{Illustration of the parabolic case II)}\label{sec:1135}
%------------------------------------------------------------------------------------------------------

In what follows, we study the parabolic case recalling the choice of the point $a\in[0,1]$. In fact, we will assume $a\in(0,1)$, when $a\in \{0,1\}$ the statement also is true but the proof resembles methods of the previous subsection.

\begin{proposition}[Convergence to full entropy]
	We have 
\[
	\lim_{t\to t_{\rm c}}h_{\rm top}(\sigma,\Sigma^{(t)})
	= \log 2 
	= h_{\rm top}(\sigma,\Sigma^{(t_{\rm c})}).
\]	
Moreover, $\Sigma^{(t)}$ converges to $\Sigma_2$ in the Hausdorff distance as $t\to t_{\rm c}$.
\end{proposition}

\begin{proof}
Choose $k\in \bN$. For every word $\omega \in \Sigma_k' := \{0,1\}^k\setminus \{1^k\}$ there exists $\varepsilon_\omega$ such that for $t_{\rm c}-\varepsilon_\omega<t<t_{\rm c}$ we have $f_{[\omega],t}(a)>a$ (adopting the corresponding notation). As there are only finitely many words of length $k$, we can find $\varepsilon_k>0$ valid for every $\omega\in\Sigma_k'$. Thus, for $t\in(t_{\rm c}-\varepsilon_k,t_{\rm c})$ any concatenation of words from $\Sigma_k'$ is allowed in $\Sigma^{(t)}$. This allows us to estimate the entropy by
\[
	h_{\rm top}(\sigma,\Sigma^{(t)})
	\ge\frac{1}{k}\log (2^{k}-1),
\] 
which  implies the first claim and also provides the convergence in Hausdorff distance. 
\end{proof}

Note that in the parabolic case, at $t=t_{\rm c}$ we have $C(t_{\rm c})=\infty$ and hence the second claim of Proposition \ref{prop:explosionsequences} does not apply.

\begin{remark}
Arguing  as in the proof of Proposition \ref{pl.parabolichomoclinic}, one sees that the homoclinic classes $H(P,\tilde F_{t_{\rm c}})$ and $H(Q,\tilde F_{t_{\rm c}})$ both contain the parabolic fixed point $A=(1^\bZ,a)$. Hence, under certain conditions implying that the homoclinic classes of $P$ and of $Q$ are hyperbolic for $t<t_{\rm c}$ sufficiently close to $t_{\rm c}$, this  leads to a collision of hyperbolic homoclinic classes.
\end{remark}

%------------------------------------------------------------------------------------------------------
\section{Discussion: Homoclinic scenarios beyond concavity}\label{sec:homscen}
%------------------------------------------------------------------------------------------------------

Under conditions (H1)--(H2), Proposition~\ref{pro.l.homoclinicallyrelated} (together with Proposition~\ref{pl.parabolichomoclinic} in the case when there are parabolic points) allows us to consider  $H(P,\tilde F)$ and $H(Q,\tilde F)$ as the only two homoclinic classes of $\tilde F$ in $\Sigma_2 \times [0,1]$. According to the choice of $\tilde F$, there are two possibilities for the sets $H(P,\tilde F)$ and $H(Q,\tilde F)$: either they are disjoint or they have nonempty intersection (in this latter case, the sets may be equal or not). 

The constructions in \cite{DiaEstRoc:16} provide an explicit two-parameter family (with parameters $a$ and $t$)  of fiber maps $f_0=g_a$ (concave) and $f_1=g_{1,t}$ (affine) such that the corresponding skew-product map $\tilde F_{a,t}$ falls into one of the following cases, according to the choices of the parameters $a$ and $t$:
\begin{enumerate}
\item The sets $H(P,\tilde F_{a,t})$ and $H(Q,\tilde F_{a,t})$ are pairwise disjoint and hyperbolic and their union is the limit set of $F_{a,t}$ in $\Sigma_2\times [0,1]$ (\cite[Theorem 2.7 case (B)]{DiaEstRoc:16}).
\item $H(P,\tilde F_{a,t})\cap H(Q,\tilde F_{a,t})$ is the orbit of a parabolic point of $\tilde F_{a,t}$  (\cite[Theorem 2.7 case (C.c)]{DiaEstRoc:16}).
\item $H(P,\tilde F_{a,t})= H(Q,\tilde F_{a,t})=\Gamma_{a,t}$   (\cite[Theorem 2.7 case (A)]{DiaEstRoc:16}).
\end{enumerate}
Regarding the above scenarios, the corresponding space of ergodic measures splits into two parts, corresponding to the measures of contracting and expanding type:
\begin{enumerate}
\item These parts are disjoint.
\item Their closures intersect in a measure supported on a parabolic periodic orbit.
\item Their closures intersect in nonhyperbolic ergodic measures, some of them with positive entropy (see \cite{BocBonDia:16}).
\end{enumerate}
In each of these cases, we can choose the fiber dynamics in a way that the each class is locally maximal. By Proposition \ref{pro.l.homoclinicallyrelated}, any pair of saddles of the same type of hyperbolicity are homoclinically related, hence we can apply \cite{GorPes:17} and conclude that the corresponding parts of the space of ergodic measures each are arcwise connected and have closures which are a Poulsen simplex. Compare also with Corollary \ref{cor:arcwiseconn}.

We now discuss possible configurations of homoclinic classes for skew-products as in \eqref{eq:parental} assuming (H1) but not \emph{a priori}  $(H2)$ (that is, without the concavity assumptions). 
The following remark indicates that without the concavity assumption the scenery can be vast, with many possibilities for the interrelation between those classes (and hence for the resulting topological and ergodic properties). 

\begin{remark}[Homoclinic scenarios when (H2) is not satisfied]\label{rem:examples}
The map $\tilde F$ may have other hyperbolic periodic points $R$ which may fail to be homoclinically  related to $P$ or $Q$. In this setting, it is fundamental to understand how these periodic points and their homoclinic classes are inserted in the dynamics
of $\tilde F$. Indeed, the following (possibly non-exhaustive) list of dynamical scenarios may occur: 
\begin{itemize}
\item[(1)] $H(P,\tilde F)$ and $H(Q,\tilde F)$ are the only homoclinic classes of $\tilde F$ in  $\Sigma_2\times [0,1]$  (that is, any other homoclinic class of $\tilde F$ in $\Sigma_2\times [0,1]$ is equal to one of these two classes) and these two classes are:
\begin{itemize}
\item[(a)] (transitivity) 
$H(P,\tilde F)=H(Q,\tilde F)=\Gamma$ and hence $F=\tilde F|_\Gamma$ is nonhyperbolic and $\Gamma$ is a transitive set, see \cite{Dia:95}.
\item[(b)] (hyperbolicity) 
$H(P,\tilde F)\cap H(Q,\tilde F)=\emptyset$, each of them is hyperbolic, and the limit set of $\tilde F$ in $\Sigma_2\times [0,1]$ is the union $H(P,\tilde F)\cup H(Q,\tilde F)$, see
\cite{DiaRoc:97}.
\item[(c)] (overlapping)
$H(P,\tilde F)\ne H(Q,\tilde F)$ but $H(P,\tilde F)\cap H(Q,\tilde F)\ne\emptyset$ and hence $F=\tilde F|_\Gamma$ is nonhyperbolic, see \cite{DiaSan:04,DiaRoc:07}.
\end{itemize}
\item[(2)] The homoclinic classes $H(P,\tilde F)$ and $H(Q,\tilde F)$ are disjoint and hyperbolic, but there are other homoclinic classes (which are different as sets),
see  \cite[Theorem (2)(i)]{DiaRoc:02}. In this case, there are periodic orbits (say expanding) $\cO(R_1)$ and
\[
	\cO(R_2)\cap  (H(P,\tilde F)\cup H(Q,\tilde F))=\emptyset
\]
such that
\[
	\pi (\cO(R_1))
	= \pi (\cO(R_2))
	= \cO (\omega)
	\in \Sigma
	= \pi(\Gamma).
\]
Hence, the ergodic  measure supported on the periodic orbit $\cO(\omega)$ has (at least) two lifts to different ergodic measures in $\cM_{\rm{erg}, >0}$ (the ones supported on $\cO(R_1)$ and $\cO(R_2)$), thus failing Theorem~\ref{teo:1}. 
There is a similar construction replacing periodic orbits by nontrivial basic sets.
\end{itemize}

Note that, the scenarios in case (1) are compatible with our concavity assumptions (see for instance \cite{DiaEstRoc:16}) while case (2) is not (see Theorem~\ref{teo:homoclinicclasses}). 
\end{remark}

\appendix
%----------------------------------------------------------------------------------------------
\section{Wasserstein distance}\label{App:A}
%----------------------------------------------------------------------------------------------

Let us remind that for two probability measures $\mu_1, \mu_2$ supported on a compact metric space $M$ we define their \emph{couplings} as measures on $M\times M$ with marginals $\mu_1$ on the first coordinate and $\mu_2$ on the second. Denoting by $\Gamma(\mu_1,\mu_2)$ the space of all couplings of $\mu_1$ and $\mu_2$, we can define a metric on $\cM(M)$ (the \emph{Wasserstein distance}) by
\begin{equation}\label{def:W1}
	W_1(\mu_1, \mu_2) 
	\eqdef \inf_{\gamma\in\Gamma(\mu_1,\mu_2)} \int_{M\times M} \,d(x,y) \,d\gamma(x,y).
\end{equation}
As a special case of the duality theorem of Kantorovich and Rubinstein (see \cite{KanRub:58}), one can give an equivalent definition as follows
\begin{equation}\label{def:W2}
	W_1(\mu_1,\mu_2) 
	= \sup \Big\{\int_M f(x) \,d\mu_1(x) - \int_M f(x) \,d\mu_2(x)\colon \Lip(f) \leq 1\Big\}.
\end{equation}
Here $\Lip(f)$ denotes the Lipschitz constant, and the supremum is taken over Lipschitz functions only. It is well known that the Wasserstein distance is a metric on the space $\cP(M)$ of probability measures supported on $M$ and that it induces the weak$\ast$ topology on $\cP(M)$.

We will use the following lemma.

\begin{lemma}\label{lem:Wassx}
	Assume that $d$ is a metric on $M\subset X \times\bR$ satisfying $d((x_1, y_1),(x_2, y_2)) \geq \lvert x_2-x_1\rvert$, with equality if $y_1=y_2$. Assume also that $\mu_1$ and $\mu_2$ have a special coupling $\gamma\in \Gamma(\mu_1, \mu_2)$ such that $\gamma(\{((x_1, y_1), (x_2, y_2))\colon y_1=y_2, x_2 \geq x_1\})=1$. Then
\[
	W_1(\mu_1, \mu_2) 
	= \int_M x \,d\mu_2(x,y) - \int_M x \,d\mu_1(x,y).
\]
\end{lemma}

\begin{proof}
We have
\[
	\int_{M\times M} d((x_1,y_1),(x_2,y_2)) \,d\gamma((x_1,y_1),(x_2,y_2)) 
	= \int_{M} x \,d\mu_2(x,y) - \int_M x \,d\mu_1(x,y).
\]
By definition \eqref{def:W1} of the Wasserstein distance the left hand side of the above formula is an upper bound for $W_1(\mu_1, \mu_2)$. By definition \eqref{def:W2} of the Wasserstein distance the right hand side of the above formula is a lower bound for $W_1(\mu_1, \mu_2)$.
\end{proof}

\bibliographystyle{plain}

\end{document}